\newcommand{\dataversione}{March 25, 2015}

\documentclass[11pt,reqno,a4paper,draft]{amsart}

\usepackage{mathrsfs}
\usepackage{amssymb}
\usepackage[notref,notcite,final]{showkeys}
\usepackage{hyperref}\hypersetup{colorlinks=true, citecolor=blue}

\numberwithin{equation}{section}

\newtheoremstyle{mytheorem}
{}
{}
{\it}
{\parindent}
{\bf}
{.}
{ }
{\thmnumber{#2.~}\thmname{#1}\thmnote{~\rm#3}}

\newtheoremstyle{myremark}
{}
{}
{\rm}
{\parindent}
{\bf}
{.}
{ }
{\thmnumber{#2.~}\thmname{#1}\thmnote{~\rm#3}}

\newtheoremstyle{myparagraph}
{}
{}
{\rm}
{\parindent}
{\bf}
{.}
{ }
{\thmnumber{#2.~}\thmname{#1}\thmnote{#3}}

\theoremstyle{mytheorem}

\newtheorem{theorem}[subsection]{Theorem}
\newtheorem{lemma}[subsection]{Lemma}
\newtheorem{corollary}[subsection]{Corollary}
\newtheorem{proposition}[subsection]{Proposition}

\theoremstyle{myremark}
\newtheorem{remark}[subsection]{Remark}

\theoremstyle{myparagraph}
\newtheorem{parag}[subsection]{}
\newtheorem*{parag*}{}

\makeatletter
\def\@secnumfont{\sc}
\def\section{\@startsection{section}{1}%
\z@{1.5\linespacing\@plus .2\linespacing}{.7\linespacing}%
{\normalfont\sc\centering}}
\makeatother

\makeatletter
\def\ps@headings{\ps@empty 
 \def\@evenhead{%
  \setTrue{runhead}%
  \normalfont\footnotesize 
  \rlap{\thepage}\hfil 
  \def\thanks{\protect\thanks@warning}%
  \leftmark{}{}\hfil}%
 \def\@oddhead{%
  \setTrue{runhead}%
  \normalfont\footnotesize\hfil 
  \def\thanks{\protect\thanks@warning}%
  \rightmark{}{}\hfil \llap{\thepage}}%
\let\@mkboth\markboth}
\makeatother
\makeatletter
\renewenvironment{proof}[1][\proofname]{\par 
  \pushQED{\qed}%
  \normalfont \topsep6\p@\@plus6\p@\relax 
  \trivlist 
  \itemindent\normalparindent 
  \item[\hskip\labelsep 
    \bfseries 
    #1\@addpunct{.}]\ignorespaces 
}{%
  \popQED\endtrivlist\@endpefalse 
} 
\providecommand{\proofname}{Proof}
\makeatother
\newcommand{\Mass}{\mathbb{M}}

\newcommand{\R}{\mathbb{R}}
\newcommand{\D}{\mathscr{D}}
\newcommand{\F}{\mathscr{F}}
\newcommand{\G}{\mathscr{G}}
\newcommand{\Haus}{\mathscr{H}}
\newcommand{\Leb}{\mathscr{L}}
\newcommand{\K}{\mathscr{K}}
\newcommand{\M}{\mathscr{M}}

\newcommand{\X}{\mathscr{X}}

\newcommand{\bfi}{\mathbf{i}}
\newcommand{\bfj}{\mathbf{j}}

\newcommand{\Lip}{\mathrm{Lip}}
\newcommand{\Span}{\mathrm{span}}
\newcommand{\Tan}{\mathrm{Tan}}
\newcommand{\Int}{\mathrm{Int}}
\newcommand{\dist}{\mathrm{dist}}

\newcommand{\Gr}{\mathrm{Gr}}
\newcommand{\dgr}{d_{\mathrm{gr}}}
\newcommand{\wrt}{w.r.t.\ }
\newcommand{\bd}{\partial}
\newcommand{\eps}{\varepsilon}
\newcommand{\dV}{d_V\kern-1pt}

\newcommand{\clos}[1]{\smash{\overline{#1}}}
\newcommand{\barB}{{\clos{B}}}

\newcommand{\textfrac}[2]{{\textstyle\frac{#1}{#2}}}

\newcommand{\scal}[2]{\langle #1\, ; \, #2\rangle}
\newcommand{\bigscal}[2]{\big\langle #1 \, ; \, #2\big\rangle}

\DeclareMathOperator{\largewedge}{\mbox{\large$\wedge$}}
\DeclareMathOperator{\largewedgef}{\mbox{\small$\wedge$}}

\newcommand{\trait}[3]{\vrule width #1ex height #2ex depth #3ex}

\newcommand{\trace}{\mathchoice%
  {\mathbin{\trait{.12}{1.2}{.03}\trait{.8}{0.09}{0.03}}}
  {\mathbin{\trait{.12}{1.2}{.03}\trait{.8}{0.09}{0.03}}}
  {\mathbin{\hskip.15ex\trait{.09}{.84}{0.02}\trait{.56}{.07}{.02}}\hskip.15ex}
  {\mathbin{\trait{.07}{.6}{.01}\trait{.4}{.06}{.01}}}}

\newcommand{\footnoteb}[1]{\footnote{~#1}}
\makeatletter
\@addtoreset{footnote}{section}
\makeatother

\newenvironment{itemizeb}
{\begin{itemize}\itemsep=2pt}{\end{itemize}}

\begin{document}

	%
\pagestyle{empty}
\pagestyle{myheadings}
\markboth%
{\underline{\centerline{\hfill\footnotesize%
\textsc{Giovanni Alberti and Andrea Marchese}%
\vphantom{,}\hfill}}}%
{\underline{\centerline{\hfill\footnotesize%
\textsc{Differentiability of Lipschitz functions}%
\vphantom{,}\hfill}}}

	%
\thispagestyle{empty}

~\vskip -1.1 cm

	%
{\footnotesize\noindent
[version:~\dataversione]%
\hfill
}

\vspace{1.7 cm}

	%
{\large\bf\centering
On the differentiability of Lipschitz functions 
\\
with respect to measures in the Euclidean space
\\
}

\vspace{.6 cm}

	%
\centerline{\sc Giovanni Alberti and Andrea Marchese}

\vspace{.8 cm}

{\rightskip 1 cm
\leftskip 1 cm
\parindent 0 pt
\footnotesize

	%
{\sc Abstract.}
Rademacher theorem states that every Lipschitz function on 
the Euclidean space is differentiable almost everywhere, 
where ``almost everywhere'' refers to the Lebesgue measure.
In this paper we prove a differentiability result of  
similar type, where the Lebesgue measure is replaced 
by an arbitrary finite measure $\mu$.
In particular we show that the differentiability properties
of Lipschitz functions at $\mu$-almost every point are related 
to the decompositions of $\mu$ in terms of rectifiable 
one-dimensional measures. 
As a consequence we obtain a differentiability result for
Lipschitz functions with respect to 
(measures associated to) $k$-dimensional normal currents, 
which we use to extend certain formulas involving normal 
currents and maps of class $C^1$ to Lipschitz maps.

\par
}

%
%

\section{Introduction}
\label{s1}
The study of the differentiability
properties of Lipschitz functions has a long
story, and many facets.
In recent years much attention has been devoted
to the differentiability of Lipschitz functions
on infinite dimensional Banach spaces 
(see the monograph by J.~Lindenstrauss, D.~Preiss 
and J.~Ti{\v s}er \cite{LPT})
and on metric spaces (we just mention here the 
works by J.~Cheeger~\cite{Cheeger}, 
S.~Keith~\cite{Keith} and D.~Bate~\cite{Bate}), 
but at about the same time it became 
clear that even Lipschitz functions on 
$\R^n$ are not completely understood, 
and that Rademacher theorem, which states
that every Lipschitz function on 
$\R^n$ is differentiable almost everywhere,%
\footnoteb{We use the expressions ``almost everywhere'' 
and ``null set'' without further 
specification to mean ``almost everywhere'' 
and ``null set'' with respect to the Lebesgue measure.
The same for ``absolutely
continuous measure'' and ``singular measure''.}
is not the end of story.

To this regard, the first fundamental contribution 
is arguably the paper \cite{Preiss}, where Preiss
proved, among other things, that there exist 
null sets $E$ in $\R^2$ such that every 
Lipschitz function on $\R^2$ is differentiable at
some point of $E$.%
\footnoteb{The construction of such sets 
has been variously improved in recent years, 
cf.~\cite{DoMa}.}
This result showed that Rademacher 
theorem is not sharp, in the sense that 
while the set of non-differentiability points
of a Lipschitz function is always contained
in a null set, the opposite inclusion does
not always hold.

Note that this construction 
does not give a null sets $E$ such that
every Lipschitz map%
\,\footnoteb
{As usual, we reserve the word ``function'' for real-valued maps.}
is differentiable at some point of $E$,
and indeed it was later proved by G.~Alberti, 
M.~Cs\"ornyei and D.~Preiss 
that every null set in $\R^2$ 
is contained in the non-differentiability
set of some Lipschitz map $f:\R^2 \to \R^2$ 
(this result was announced in \cite{ACP1}, \cite{ACP2} 
and appears in \cite{ACP3}
).
The situation in higher dimension has 
been clarified only very recently: 
first M.~Cs\"ornyei and P.W.~Jones announced 
(see \cite{CJ}) that every null 
set in $\R^n$ is contained in the non-differentiability 
set of a Lipschitz map $f:\R^n\to\R^n$,
and then D.~Preiss and G.~Speight~\cite{PS}
proved that there exist null sets
$E$ in $\R^n$ such that every Lipschitz
map $f:\R^n\to\R^m$ with $m<n$ is 
differentiable at some point of $E$.

\medskip
In this paper we approach the differentiability
of Lipschitz functions from a slightly different 
point of view. Consider again the statement of
Rademacher theorem: the ``almost everywhere'' 
there refers to the Lebesgue measure,  
and clearly the statement remains true if we 
replace the Lebesgue measure with a
measure $\mu$ which is absolutely continuous,
but of course it fails if $\mu$ is an arbitrary 
singular measure.

However in many cases it is clear how to modify the
statement to make it true. For example, if $S$ is a
$k$-dimensional surface of class $C^1$ 
contained in $\R^n$ and $\mu$ is the $k$-dimensional 
volume measure on $S$, then every Lipschitz function 
on $\R^n$ is differentiable at $\mu$-a.e.~$x\in S$ 
in all directions in the tangent space $\Tan(S,x)$. 
Furthermore this statement is optimal in the sense that
there are Lipschitz functions $f$ on $\R^n$ which, for 
every $x\in S$, are not differentiable at $x$ in any 
direction which is not in $\Tan(S,x)$ (the obvious example 
is the distance function $f(x):=\dist(x,S)$).

We aim to prove a statement of similar nature
for an arbitrary finite measure $\mu$ on $\R^n$.
More precisely, we want to identify
for $\mu$-a.e.~$x$ the largest
set of directions $V(\mu,x)$ such that every 
Lipschitz function on $\R^n$ is 
differentiable at $\mu$-a.e.~$x$ 
in every direction in $V(\mu,x)$.

\medskip
We begin with a simple observation:
let $\mu$ be a measure on $\R^n$ that 
can be decomposed as%
\,\footnoteb{The meaning of formula \eqref{e-1.1}
is that $\smash{\mu(E) = \int_I \mu_t(E) \, dt}$
for every Borel set $E$; the precise 
definition of integral of a measure-valued
map is given in \S\ref{s-measint}.
}
\begin{equation}
\label{e-1.1}
\mu = \int_I \mu_t \, dt
\end{equation}
where $I$ is the interval $[0,1]$ endowed with
the Lebesgue measure $dt$, and each 
$\mu_t$ is the length measure on some rectifiable 
curve $E_t$, and assume in addition that there 
exists a vectorfield $\tau$ on $\R^n$ such that 
for a.e.~$t$ and $\mu_t$-a.e.~$x\in E_t$ 
the vector $\tau(x)$ is tangent to $E_t$ at $x$.
Then every Lipschitz function $f$ on $\R^n$
is differentiable at $x$ in the direction 
$\tau(x)$ for $\mu$-a.e.~$x$.

Indeed by applying Rademacher theorem to the 
Lipschitz function $f\circ\gamma_t$, where
$\gamma_t$ is a parametrization of $E_t$ by
arc-length, we obtain that 
$f$ is differentiable at the point $\gamma(s)$ 
in the direction $\dot\gamma(s)$ for a.e.~$s$, 
which means that $f$ is differentiable at $x$ in 
the direction $\tau(x)$ for $\mu_t$-a.e.~$x$
and a.e.~$t$, and by formula \eqref{e-1.1}
``for $\mu_t$-a.e.~$x$ and a.e.~$t$'' is equivalent
to ``for $\mu$-a.e.~$x$''.

\medskip
This observation suggests that the set of 
directions $V(\mu,x)$ we are looking for should 
be related to the set of all decompositions 
of $\mu$, or of parts of $\mu$, of the type
considered in the previous paragraph.
Accordingly, we give the following ``provisional'' 
definition: consider all possible families of 
measures $\{\mu_t\}$ such that the measure
$\int_I	\mu_t \, dt$ is absolutely continuous 
\wrt $\mu$, and each $\mu_t$ is the restriction
of the length measure to a subset $E_t$ of a 
rectifiable curve, and for every $x\in E_t$ let
$\Tan(E_t,x)$ be the tangent line to this curve
at $x$ (if it exists); 
let then $V(\mu,x)$ be the smallest linear subspace
of $\R^n$ such that for every family $\{\mu_t\}$
as above there holds $\Tan(E_t,x) \subset V(\mu,x)$ 
for a.e.~$t$.
We call the map $x\mapsto V(\mu,x)$ the 
\emph{decomposability bundle} of $\mu$.

Even though this definition presents some flaws 
at close scrutiny, it should be sufficient to 
understand the following theorem, which is the main 
result of this paper; 
the ``final'' version of this definition is slightly 
different and more involved, and has been postponed 
to the next section (see \S\ref{s-decomp}).

\begin{theorem}\label{s-main}
Let $\mu$ be a finite measure on $\R^n$, and
let $V(\mu,\cdot)$ be the decomposability bundle
of $\mu$ (see \S\ref{s-decomp}). 
Then the following statements hold:
\begin{itemizeb}
\item[(i)]
Every Lipschitz function $f$ on $\R^n$
is differentiable at $\mu$-a.e.~$x$ with respect 
to the linear subspace $V(\mu,x)$, that is, there exists
a linear function from $V(\mu,x)$ to $\R$, denoted by
$\dV f(x)$, such that
\[
f(x+h) = f(x) + \dV f(x) \, h  + o(|h|)
\quad\text{for $h\in V(\mu,x)$.}%
\footnoteb{We use the Landau notation 
$o(|h|)$ with the understanding that $h$ tends to~$0$.}
\]
\item[(ii)]
The previous statement is optimal 
in the sense that there exists a Lipschitz
function $f$ on $\R^n$ such that for $\mu$-a.e.~$x$
and every $v\notin V(\mu,x)$ the derivative
of $f$ at $x$ in the direction $v$ does not exist.
\end{itemizeb}
\end{theorem}

\begin{remark}
The differentiability part of this theorem, 
namely statement~(i), applies also
to Lipschitz maps $f:\R^n\to\R^m$
(because it applies to each component of $f$).
Note that for the non-differentiability part 
we are able to give a real-valued example.
In other words, the results on the 
differentiability of Lipschitz 
maps with respect to measures---at least
those presented in this paper---are 
not sensitive to the dimension of the 
codomain of the maps, unlike the results 
concerning the differentiability at every point
of a given set (see the discussion above).%
\footnoteb{The difference between measures and sets 
is partly explained by the fact that given a
Lipschitz map $f:\R^n\to\R^m$ which is 
non-differentiable $\mu$-a.e., it is relatively
easy to obtain a function $g:\R^n\to\R$ which is
not differentiable  $\mu$-a.e. For instance, 
if $f_i$ are the components of $f$,
the linear combination $g:=\sum \alpha_i f_i$
would do the job for (Lebesgue-) almost every choice of 
the vector of coefficients $(\alpha_1,\dots,\alpha_m)$.
This statement, which incidentally is not difficult 
to prove, has no counterpart for sets.}
\end{remark}

\begin{parag}[Relations with results in the literature]
Even though it was never written down explicitly, 
the idea that the differentiability properties of Lipschitz
functions \wrt a general measure $\mu$ are encoded
in the decompositions of $\mu$ in terms of rectifiable
measures was somewhat in the air (for example, it is 
clearly assumed as a starting point in \cite{Bate}, 
where this idea is extended to the context of 
metric spaces to give a characterization of 
Lipschitz differentiability spaces).  

Moreover the proof of Theorem~\ref{s-main}(ii), 
namely the construction of a Lipschitz function 
which is not differentiable in any direction 
which is not in $V(\mu,x)$, is a simplified version 
of a construction given in \cite{ACP3}.
Note that the original construction 
gives Lipschitz functions (actually maps) which 
are non-differentiable at any point of a given 
set, but if we need a function which is
non-differentiable $\mu$-a.e.\ we 
are allowed to discard $\mu$-null sets, 
and this possibility makes room for significant 
simplifications.

Finally, let us mention that the notion
of decomposability bundle is related to a 
notion of tangent space to measures
given in \cite{Alb} (see Remark~\ref{s-6.5}(v)).
\end{parag}

\begin{parag}[Computation of the decomposability bundle]
In certain cases the decomposability bundle $V(\mu,x)$ 
can be computed using Proposition~\ref{s-basic}. 
We just recall here that if $\mu$ is absolutely continuous \wrt 
the Lebesgue measure then $V(\mu,x)=\R^n$ for 
$\mu$-a.e.~$x$, and if $\mu$ is absolutely continuous \wrt 
the restriction of the Hausdorff measure $\Haus^k$ to
a $k$-dimensional surface $S$ of class $C^1$
(or even a $k$-rectifiable set $S$, cf.~\S\ref{s-rectif})
then $V(\mu,x)=\Tan(S,x)$ for $\mu$-a.e.~$x$
(Proposition~\ref{s-basic}(iii)). 
On the other hand, if $\mu$ is the canonical measure 
associated to well-known examples of self-similar 
fractals such as the snowflake curve and the 
Sierpi\'nski carpet, then $V(\mu,x)=\{0\}$ 
for $\mu$-a.e.~$x$ (see Remark~\ref{s-examples}).
\end{parag}

\begin{parag}[Application to the theory of currents]
\label{s-introcurr}
In Section~\ref{s3} we study the 
decomposability bundle of measures related 
to normal currents.
More precisely, given a measure $\mu$
and a normal $k$-current $T$, we denote by $\tau$ 
the Radon-Nikod\'ym density of $T$ \wrt $\mu$
(see \S\ref{s-repcurr}),
and show that the linear subspace of $\R^n$ spanned by 
the $k$-vector $\tau(x)$ is contained in $V(\mu,x)$ 
for $\mu$-a.e.~$x$
(see \S\ref{s-span} and Theorem~\ref{s-curr2}). 
We then use this result to give explicit formulas
for the boundary of the interior product of a normal 
$k$-current and a Lipschitz $h$-form 
(Proposition~\ref{s-boundary2})
and for the push-forward of a normal
$k$-current according to a Lipschitz map 
(Proposition~\ref{s-pushf2}).

In section~\ref{s6} we give a characterization of 
the decomposability bundle of a measure $\mu$ in terms
of $1$-dimensional normal currents (Theorem~\ref{s-6.3}); 
building on this result we obtain that a vectorfield $\tau$ 
on $\R^n$ can be written as the Radon-Nikod\'ym density 
of a $1$-dimensional normal current $T$ \wrt $\mu$ 
(cf.~\S\ref{s-repcurr}) if and only if $\tau(x)$ belongs 
to $V(\mu,x)$ for $\mu$-a.e.~$x$ (Corollary~\ref{s-6.4}).
\end{parag}

\begin{parag}[On the validity of Rademacher theorem]
%
It is natural to ask for which measures $\mu$ 
on $\R^n$ Rademacher theorem holds in the usual form, 
that is, every Lipschitz function (or map) on $\R^n$ 
is differentiable $\mu$-a.e.
Clearly the class of such measures contains all 
absolutely continuous measures, but does it contains 
any singular measure?

The answer turns out to be negative in every dimension
$n$, because a singular measure $\mu$ is supported on 
a null set $E$, and for every null 
set $E$ contained in $\R^n$ there exists a Lipschitz 
map $f:\R^n\to\R^n$ which is non-differentiable 
at every point of $E$. 
In dimension $n=1$ the existence of such map has 
been known for long time (see for instance \cite{Za});
a construction in general dimension is given 
\cite{ACP3} for a certain class of sets $E$, 
and it is proved that in dimension $n=2$ such class 
agrees with the class of all null sets;
the last result was recently proved for every $n\ge 2$ 
by Cs\"ornyei and Jones, as announced in \cite{CJ}.
\end{parag}

\begin{remark}
\label{s-remrade}
(i)~By Theorem~\ref{s-main},
Rademacher theorem holds for a measure $\mu$ if 
and only if $V(\mu,x) = \R^n$ for $\mu$-a.e.~$x$.
This fact allows us to rephrase the conclusions 
of the previous subsection as follows:
if $\mu$ is a singular measure, then 
$V(\mu,x) \ne \R^n$ for $\mu$-a.e.~$x$.%
\footnoteb{In view of Proposition~\ref{s-basic}(i) 
we can say (slightly) more: for every measure $\mu$ on 
$\R^n$ there holds $V(\mu,x) \ne \R^n$ for 
$\mu^s$-a.e.~$x$, where $\mu^s$ is the singular part 
of the measure $\mu$ \wrt the Lebesgue measure.}

(ii)~For $n=1$ it is easy to show directly 
that $V(\mu,x) = \{0\}$ for $\mu$-a.e.~$x$
when $\mu$ is singular:
indeed $\mu$ is supported on a null 
set $E$, null sets in $\R$ are
purely unrectifiable (\S\ref{s-rectif}), 
and in every dimension the decomposability bundle 
of a measure supported on a purely unrectifiable 
set is trivial (Proposition~\ref{s-basic}(iv)).

(iii)~For $n=2$, the fact that $V(\mu,x) \ne \R^2$ for 
$\mu$-a.e.~$x$ when $\mu$ is singular
follows also from a result proved in \cite{Alb} 
(see Remark~\ref{s-6.5}(vi) for more details).
\end{remark}

\begin{parag}[Higher dimensional decompositions]
\label{s-hddecomp}
For $k=1,\dots,n$ let $\F_k(\R^n)$ be the
class of all measure $\mu$ on $\R^n$ which are 
absolutely continuous \wrt a measure of the form
$\smash{\int_I \mu_t \, dt}$ where each
$\mu_t$ is the volume measure on a $k$-dimensional
surface $E_t$ of class $C^1$.%
\footnoteb{Or, equivalently, each $\mu_t$
is the restriction of the $k$-dimensional 
Hausdorff measure $\Haus^k$ to a $k$-rectifiable 
set $E_t$ (see \S\ref{s-rectif}).}
By Proposition~\ref{s-basic}(vi), for every $\mu$ in 
this class the decomposability bundle $V(\mu,x)$ has 
dimension at least $k$ at $\mu$-a.e.~$x$ 
and it is natural to ask whether the converse
is true, namely that $\dim(V(\mu,x))\ge k$
for $\mu$-a.e.~$x$ implies that 
$\mu$ belongs to $\F_k(\R^n)$.

The answer is positive for $k=1$ and $k=n$:
the case $k=1$ is trivial, while the case $k=n$ is 
equivalent to the statement mentioned in 
Remark~\ref{s-remrade}(i), namely that if 
$\dim(V(\mu,x))=n$ for $\mu$-a.e.~$x$ then 
the measure $\mu$ is absolutely continuous. 
Recently, A.~M\'ath\'e proved  in \cite{Mathe} that 
the answer is negative in all other cases.\end{parag}

\begin{parag}[Differentiability of Sobolev functions]
\label{s-introsob}
Since the continuous representatives of functions in 
the Sobolev space $W^{1,p}(\R^n)$ with 
$p>n$ are differentiable almost everywhere, 
it is natural to ask what differentiability result 
we have when the Lebesgue measure is replaced by 
a singular measure $\mu$. 
In \cite{ACP3} 
an example is given of a continuous function 
in $W^{1,p}(\R^n)$ which is not differentiable 
in any direction at $\mu$-a.e.~point; 
it seems therefore that Theorem~\ref{s-main} 
admits no significant extension to 
(first order) Sobolev space.
\end{parag}

\medskip
The rest of this paper is organized as follows: 
in Section~\ref{s2} we give the precise definition of 
decomposability bundle and a few of its basic properties,
while Sections~\ref{s4} and \ref{s5} contain the proof 
of Theorem~\ref{s-main}. 

In Section~\ref{s3} we study the decomposability 
bundle of measures associated to normal currents,
and describe a few applications related to the theory 
of normal currents, 
while in Section~\ref{s6} we give a characterization 
of the decomposability bundle of a measure in terms of  
$1$-dimensional normal currents.

Note that sections~\ref{s3} and \ref{s6} are essentially 
independent of the rest of the paper (and of each other), 
and can be skipped by readers who are not specifically 
interested in the theory of currents.

In order to make the structure of the proofs
more transparent, we have postponed to the
appendix (Section~\ref{s7}) the proofs of several 
technical lemmas used in the rest of the paper.

%
%

\begin{parag*}[Acknowledgements]
This work has been partially supported by the Italian Ministry 
of Education, University and Research (MIUR) through the 2008 PRIN 
Grant ``Trasporto ottimo di massa, disuguaglianze geometriche 
e funzionali e applicazioni'', and by the European Research Council
(ERC) through the Advanced Grant ``Local Structure of Sets, 
Measures and Currents''.

We are particularly indebted to David Preiss for suggesting 
Proposition~\ref{s-4.3}, which lead to a substantial simplification
of the proof of Theorem~\ref{s-main}(i). We also thank
Ulrich Menne and Emanuele Spadaro for asking crucial questions.
\end{parag*}

%
%

\section{Decomposability bundle}
\label{s2}
We begin this section by recalling some 
definitions and notation used through the entire paper 
(more specific definitions and notation will be introduced 
when needed).
We then give the definition of decomposability bundle
$V(\mu,\cdot)$ of a measure $\mu$ (see \S\ref{s-decomp})
and prove a few basic properties 
(Proposition~\ref{s-basic}). 

\begin{parag}[General notation]
\label{s-not}
For the rest of this paper,
sets and functions are tacitly assumed 
to be Borel measurable, and measures are always 
defined on the appropriate Borel $\sigma$-algebra. 
Unless we write otherwise, measures are 
positive and finite.
We say that a measure $\mu$ on a space $X$ 
is supported on the Borel set $E$ if 
$|\mu|(X\setminus E)=0$.%
\footnoteb{Note that $E$ does not need to be closed, 
and hence it may not contain the support of $\mu$.} 

We say that a map $f:\R^n\to\R^m$ is 
differentiable at the point $x\in\R^n$ \wrt
a linear subspace $V$ of $\R^n$ if 
there exists a linear map $L:V\to\R^m$ such 
that the following first-order Taylor 
expansion holds
\[
f(x+h) = f(x) + L h + o(|h|)
\quad\hbox{for all $h\in V$.}
\]
The linear map $L$ is unique, 
it is called the derivative of $f$ at
$x$ \wrt $V$, and is denoted by $\dV f(x)$.%
\footnoteb{If $V=\R^n$ then $\dV f(x)$ is the usual 
derivative, and is denoted by $df(x)$. Note that 
if $V=\{0\}$ the notion of differentiability 
makes sense, but is essentially void (every map
is differentiable \wrt $V$ at every point).}

\smallskip
We add below a list of frequently used notations
(for the notations related to multilinear algebra 
and currents see \S\ref{s-not2}):
\begin{itemizeb}\leftskip 0.8 cm\labelsep=.3 cm
\item[$B(r)$]
closed ball with center $0$ and radius $r$ in $\R^n$;
\item[$B(x,r)$]
closed ball with center $x$ and radius $r$ in $\R^n$;
\item[$\dist(x,E)$]
distance between the point $x$ and the set $E$;
\item[$v\cdot w$]
scalar product of $v,w \in \R^n$;
\item[$C(e,\alpha)$]
convex closed cone in $\R^n$
with axis $e$ and angle $\alpha$ (see \S\ref{s-cones});
\item[$1_E$] 
characteristic function of a set $E$, defined on 
any ambient space and taking values $0$ and $1$;
\item[$\Gr(\R^n)$] 
set of all linear subspaces of $\R^n$, that is, the union 
of the Grassmannians $\Gr(\R^n,k)$ with $k=0,\dots,n$.
\item[$\dgr(V,V')$]  
distance between $V,V'\in\Gr(\R^n)$, 
defined as the maximum of $\delta(V,V')$ and $\delta(V',V)$,
where $\delta(V,V')$ is the smallest number $d$ such that 
for every $v\in V$ there exists $v'\in V'$ with
$|v-v'|\le d|v|$;%
\footnoteb{Note that $\dgr(V,V')$ agrees with the Hausdorff
distance between the closed sets $V\cap B(1)$ and $V'\cap B(1)$; 
this shows that $\dgr$ is indeed a metric on $\Gr(\R^n)$.}
\item[$\scal{L}{v}$]
(also written $Lv$)
action of a linear map $L$ on a vector $v$;
linear maps are always endowed with the operator 
norm, denoted by $|\cdot|$;
\item[$D_vf(x)$] 
derivative of a map $f:\R^n\to\R^m$ 
in the direction $v$ at a point $x$;
\item[$df(x)$] 
derivative of a map $f:\R^n\to\R^m$ at a point $x$, 
viewed as a linear map from from $\R^n$ to $\R^m$;
\item[$\dV f(x)$] 
derivative of a map $f:\R^n\to\R^m$ at a point $x$
\wrt a subspace $V$ (see at the beginning of this 
subsection);
\item[$\Tan(S,x)$]
tangent space to $S$ at a point $x$, where $S$ is a surface 
(submanifold) of class $C^1$ in $\R^n$ or a rectifiable set
(see \S\ref{s-rectif});
\item[$\Lip(f)$] 
Lipschitz constant of a map $f$ (between two 
metric spaces);
%
%
%
%
\item[$\Leb^n$]
Lebesgue measure on $\R^n$;
\item[$\Haus^d$]
$d$-dimensional Hausdorff measure
(on any metric space $X$);
\item[$L^p$]
stands for $L^p(\R^n,\Leb^n)$; 
for the $L^p$ space on a different measured space 
$(X,\mu)$ we use the notation $L^p(\mu)$;
\item[$\|\cdot \|_p$]
norm in $L^p=L^p(\R^n,\Leb^n)$; 
we use $\|\cdot\|_\infty$ also to denote the supremum norm
of continuous functions;
\item[$1_E \, \mu $]
restriction of a measure $\mu$ to a set $E$;
\item[$f_\#\,\mu$]
push-forward of a measure $\mu$ on a space $X$
according to a map $f:X\to X'$, that is,
$[f_\#\,\mu](E):=\mu(f^{-1}(E))$
for every Borel set $E$ contained in $X'$;
\item[$|\mu|$] 
total variation measure associated to the real- or vector-valued 
measure $\mu$; thus $\mu$ can be written as $\mu=\rho\,|\mu|$ 
where the (real- or vector-valued) density $\rho$ satisfies
$|\rho|=1$ $\mu$-a.e.
\item[$\Mass(\mu)$]
$:=|\mu|(X)$, mass of the measure $\mu$;
\item[$\lambda\ll\mu$]
means that the measure $\lambda$ is absolutely continuous \wrt 
$\mu$, hence $\lambda = \rho\mu$ for 
a suitable density $\rho$;
%
%
\item[$\int_I \mu_t \, dt$]
integral of the measures $\mu_t$ with $t\in I$ 
with respect to a measure $dt$ 
(see \S\ref{s-measint}).
\end{itemizeb}
\end{parag}

\begin{parag}[Essential span of a family of vectorfields]
\label{s-essspan}
Given a measure $\mu$ on $\R^n$ and 
a family $\G$ of Borel maps from $\R^n$
to $\Gr(\R^n)$, we say that $V$ is a minimal
element of $\G$ if for every $V'\in\G$ there holds
$V(x) \subset V'(x)$ for $\mu$-a.e.~$x$,
and we say that $\G$ is closed 
under countable intersection if 
given a countable subfamily $\{V_i\} \subset\G$
the map $V$ defined by $V(x):=\cap_i V_i(x)$ 
for every $x\in\R^n$ belongs to $\G$.

Let now $\F$ be a family of Borel vectorfields 
on $\R^n$, and let $\G$ be the class of all Borel maps
$V:\R^n\to \Gr(\R^n)$ such that for every 
$\tau\in \F$ there holds
\[
\tau(x) \subset V(x)
\quad\text{for $\mu$-a.e.~$x$.}
\]
Since $\G$ is closed under 
countable intersection, 
by Lemma~\ref{e-minmap} below it admits 
a unique minimal element, which we call 
\emph{$\mu$-essential span} of $\F$.
\end{parag}

\begin{lemma}
\label{e-minmap}
Let $\G$ be a family of Borel maps from 
$\R^n$ to $\Gr(\R^n)$ which is closed under
countable intersection. Then 
$\G$ admits a minimal element $V$, which is unique
modulo equivalence $\mu$-a.e.%
\footnoteb{In other words, every other minimal element
$V'$ satisfies $V(x)=V'(x)$ for $\mu$-a.e.~$x$.}
\end{lemma} 

\begin{proof}
Uniqueness follows immediately
from minimality. To prove existence, set 
\[
\Phi(V):=\int_{\R^n} \dim(V(x)) \, d\mu(x)
\]
for every $V\in \G$, then take a sequence $\{V_j\}$ 
in $\G$ such that $\Phi(V_j)$ tends to the infimum 
of $\Phi$ over $\G$, and let $V$ be the intersection 
of all $V_j$; thus $V$ belongs to $\G$ and is a minimum 
of $\Phi$ over $\G$, and we claim that $V$ is a minimal 
element of $\G$: if not, there would exists $V'\in\G$ such that 
$V''(x):= V(x)\cap V'(x)$ is
strictly contained in $V(x)$ for all $x$ in some
set of positive measure, thus $V''$ belongs 
to $\G$ and $\Phi(V'') <\Phi(V)$.
\end{proof}

\begin{parag}[Rectifiable and unrectifiable sets]
\label{s-rectif}
Given $k=1,2,\dots$
we say that a set $E$ contained in $\R^n$ is $k$-rectifiable 
if it has finite $\Haus^k$ measure and it can be covered, 
except for a $\Haus^k$-null subset, by countably many 
images of Lipschitz maps from $\R^k$ to $\R^n$, or, equivalently, 
by countably many $k$-dimensional surfaces (submanifolds) 
of class $C^1$.%
\footnoteb{See for instance \cite{Mo}, \S3.10 and 
Proposition~3.11, or \cite{KP}, Definition~5.4.1 
and Lemma~5.4.2. Note that the terminology and 
even the definition vary (slightly) depending on 
the author.}

In particular it is possible to define for 
$\Haus^k$-a.e.~$x\in E$
an approximate tangent space $\Tan(E,x)$.%
\footnoteb{See for example \cite{KP}, 
Theorem~5.4.6, or \cite{Mo}, Proposition~3.12.}
The tangent bundle is actually characterized by 
the property that for every 
$k$-dimensional surface $S$ of class $C^1$ 
there holds 
\begin{equation}
\label{e-tangent}
\Tan(E,x)=\Tan(S,x)
\quad\text{for $\Haus^k$-a.e.~$x\in E\cap S$.}
\end{equation}

Finally we say that a set $E$ in $\R^n$ is purely 
unrectifiable (or more precisely $1$-purely 
unrectifiable) if $\Haus^1(E\cap S)=0$
for every $1$-rectifiable set $S$, or equivalently
for every curve $S$ of class $C^1$.
\end{parag}

\begin{parag}[Integration of measures]
\label{s-measint}
Let $I$ be a locally compact, separable metric space
endowed with a finite measure $dt$, 
and for every $t\in I$ let $\mu_t$ be a 
measure on $\R^n$, possibly real- or vector-valued,
such that:
\begin{itemizeb}
\item[(a)]
the function $t\mapsto \mu_t(E)$ 
is Borel for every Borel set $E$ in $\R^n$; 
\item[(b)]
the function $t\mapsto \Mass(\mu_t)$ is Borel
and the integral $\int_I \Mass(\mu_t) \, dt$ is finite.
\end{itemizeb}
Then we denote by $\int_I \mu_t\, dt$ the measure on 
$\R^n$ defined by
\[
{\textstyle \big[ \int_I \mu_t\, dt \big]}(E)
:= \int_I \mu_t(E) \, dt
\quad\text{for every Borel set $E$ in $\R^n$.}
\]
\end{parag}

\begin{remark}
\label{s-2.4}
(i)~Assumption~(a) above is equivalent 
to say that $t\mapsto \mu_t$ is a Borel map
from $I$ to the space of finite (real- or vector-valued)
measures on $\R^n$ endowed with the weak* topology.%
\footnoteb{As dual of the appropriate Banach space of 
continuous functions on $\R^n$.}
Note that assumption~(a) and the definition of mass 
imply that the function $t\mapsto \Mass(\mu_t)$ 
is Borel, thus the first part of assumption~(b)
is redundant.

(ii)~Given $I$ and $dt$ as in the previous subsection,
there always exists a Borel map $\phi:[0,1]\to I$ 
such that the push-forward according to $\phi$ of the
Lebesgue measure agrees up to a constant factor 
with the measure $dt$ (see for instance \cite{BCM}, Theorem~A.3); 
therefore, by composing the map $t\mapsto \mu_t$ with $\phi$, 
one can always assume that in the expression 
$\int_I \mu_t \, dt$, $I$ is the interval $[0,1]$
and $dt$ is (a multiple of) the Lebesgue measure.
\end{remark}

\begin{parag}[Decomposability bundle]
\label{s-decomp}
Let $I$ be the interval $[0,1]$ 
endowed with the Lebesgue measure $dt$.
Given a measure $\mu$ on $\R^n$ we denote by
$\F_\mu$ the class of all families 
$\{\mu_t: \, t\in I\}$ such that
\begin{itemizeb}
\item[(a)]
each $\mu_t$ is the restriction of $\Haus^1$ to a
$1$-rectifiable set $E_t$;
\item[(b)]
$\{\mu_t\}$ satisfies the assumptions
(a) and (b) in \S\ref{s-measint};
\item[(c)]
the measure $\int_I \mu_t \, dt$ is 
absolutely continuous \wrt $\mu$.
\end{itemizeb}
Then we denote by $\G_\mu$ the class of all Borel maps
$V:\R^n\to \Gr(\R^n)$ such that for every $\{\mu_t\}\in\F_\mu$ 
there holds
\begin{equation}\label{test}
\Tan(E_t,x) \subset V(x)
\quad\text{for $\mu_t$-a.e.~$x$ and a.e.~$t\in I$.}
 \end{equation}
Since $\G_\mu$ is closed under countable intersection, 
by Lemma~\ref{e-minmap} it admits a minimal element
which is unique modulo equivalence $\mu$-a.e.
We call this map \emph{decomposability bundle} 
of $\mu$, and denote it by $x\mapsto V(\mu,x)$.
\end{parag}

\begin{remark}\label{s-2.5}
(i)~In view of Remark~\ref{s-2.4}(ii), 
nothing would change in the definition of 
decomposability bundle if we let $I$ range 
among all locally compact, separable metric 
spaces, and $dt$ range among all finite measures
on $I$. We tacitly use this fact in 
the following. 

(ii)~This definition of the 
decomposability bundle differs
from the one given in the Introduction in 
two respects: firstly, the minimality property
that characterizes $V(\mu,\cdot)$ is now precisely 
stated, and secondly the sets $E_t$ are now
$1$-rectifiable sets, and not just subsets 
of a rectifiable curve. This modification 
does not affect the definition, 
but it is convenient for technical reasons.

%
%
%
\end{remark}

Propositions~\ref{s-basic} and \ref{generator} contain
some relevant properties of the decomposability bundle 
(see also Remark~\ref{s-remrade}).

\begin{proposition}
\label{s-basic}
Let $\mu$, $\mu'$ be measures on $\R^n$. 
Then the following statements hold:
\begin{itemizeb}
\item[(i)](strong locality principle)
if $\mu' \ll \mu$ then $V(\mu',x) = V(\mu,x)$ 
for $\mu'$-a.e.~$x$; more generally, 
if $1_E \, \mu' \ll\mu$ for some set $E$,
then $V(\mu',x) = V(\mu,x)$ for $\mu'$-a.e.~$x\in E$;
\item[(ii)]
if $\mu$ is supported on a $k$-dimensional surface $S$ 
of class $C^1$ then $V(\mu,x) \subset \Tan(S,x)$ 
for $\mu$-a.e.~$x$;
\item[(iii)]
if $\mu \ll 1_E \, \Haus^k$ where $E$ is a 
$k$-rectifiable set, then $V(\mu,x) = \Tan(E,x)$ 
for $\mu$-a.e.~$x$; 
in particular if $\mu \ll \Leb^n$ then 
$V(\mu,x) =\R^n$ for $\mu$-a.e.~$x$; 
\item[(iv)]
$V(\mu,x) = \{0\}$ for $\mu$-a.e.~$x$
if and only if $\mu$ is supported on a 
purely unrectifiable set $E$;
\end{itemizeb}
Moreover, given a family of measures 
$\{\nu_s: \, s\in I\}$ as in \S\ref{s-measint},
\begin{itemizeb}
\item[(v)]
if $\int_I \nu_s \, ds \ll \mu$ then 
$V(\nu_s,x) \subset V(\mu,x)$ for 
$\nu_s$-a.e.~$x$ and a.e.~$s$;
\item[(vi)]
if $\mu \ll \int_I \nu_s \, ds$ and 
$\nu_s$ is of the form $\nu_s =1_{E_t} \, \Haus^k$ 
where $E_s$ is a $k$-rectifiable set for a.e.~$s$, 
then $V(\mu,x)$ has dimension at least $k$ 
for $\mu$-a.e.~$x$. 
\end{itemizeb}
\end{proposition}

\begin{remark}
\label{s-examples}
(i)~Many popular examples of self-similar fractals, 
including the Von~Koch snowflake curve, the Cantor set, 
and the so-called Cantor dust (a cartesian product of
Cantor sets) are purely unrectifiable, and therefore 
every measure $\mu$ supported on them satisfies
$V(\mu,x)=\{0\}$ for $\mu$-a.e.~$x$ 
(Proposition~\ref{s-basic}(iv)). 

(ii)~The Sierpi\'nski carpet is a self-similar fractal
that contains many segments, and therefore is not purely 
unrectifiable. However, the \emph{canonical} probability measure 
$\mu$ associated to this fractal is supported on a purely 
unrectifiable set, and therefore $V(\mu,x)=\{0\}$
for $\mu$-a.e.~$x$.
The same occurs to other fractals of similar 
nature, such as the Sierpi\'nski triangle and the 
Menger-Sierpi\'nski sponge. 
\end{remark}

Before stating Proposition~\ref{generator} we need a definition.
Fix a measure $\mu$ on $\R^n$ and a family $\{\mu_t\}\in\F_\mu$, and
consider the class of all Borel maps $V:\R^n\to\Gr(\R^n)$
such that the inclusion \eqref{test} holds;
since this class is closed under countable intersection, by Lemma~\ref{e-minmap} 
it admits a minimal element which is unique modulo equivalence $\mu$-a.e. 
We call this map \emph{tangent bundle} associated to the family $\{\mu_t\}$, 
and we denote it by $x\mapsto V(\{\mu_t\},x)$.

\begin{proposition}
\label{generator}
Let $\mu$ be a measure on $\R^n$. Then
\begin{itemizeb}
\item[(i)]
for every $\{\mu_t\} \in \F_\mu$ there holds
$V(\mu,x) \supset V(\{\mu_t\}, x)$ for $\mu$-a.e.~$x$;
\item[(ii)]
there exists $\{\tilde\mu_t\} \in \F_\mu$
such that $V(\mu,x) = V(\{\tilde\mu_t\}, x)$ 
for $\mu$-a.e.~$x$.
\end{itemizeb}
\end{proposition}

\begin{proof}
Through this proof the index space $I$ is the interval $[0,1]$
equipped with the Lebesgue measure (cf.~Remark~\ref{s-2.4}(ii)).

Statement~(i) is obvious, 
and to prove statement~(ii) it suffices to find $\{\tilde\mu_t\}$ such that
\begin{equation}
\label{e-2.3}
V(\mu,x) \subset V(\{\tilde\mu_t\}, x) 
\quad\text{for $\mu$-a.e.~$x$.}
\end{equation}
For every $\{\mu_t\} \in \F_\mu$ we set
\[
F(\{\mu_t\}) := \int_{\R^n} \dim V(\{\mu_t\}, x)\,d\mu(x)
\, .
\]
We claim that there exists a family $\{\tilde\mu_t\} \in \F_\mu$ 
which maximizes $F$ over $\F_\mu$, and that this family 
satisfies \eqref{e-2.3}.

To prove the existence, we take a sequence of families
$\{\mu_{j,t}: t\in I\}$ which is maximizing for $F$, 
and let $\{\tilde\mu_t: t\in I\}$ be the ``union'' of these families, 
which is defined by  
\[
\tilde\mu_{2^{-j}(1+t)} := \mu_{j,t}
\quad\text{for every $j=1,2,\dots$ and $t\in (0,1]$.}
\]
One easily checks that $\{\tilde\mu_t\} \in \F_\mu$, 
and that $F(\{\tilde\mu_t\}) \ge F(\{\mu_{j,t}\})$ for every $j$, 
which proves that $\{\tilde\mu_t\}$ maximizes $F$.
In turn, this implies that 
for every other family $\{\mu_t\}\in\F_\mu$
there holds
\[
V(\{\tilde\mu_t\},x) \supset V(\{\mu_t\},x)
\quad\text{for $\mu$-a.e.~$x$}
\]
(if this were not the case, taking the ``union'' of $\{\tilde\mu_t\}$ 
and $\{\mu_t\}$ we would obtain a new family for which $F$ has a larger
values than for $\{\tilde\mu_t\}$.)
This inclusion clearly proves that $x\mapsto V(\{\tilde\mu_t\},x)$ 
belongs to $\G_\mu$, which yields \eqref{e-2.3}.
\end{proof}

To prove Proposition~\ref{s-basic}
we need the following lemma, 
the proof of which is postponed to Section~\ref{s7}.

\begin{lemma}
\label{s-rainwatercor3}
For every measure $\mu$ on $\R^n$ one of the 
following (mutually incompatible) alternatives 
holds:
\begin{itemizeb}
\item[(i)]
$\mu$ is supported on a purely unrectifiable
set $E$ (see \S\ref{s-rectif});
\item[(ii)]
there exists a nontrivial measure of the form 
$\smash{\mu'=\int_I \, \mu_t\, dt}$ 
such that $\mu'$ is absolutely continuous 
\wrt $\mu$ and each $\mu_t$ is the restriction 
of $\Haus^1$ to some $1$-rectifiable set $E_t$.
\end{itemizeb}
\end{lemma}

\begin{proof}
[Proof of Proposition~\ref{s-basic}]
We choose $\{\tilde\mu_t\} \in \F_\mu$ 
and $\{\tilde\mu'_t\} \in \F_{\mu'}$ 
by applying Proposition~\ref{generator}(ii)
to $\mu$ and $\mu'$ respectively.

\medskip
\textit{Statement~(i).}
If $\mu'\ll\mu$ then one easily checks that 
$\{\tilde\mu'_t\}\in\F_\mu$, and 
by Proposition~\ref{generator}
\[
V(\mu',x) 
= V(\{\tilde\mu'_t\},x) 
\subset V(\mu,x)
\quad\text{for $\mu'$-a.e.~$x$.}
\]
To prove the opposite inclusion, take a Borel set $F$ 
such that the restriction of $\mu$ to $F$ satisfies 
$1_F\mu\ll\mu'$.%
\footnoteb{For example, $F$ is the set where the 
Radon-Nikod\'ym density of $\mu'$ \wrt $\mu$
is strictly positive}
Then the family of all restrictions 
$\smash{ \{\mu'_t:=1_F\tilde\mu_t \} }$ 
belongs to $\F_{\mu'}$, and 
\[
V(\mu,x) 
=V(\{\tilde\mu_t\},x)
=V(\{\mu'_t\},x)
\subset V(\mu',x)
\quad\text{for $\mu$-a.e.~$x\in F$,}
\]
that is, for $\mu'$-a.e.~$x$.
The proof of the first part of statement~(i) is
thus complete.

By applying the first part of the statement~(i) to 
the measures $1_E\mu'$ and $\mu$, and then to 
the measures $1_E\mu'$ and $\mu'$ we obtain 
$V(1_E\mu',x)=V(\mu,x)=V(\mu',x)$ for $\mu'$-a.e.~$x\in E$, 
which proves the second part of statement~(i).

\medskip
\textit{Statement~(ii).}
Take $F_t$ such that $\tilde\mu_t$ is the restriction 
of $\Haus^1$ to $F_t$.
We observe that $\int \tilde\mu_t \, dt \ll\mu$ and 
$\mu(\R^n\setminus S)=0$ imply
\[
0 
=\int \tilde\mu_t(\R^n\setminus S) \, dt
=\int \Haus^1(F_t\setminus S) \, dt
\, , 
\]
which in turn implies that, for a.e.~$t$, the set 
$F_t$ is contained (up to an $\Haus^1$-null set) 
in $S$. Thus $\Tan(F_t,x)\subset \Tan(S,x)$ for 
$\mu_t$-a.e.~$x$, which implies that 
\[
V(\mu,x) 
= V(\{\tilde\mu_t\},x) 
\subset \Tan(S,x)
\quad\text{for $\mu$-a.e.~$x$.}
\]

\textit{Statement~(iii).}
Using statement~(i) and the definition of $k$-rectifiable 
set (see \S\ref{s-rectif}) we can reduce to the case 
$\mu=1_E\Haus^k$ where $E$ is a subset of a $k$-dimensional
surface $S$ of class $C^1$, and we can further assume 
that $S$ is parametrized by a diffeomorphism $g:A\to S$
of class~$C^1$, where $A$ is a bounded open set in $\R^k$.

We set $E':=g^{-1}(E)$ and $\smash{\mu':=1_{E'}\Leb^k}$.
Then we fix a nontrivial vector $e\in\R^k$,
and for every $t$ in the hyperplane $\smash{e^\perp}$
we let $E'_t$ be the intersection of the set $E'$ with the 
line $\{x'=t+he: \, h\in\R\}$, 
and set $\smash{\mu'_t:=1_{E'_t}\Haus^1}$.
By Fubini's theorem we have that $\smash{\mu'=\int \mu'_t\, dt}$
where $dt$ is the restriction of $\Haus^{k-1}$ to the
projection of $A$ onto~$\smash{e^\perp}$.

Next we set $E_t:=g(E'_t)$ and $\smash{\mu_t:=1_{E_t}\Haus^1}$.
Thus each $E_t$ is a $1$-rectifiable set whose tangent space 
at $x=g(x')$ is spanned by the vector $\tau(x) := dg(x') \, e$.
Moreover, taking into account that $g$ is a diffeomorphism, 
we get that $\int\mu_t \, dt$ and $\mu$ are mutually absolutely 
continuous. 
Therefore $\tau(x) \in V(\mu,x)$ for $\mu_t$-a.e.~$x$ 
and a.e.~$t$, that is, for $\mu$-a.e.~$x$.

Finally, we take vectors $e_1,\dots , e_k$ that span $\R^k$, 
thus the corresponding vectorfields $\tau_i(x):=dg(x') \, e_i$
span $\Tan(S,x)$ for every $x$, and we conclude that 
$\Tan(E,x) = \Tan(S,x) \subset V(\mu,x)$ for $\mu$-a.e.~$x$. 
The opposite inclusion follows from statement~(ii).

\medskip
\textit{Statement~(iv).}
We prove the ``if'' part first.
If $\mu$ is supported on a set $E$, 
then, arguing as in the proof
of statement~(ii), we obtain that for a.e.~$t$ the set
$F_t$ associated to $\tilde\mu_t$ is contained in $E$ 
up to an $\Haus^1$-null set. 
In particular, if $E$ is purely unrectifiable we
obtain that $\Haus^1(F_t)=0$, that is, $\tilde\mu_t=0$.
Hence $V(\mu,x)=V(\{\tilde\mu_t\},x)=\{0\}$ for $\mu$-a.e.~$x$.

The ``only if'' part follows
from Lemma~\ref{s-rainwatercor3}; indeed
the alternative~(ii) in this lemma is ruled 
out by the fact that $V(\mu,x)=\{0\}$ 
for $\mu$-a.e.~$x$.

\medskip
\textit{Statement~(v).}
For every $s\in I$ we choose a family
$\smash{ \{\tilde\nu_{s,t}: \, t\in I\} \in \F_{\nu_s} }$ 
according to Proposition~\ref{generator}(ii).
Then it is easy to check that the (two-parameter) 
family $\smash{ \{\tilde\nu_{s,t}: \, t,s\in I\} }$
belongs to $\F_\mu$,%
\footnoteb{On the parameter space $I\times I$ we consider
the product measure $dt\times ds$.}
and therefore Proposition~\ref{generator}(i)
implies
\[
V \big( \{\tilde\nu_{s,t}: \, s,t\in I\} , x \big) \subset V(\mu,x)
\quad\text{for $\mu$-a.e.~$x$}
\]
and therefore also for $\nu_s$-a.e.~$x$ and a.e.~$s$
(by assumption $\int \nu_s\, ds\ll\mu$).
On the other hand one can check that
\[
V(\nu_s,x) 
= V \big( \{\tilde\nu_{s,t}: \, t\in I\} , x \big) 
\subset V \big( \{\tilde\nu_{s,t}: \, s,t\in I\}, x \big)
\]
for $\nu_s$-a.e.~$x$ and a.e.~$s$, 
and statement~(v) is proved.%
\footnoteb{This proof is not correct as written, 
because the map $(s,t) \mapsto\tilde\nu_{s,t}$ is not 
necessarily Borel measurable in both variables 
(in the sense of Remark~\ref{s-2.4}(i)).
For a correct proof, the families $\{\tilde\nu_{s,t}:\, t\in I\}$
should be chosen for every $s\in I$ in a 
measurable fashion, and this can be achieved by means
of a suitable measurable selection theorem; 
since this statement is not essential for the rest
of the paper we omit the details.}

\medskip
\textit{Statement~(vi).}
By statement~(i) it suffices to prove the claim when
$\mu=\int \nu_s \, ds$. 
In this case statement~(v) implies that 
$V(\mu,x)$ contains $V(\nu_s,x)$
for $\nu_s$-a.e.~$x$ and a.e.~$s$, 
and $V(\nu_s,x)$ has dimension $k$
by statement~(iii). 
Thus $V(\mu,x)$ has dimension at least $k$
for $\nu_s$-a.e.~$x$ and a.e.~$s$, 
that is, for $\mu$-a.e.~$x$.
\end{proof}

%
%

\section{Proof of Theorem \ref{s-main}{\rm(i)}}
\label{s4}
We begin this section with a definition that is related 
to the notion of \emph{derivative assignment} introduced in 
\cite{MP}; note that Proposition~\ref{s-4.3} below---the key 
result in this section---can be viewed 
as a particular case of a general chain-rule for Lipschitz 
maps proved in that paper.

\begin{parag}
[Differentiability bundle]
\label{s-diffbund}
Given a Lipschitz $f:\R^n\to\R$ and a point $x \in\R^n$, 
we denote by $\D(f,x)$ the set of all subspaces $V\in\Gr(\R^n)$ 
such that $f$ is differentiable at $x$ \wrt $V$ 
(cf.~\S\ref{s-not}), and call the map $x\mapsto \D(f,x)$ 
\emph{differentiability bundle of $f$}.
We also denote by $\D^*(f,x)$ the set of all 
$V\in \D(f,x)$ with maximal dimension.
\end{parag}

\begin{remark}
(i)~Statement~(i) of Theorem~\ref{s-main}
can be restated by saying that for every measure
$\mu$ and every Lipschitz function $f$ the decomposability
bundle $V(\mu,x)$ belongs to $\D(f,x)$ for $\mu$-a.e.~$x$.

(ii)~Both $\D(f,x)$ and $\D^*(f,x)$ are 
closed and nonempty, and viewed as multifunctions 
in the variable $x$ are Borel measurable;%
\footnoteb{\label{f-4.1}
A~\emph{multifunction} from the set $X$ to the set 
$Y$ is a map that to every $x\in X$ associates a nonempty
subset of $Y$. For the definition and basic results
concerning (Borel) measurable multifunctions we refer 
to \cite{Sri}, Section~5.1. We just recall here that 
when $X$ is a topological space and $Y$ is a compact 
metric space, a closed-valued multifunction from 
$X$ to $Y$ is Borel measurable if it
is Borel measurable as a map from $X$ to the space of
non-empty closed subsets of $Y$, endowed with the 
Hausdorff distance (this case includes essentially all 
multifunctions considered in this paper).}
since these properties play only a minor role in this proof, 
we postpone the precise statement (Lemma \ref{s-meas2}) 
to Section~\ref{s7}.
Note that the set $\D^*(f,x)$ may contain more than 
one element.
\end{remark}

\begin{proposition}
\label{s-4.3}
Let $E$ be a $1$-rectifiable set in $\R^n$, and let 
$x\mapsto V(x)$ be a Borel map from $E$ to the space
metric space $\Gr(\R^n)$ (cf.~\S\ref{s-not})
such that $V(x)\in \D(f,x)$ for every $x\in E$. 
Then the following statements hold:
\begin{itemizeb}
\item[(i)]
$V(x) \oplus \Tan(E,x) \in \D(f,x)$ for $\Haus^1$-a.e.~$x\in E$;
\item[(ii)]
if in addition $V(x)\in\D^*(f,x)$ for every $x\in E$, 
then $\Tan(E,x) \subset V(x)$ for $\Haus^1$-a.e.~$x\in E$.
\end{itemizeb}
\end{proposition}

\begin{proof}[Proof of Proposition~\ref{s-4.3}]
Since $V(x)$ belongs to $\D(f,x)$, 
for every $x\in E$ we have that $f$ is differentiable
at $x$ \wrt $V(x)$, which means that 
\begin{equation}
\label{e-4.3}
f(x+h) - f(x) - \dV f(x)\, h = o_x(h)
\quad\text{for $h\in V(x)$.} 
\end{equation}
Using Lusin's and Egorov's theorems we can cover 
$\Haus^1$-almost all of $E$ with countably many 
subsets $E_j$ such that for every $j$
\begin{itemizeb}
\item[(a)]
$x\mapsto V(x)$ is continuous on $E_j$;
\item[(b)]
$x\mapsto \dV f(x)$ is continuous on $E_j$;%
\footnoteb{For this we need that the map $x\mapsto \dV f(x)$ 
is Borel measurable, see Lemma~\ref{s-meas1}.}
\item[(c)]
the remainder term $o_x(h)$ in \eqref{e-4.3} is uniform
over all $x\in E_j$, that is, for every $\eps>0$
there exists $r>0$ such that $|o_x(h)|\le \eps|h|$
for every $x\in E_j$ and every $h\in V(x)$ with $|h|\le r$.
\end{itemizeb}
Moreover, using the fact that $f$ is differentiable 
\wrt $\Tan(E,x)$ at $\Haus^1$-a.e.~$x$ in the
$1$-rectifiable set $E$,%
\footnoteb{This property is an immediate consequence
of Rademacher theorem for Lipschitz functions 
of one variable when $E$ is the image of 
a Lipschitz path, and can be extended to 
a general $1$-rectifiable set $E$ using the fact 
that by definition (see~\ref{s-rectif})
$E$ can be covered, up to an $\Haus^1$-null subset, 
by countably many curves of class $C^1$
(here we also need property \eqref{e-tangent}).}
and possibly replacing
each $E_j$ with a suitable subset,
we can further assume that 
\begin{itemizeb}
\item[(d)]
the approximate tangent space $\Tan(E_i,x)$ exists 
for all $x\in E_i$ and agrees with $\Tan(E,x)$,
which implies that for every $\tau\in\Tan(E,x)$
there exists $x'\in E_j$ such that
\begin{equation}
\label{e-4.4}
x' = x + \tau + o_x(\tau)
\, ;
\end{equation}
\item[(e)]
the function $f$ is differentiable 
\wrt $\Tan(E,x)=\Tan(E_i,x)$ at every $x\in E_j$, 
and we denote its derivative as $d_Tf(x)$.
\end{itemizeb}
We will prove next that $f$ is differentiable 
\wrt $V(x) \oplus \Tan(E,x)$ at every $x\in E_j$, 
and statement~(i) will follow from the fact that 
the sets $E_j$ cover $\Haus^1$-almost all of $E$, 
while statement~(ii) is an immediate consequence of
statement~(i) and the definition of $\D^*(f,x)$.

\medskip
For the rest of this proof we fix $j$ and $x\in E_j$.
We assume that $\Tan(E,x)$ is not contained 
in $V(x)$ (otherwise there is nothing to prove),
and we will show that
\begin{equation}
\label{e-4.5}
f(x+h+\tau) - f(x) - \dV f(x) \, h - d_T f(x) \, \tau 
= o(|h|+|\tau|) 
\end{equation}
for all $h\in V(x)$ and $\tau\in\Tan(E,x)$.
Here and in the rest of the proof the variables
are $h$ and $\tau$, and we use the Landau notation 
(for example, $o(|h|+|\tau|)$) with the understanding 
that $h$ and $\tau$ tend to~$0$.

For every $\tau\in\Tan(E,x)$ we choose $x'=x'(\tau)\in E_j$
such that \eqref{e-4.4} holds, and using assumption~(a)
above, for every $h\in V(x)$ we can find $h'=h'(\tau,h)\in V(x')$ 
such that
\begin{equation}
\label{e-4.6}
h'- h = o(1)
\, ,
\end{equation}
which implies 
\begin{equation}
\label{e-4.7}
|h'| \sim |h| 
\quad\text{as $h,\tau\to 0$.}
\end{equation}
We then obtain \eqref{e-4.5} by writing
\[
f(x+h+\tau) - f(x) - \dV f(x) \, h - d_T f(x) \, \tau
= I + I\!I + I\!I\!I + I\!V + V + V\!I
\]
where%
\,\footnoteb{%
The second estimate follows from \eqref{e-4.3}, assumption~(c), 
and \eqref{e-4.7}, the fourth one follows from
\eqref{e-4.7} and estimate
$\dV f(x') - \dV f(x)=o(1)$ (cf.~assumption~(b)),
the fifth one from \eqref{e-4.6}, and
the sixth one from assumption~(e). 
Since $f$ is Lipschitz, the third estimate follows 
from $|x'-(x+\tau)|= o(|\tau|)$ (cf.~\eqref{e-4.4}),
while the first one follows from
$|(x+h+\tau) - (x'+h')| \le |x+\tau-x'|+|h-h'| = o(|\tau|) + o(|h|)$.
}
\begin{align*}
  I 
& := f(x+h+\tau) - f(x'+h') = o(|\tau|) + o(|h|)
  \, , \\
  I\!I
& := f(x'+h') - f(x') - \dV f(x') \, h' = o(|h|) 
  \, , \\
  I\!I\!I 
& := f(x')-f(x+\tau) = o(|\tau|) 
  \displaybreak[1] \, , \\
  I\!V 
& := \big( \dV f(x') - \dV f(x) \big) h'  = o(|h|)
  \, , \\
  V 
& := \dV f(x) \, (h'-h) = o(|h|)
  \, , \\
  V\!I 
& := f(x+\tau) - f(x) - d_T f(x) \, \tau = o(|\tau|)
  \, . \qedhere
\end{align*}
\end{proof}

\begin{proof}[Proof of statement (i) of Theorem~\ref{s-main}]
Since $x\mapsto\D^*(f,x)$ is a Borel-measurable, 
close-valued multifunction from $\R^n$ to $\Gr(\R^n)$
(Lemma~\ref{s-meas2}), we can use Kuratowski and 
Ryll-Nardzewski's measurable selection theorem 
(see \cite{Sri}, Theorem~5.2.1) to find a Borel 
map $x\mapsto V(x)$ from $\R^n$ to $\Gr(\R^n)$ such that 
$V(x)\in\D^*(f,x)$ for every $x$.

Take now $\F_\mu$ and $\G_\mu$ as in \S\ref{s-decomp}
and let $\{\mu_t: t\in I\}$ be an arbitrary family in $\F_\mu$.
Then for every $t\in I$ the measure $\mu_t$ is the 
restriction of $\Haus^1$ to some rectifiable set $E_t$, 
and therefore $\Tan(E_t,x)$ is contained in $V(x)$ for 
$\mu_t$-a.e.~$x$ by statement~(ii) in Proposition~\ref{s-4.3}.
This implies that the map $x\mapsto V(x)$ belongs to 
the class $\G_\mu$, which means that $V(\mu,x)$ is contained 
in $V(x)$ for $\mu$-a.e.~$x$, and since $f$ is differentiable 
\wrt $V(x)$ for every $x$, it is also differentiable
\wrt $V(\mu,x)$ for $\mu$-a.e.~$x$.
\end{proof}

\begin{remark}
The proof above can be easily modified to obtain a  
stronger statement, namely that for $\mu$-a.e.~$x$ 
the linear space $V(\mu,x)$ is contained in $V$ 
for every $V\in\D^*(f,x)$.
In fact one can prove even slightly more: 
for $\mu$-a.e.~$x$ the linear space $V(\mu,x)$ is contained 
in every element of $\D(f,x)$ which is maximal with respect 
to inclusion.
\end{remark}

%
%

\section{Proof of Theorem \ref{s-main}{\rm(ii)}}
\label{s5}
Statement~(ii) of Theorem \ref{s-main}
is an immediate consequence of a more precise 
statement proved in Theorem~\ref{s-nondiff3} below.
To obtain this theorem we begin by proving a 
statement of similar nature under more restrictive 
assumptions (Proposition ~\ref{s-nondiff1}).

We begin this section with a few definitions,
then we give the main statements, 
and then the proofs.

Through this section 
$\mu$ is a measure on $\R^n$. 
Given a function $f$ on $\R^n$, 
a point $x\in\R^n$ and a vector $v\in\R^n$, 
we consider the upper and lower (one-sided)
directional derivatives
\begin{align*}
D^+_vf(x) & := \limsup_{h\to 0^+} \frac{f(x+hv)-f(x)}{h}
\, , \\
D^-_vf(x) & := \liminf_{h\to 0^+} \frac{f(x+hv)-f(x)}{h}
\, .
\end{align*}

\begin{parag}[The set $E$ and the space $X$]
\label{s-spacex}
For the rest of this section $E$ is a Borel set 
in $\R^n$ with the following property:
there exist an integer $d$ with $0< d \le n$, 
and continuous vectorfields $e_1,\dots,e_n$ 
on $\R^n$ such that
\begin{itemizeb}
\item
$e_1(x),\dots,e_n(x)$ form an orthonormal basis of
$\R^n$ for every $x\in\R^n$;
\item
$e_1(x),\dots,e_d(x)$ span $V(\mu,x)^\perp$ for every $x\in E$.
\end{itemizeb}
We write $D_j$ for the directional derivative $D_{e_j(x)}$, 
and denote by $X$ the space of all Lipschitz functions 
$f$ on $\R^n$ such that
\[
\big| D_j f(x) \big| \le 1
\quad\text{for $\Leb^n$-a.e.~$x$ and every $j=1,\dots,n$,}
\]
endowed $X$ with the supremum distance. 
It is then easy to show that $X$ is a complete metric space.
Note that this space depends on the measure $\mu$ and also
on the choice of the vectorfields $e_j$.
\end{parag}

\begin{proposition}
\label{s-nondiff1}
Given a vector $v\in\R^n$, 
let $N_v$ be the set of all $f\in X$ such 
that for $\mu$-a.e.~$x\in E$ there holds
\begin{equation}
\label{e-nondiff}
D^+_v f(x) - D^-_v f(x) 
\ge \frac{\dist(v,V(\mu,x))}{3\sqrt d}
\, .
\end{equation}
Then $N_v$ is residual in $X$, and in 
particular it is dense.%
\footnoteb{Recall that a set in a topological 
space is residual if it contains a countable
intersection of open dense sets; by Baire Theorem 
a residual set in a complete metric space is dense.}
\end{proposition}

\begin{proposition}
\label{s-nondiff2}
Let $N$ be the set of all $f\in X$ such that, 
for $\mu$-a.e.~$x\in E$, inequality 
\eqref{e-nondiff} holds for every $v\in\R^n$.
Then $N$ is residual in $X$, and in 
particular it is dense in $X$.
\end{proposition}

\begin{theorem}
\label{s-nondiff3}
There exists a Lipschitz function $f$ on $\R^n$ such that, 
for $\mu$-a.e.~$x\in\R^n$, there holds 
$D^+_v f(x) - D^-_v f(x) >0$ for every $v\notin V(\mu,x)$.
\end{theorem}

\begin{remark}
(i)~The function $f$  in 
Theorem~\ref{s-nondiff3} is not differentiable at $x$ in 
the direction $v$ for every $v\notin V(\mu,x)$
and for $\mu$-a.e.~$x$; this proves
statement~(i) of Theorem~\ref{s-main}.

(ii)~In Propositions~\ref{s-nondiff1} and \ref{s-nondiff2}
the class of non-differentiable functions 
under consideration is proved to be residual, and 
not just nonempty; from this point of view both 
statements are stronger than Theorem~\ref{s-nondiff3}, 
which we could not frame as a residuality result.

(iii)~Note that $N_{cv}=N_v$ for every $v\in\R^n$ 
and every $c>0$ (because both terms in 
inequality \eqref{e-nondiff} are $1$-homogeneous
in $v$), and therefore it suffices to prove
Proposition~\ref{s-nondiff1} under 
the additional assumption $|v|=1$.  
\end{remark}

Proposition~\ref{s-nondiff2} and
Theorem~\ref{s-nondiff3} follow easily from 
Proposition~\ref{s-nondiff1}, which is 
therefore the key result in the whole section.
To prove Proposition~\ref{s-nondiff1} 
we follow a general strategy 
for the proof of residuality results
devised by B.~Kirchheim (see for instance \cite{BK}).
The starting point is the following definition.

\begin{parag}[The 
operators $\smash{T^\pm_{\sigma,\sigma'}}$ and $U_\sigma$]
\label{s-5.4}
From now till the end of the proof of 
Lemma~\ref{s-5.9} we fix a vector $v$ in $\R^n$ with 
$|v|=1$, and for every $x\in\R^n$ we set
\[
d_v(x):= \dist(v, V(\mu,x))
\, .
\]
For every $\sigma>\sigma' \ge 0$,
and every function $f:\R^n\to\R$
we consider the functions  
$T^\pm_{\sigma,\sigma'}f$ and $U_\sigma f$
defined as follows for every $x\in\R^n$:
\begin{align*}
 T^+_{\sigma,\sigma'} f(x)
& := \sup_{\sigma'< h \le\sigma} \, \frac{f(x+hv) - f(x)}{h}
     \, , \\
 T^-_{\sigma,\sigma'} f(x)
& := \inf_{\sigma' < h \le\sigma} \, \frac{f(x+hv) - f(x)}{h}
     \, , \\
 U_\sigma f(x)
& := T^+_{\sigma,0}f(x) - T^-_{\sigma,0}f(x)
     \, .
\end{align*}
One readily checks that $\smash{T^+_{\sigma,0} f(x)}$ and
$\smash{T^-_{\sigma,0} f(x)}$ are respectively increasing and 
decreasing in $\sigma$, and therefore $\smash{U_\sigma f(x)}$
is increasing in $\sigma$. Moreover
\begin{equation}
\label{e-5.0}
D^+_vf(x) - D^-_vf(x) 
= \inf_{\sigma>0} \big( U_\sigma f(x) \big) 
= \inf_{m=1,2,\dots} \big( U_{1/m} f(x) \big) 
\, .
\end{equation}
Finally we notice that $\frac{1}{h}(f(x+hv)-f(x))$
and $D_v f(x)$ (if it exists) 
are both smaller than $\smash{T^+_{\sigma,0} f(x)}$ 
and larger than $\smash{T^-_{\sigma,0} f(x)}$
if $h\le \sigma$, 
which yields the following useful estimate:
\begin{equation}
\label{e-5.0.0}
U_\sigma f(x) 
\ge \left| D_v f(x) - \frac{f(x+h v)-f(x)}{h} \right| 
\quad\text{for every $0<h\le\sigma$.} 
\end{equation}

\end{parag}

\begin{proof}[Proof of Proposition~\ref{s-nondiff1}]
By Lemma~\ref{s-baireone} below 
$U_\sigma$ is of Baire~class~1 as a map 
from $X$ to $L^1(\mu)$ for every $\sigma>0$,%
\footnoteb{For the definition of maps of 
Baire~class~1 between two metrizable spaces see 
\cite{Kechris}, Definition~24.1.
We just recall here that this class contains 
(but does not always agrees with) the class 
of all pointwise limit of a sequences of 
continuous maps.}
and by Lemma~\ref{s-5.9} we have 
\begin{align}
\label{e-5.1}
U_\sigma f(x) \ge \frac{d_v(x)}{3\sqrt d}
\quad\text{for $\mu$-a.e.~$x\in E$}
\end{align}
whenever $f$ is a continuity point of $U_\sigma$.
Since the points of continuity of a map of 
Baire~class~1 are residual (see \cite{Kechris}, 
Theorem~24.14), it follows that the class 
$N_{v,\sigma}$ of all $f\in X$ which satisfy 
\eqref{e-5.1} is residual in $X$.
But then also the intersection of all
$N_{v,1/m}$ with $m$ a positive integer is residual,
and \eqref{e-5.0} implies that this intersection 
agrees with $N_v$.
\end{proof}

\begin{lemma}
\label{s-baireone}
For every $\sigma > \sigma' \ge 0$, the maps 
$\smash{T^\pm_{\sigma,\sigma'}}$ take $X$ into $L^1(\mu)$, 
are continuous for $\sigma'>0$, 
and are of Baire~class~1 for $\sigma'=0$.
Consequently $U_\sigma$ is a map from $X$ to $L^1(\mu)$
of Baire~class~1.
\end{lemma}

\begin{proof}
The functions $\smash{T^+_{\sigma,\sigma'}f}$ 
belong to $L^1(\mu)$ for every 
$\sigma > \sigma' \ge 0$
and every $f\in X$ 
because they are bounded, 
and more precisely 
\begin{equation}
\label{e-5.2}
\big| T^+_{\sigma,\sigma'}f(x) \big| \le \Lip(f)
\quad\text{for every $x\in\R^n$.}
\end{equation}

Concerning the continuity of 
$\smash{T^+_{\sigma,\sigma'}}$, 
one readily checks that given $\sigma'>0$
and $f,f'\in X$ there holds
\[
\big| T^+_{\sigma,\sigma'} f'(x) 
     -T^+_{\sigma,\sigma'} f(x) \big| 
\le \frac{2}{\sigma'}\|f' - f\|_\infty 
\quad\text{for every $x\in\R^n$,}
\]
and therefore 
\[
\big\| T^+_{\sigma,\sigma'} f' 
      -T^+_{\sigma,\sigma'} f \big\|_{L^1(\mu)} 
\le \frac{2}{\sigma'} \Mass(\mu) \, \|f' - f\|_\infty 
\, .
\]

To prove that $\smash{T^+_{\sigma,0}}$ is of
of Baire~class~1 it suffices to notice that it
is the pointwise limit of the continuous maps
$\smash{T^+_{\sigma,\sigma'}}$ as $\sigma'\to 0$.
Indeed, it follows from the definition that, 
as $\sigma'$ tends to $0$,  
$\smash{T^+_{\sigma,\sigma'} f(x)}$ converges 
to $\smash{T^+_{\sigma,0} f(x)}$ for every 
$f\in X$ and every $x\in\R^n$, 
and then $\smash{T^+_{\sigma,\sigma'} f}$ 
converges to $\smash{T^+_{\sigma,0} f}$
in $L^1(\mu)$ by the dominated convergence 
theorem (a domination is given by estimate 
\eqref{e-5.2}).

The rest of the statement can be proved in a similar
way.
\end{proof}

\begin{lemma}
\label{s-perturb3}
Let $\eps, \sigma$ be positive real numbers, 
$f$ a function in $X$, and $E'$ a Borel subset of $E$.
Then there exist a smooth function $f''\in X$ 
and a compact set $K$ contained in $E'$ such that 
\begin{itemizeb}
\item[(i)]
$\|f''-f\|_\infty \le 2\eps$;
\item[(ii)]
$\mu(K) \ge \mu(E')/(4d)$;
\item[(iii)]
$U_\sigma f''(x) \ge d_v(x)/(3\sqrt d)$
for every $x\in K$.
\end{itemizeb}
\end{lemma}

This is the key step in the proof of 
Proposition~\ref{s-nondiff1}; for the proof 
we need the following definition and results
(the proofs of which are postponed to Section~\ref{s7}).

\begin{parag}[Cones and cone-null sets]
\label{s-cones}
Given a unit vector $e$ in $\R^n$ and a real number 
$\alpha \in (0,\pi/2)$ we denote by
$C(e,\alpha)$ the closed cone of axis $e$ 
and angle $\alpha$ in $\R^n$, that is, 
\[
C(e,\alpha) 
:= \big\{ v\in\R^n: \ v \cdot e \ge \cos\alpha\cdot |v| \big\}
\, .
\]
Given a cone $C = C(e,\alpha)$ in $\R^n$, 
we call \emph{$C$-curve} any set of the form
$\gamma(J)$ where $J$ is a compact interval in $\R$ 
and $\gamma:J\to\R^n$ is a Lipschitz path 
such that
\[
\dot\gamma(s) \in C
\quad\text{for $\Leb^1$-a.e.~$s\in J$.}
\]
Following~\cite{ACP3}, we say that a set $E$ in $\R^n$ 
is \emph{$C$-null} if 
\[
\Haus^1 \big( E\cap G \big) = 0
\quad\text{for every $C$-curve $G$.}
\]
\end{parag}

\begin{proposition}
\label{s-rainwatercor4}
Let be given a Borel set $F$ in $\R^n$ 
and a cone $C=C(e,\alpha)$ in $\R^n$ such that
\[
V(\mu,x) \cap C =\{0\}
\quad\text{for $\mu$-a.e.~$x\in F$.}
\]
Then there exists a $C$-null set $F'$ 
contained in $F$ such that $\mu(F')=\mu(F)$.
\end{proposition}

\begin{lemma}
\label{s-perturb2}
Let be given a closed ball $B=B(\bar x,r)$ in $\R^n$, 
a cone $C=C(e,\alpha)$ in $\R^n$, 
and a $C$-null compact set $K$ contained in $B$. 
Then for every $\eps>0$ and every $r'>r$ there exists a 
smooth function $g:\R^n\to\R$ such that
\begin{itemize}
\item[(i)]
$\|g\|_\infty \le \eps$ and the support of $g$ is contained 
in $B':=B(\bar x,r')$;
\end{itemize}
and for every $x\in B'$,
\begin{itemize}
\item[(ii)]
$-\eps\le D_eg(x) \le 1$, 
and $D_eg(x) = 1$ if $x\in K$;
\item[(iii)]
$|d_W g(x)|\le 2/\tan\alpha$ where $W:=e^\perp$.%
\footnoteb{Here $\dV g(x)$ is the derivative of $g$ at $x$ \wrt $V$
(see \S\ref{s-not}), and $|\dV g(x)|$ is its operator norm.}
\end{itemize}
\end{lemma}

\begin{proof}
[Proof of Lemma~\ref{s-perturb3}]
The idea is to take a smooth function $f'$ 
close to $f$, and then modify it into
a function $f''$ so to get $U_\sigma f''(x)$ large
enough for sufficiently many $x\in E'$.
This modification will be obtained by adding to $f'$ 
a finite number of smooth ``perturbations'' 
with small supremum norms and small, disjoint  supports,
but with large derivatives in at least one direction.

\medskip
\textit{Step~1. 
We take a smooth function 
$f'$ on $\R^n$ such that
\begin{itemizeb}
\item[(a)]
$\|f'-f\|_\infty \le \eps$;
\item[(b)]
$\|D_jf'\|_\infty <1$
for $j=1,\dots,n$, and in particular $f'$ belongs to $X$
(cf.~\S\ref{s-spacex}).
\end{itemizeb}}
\indent
We take $s>0$ such that $s\|f\|_\infty<\eps$ and set
\[
f':=(1-s) f*\rho_t
\]
where $\rho_t(x)=t^{-n}\rho(x/t)$ 
and $\rho$ is a mollifier with compact support.

Using that $f$ is uniformly continuous one can easily check 
that $f'$ converges uniformly to $(1-s) f$ as $t \to 0$, 
then $\|f'-f\|_\infty$ converges to 
$s\|f\|_\infty <\eps$, which implies that (a) holds
for $t$ small enough.

Since the vectorfield $e_j$ that defines the
partial derivative $D_j$ is continuous, it is not
difficult to show that $\|D_jf'\|_\infty$ converges
as $t\to 0$ to $(1-s)\|D_jf\|_\infty <1$ 
(recall that $\|D_jf\|_\infty\le 1$ because $f\in X$)
and therefore also (b) holds for $t$ small enough.

\medskip
\textit{Step~2. Construction of the set $E'_k$.}
\\ 
\indent
For every integer $k$ with $1 < k \le d$ we set 
\begin{equation}
\label{e-5.4}
E'_k 
:= \Big\{ x\in E' \, : \
      |v\cdot e_k(x)| \ge \frac{d_v(x)}{\sqrt d}
   \Big\}
\, .
\end{equation}
Now, for every $x\in E$ we have that $V(\mu,x)^\perp$ is spanned 
by the orthonormal basis $e_1(x),\dots, e_d(x)$ 
(see \S\ref{s-spacex}) and therefore
\[
d_v(x) 
= \dist(v,V(\mu,x))
= \bigg[ \sum_{k=1}^d ( v\cdot e_k(x) )^2 \bigg]^{1/2}
\le \sqrt{d} \, \sup_{1 \le k\le d} |v\cdot e_k(x)|
  \, .
\]
This implies that every $x\in E'$ 
must belong to $E'_k$ for at least one $k$,
that is, the sets $E'_k$ cover $E'$.
In particular there exists at least one
value of $k$ such that
\begin{equation}
\label{e-5.5}
\mu(E'_k)\ge \frac{\mu(E')}{d}
\, .
\end{equation}
For the rest of the proof $k$ is assigned this 
specific value.

\medskip
For the next four steps we fix 
a point $\bar x\in E'_k$ and positive number $r,r'$ 
such that 
\begin{equation}
\label{e-5.5.1}
0<r<\sigma/3 
\, , \quad
r< r'\le 2r
\, .
\end{equation}

\medskip
\textit{Step~3. Construction of the sets $E_{\bar x,r}$.}
\\ 
\indent
Let $\alpha(\bar x,r)$ be the supremum of the 
angle between $V(\mu,x)$ and $V(\mu,\bar x)$ as $x$ varies 
in $E\cap B(\bar x,r)$.%
\footnoteb{The angle between two $d$-dimensional 
planes $V$ and $V'$ in $\R^n$ is $\arcsin(\dgr(V,V'))$, where 
$\dgr(V,V')$ is defined in \S\ref{s-essspan}.}
Since $V(\mu,x)$ is continuous in $x\in E$
(cf.~\S\ref{s-spacex}), we have that 
\begin{equation}
\label{e-5.6}
\alpha(\bar x,r) \to 0 
\quad\text{as $r\to 0$.}
\end{equation}
Moreover the cone 
\[
C(\bar x,r) 
:= C \big( e_k(\bar x), \pi/2 - 2\, \alpha(\bar x,r) \big)
\]
satisfies
\[
C(\bar x,r) \cap V(\mu,x)=\{0\}
\quad\text{for all $x\in E\cap B(\bar x,r)$,}
\]
and since $E'_k$ is contained in $E$, 
we can apply Proposition~\ref{s-rainwatercor4} 
to the set $F:=E'_k\cap B(\bar x,r)$
and find a $C(\bar x,r)$-null set $F'$ 
contained in $F$ that $\mu(F')=\mu(F)$.
Then we can take a compact set $K_{\bar x,r}$ contained
in $F'$ such that 
\begin{equation}
\label{e-5.7}
\mu(K_{\bar x,r}) 
\ge \frac{1}{2} \mu(F')
= \frac{1}{2} \mu (E'_k\cap B(\bar x,r)) 
\, .
\end{equation}
Note that $K_{\bar x,r}$ is $C(\bar x,r)$-null because it is contained 
in $F'$.

\medskip
\textit{Step~4. Construction of the perturbations $\bar g_{\bar x,r,r'}$.}
\\ 
\indent
We set
\[
\eps' := \min\big\{
   \eps, \, r(r'-r), \, 1-\|D_jf'\|_\infty~\text{with $j=1,\dots,n$}
\big\}
\]
Since $K_{\bar x,r}$ is $C(\bar x,r)$-null and $\eps'>0$,%
\footnoteb{The number $\eps'$ is strictly positive 
because $r'>r$ and because of statement~(b) in Step~1.}
we can use Lemma~\ref{s-perturb2} 
to find a function $g_{\bar x,r,r'}$ such that
\begin{itemize}
\item[(c)]
$\|g_{\bar x,r,r'}\|_\infty \le \eps'$ 
and the support of $g_{\bar x,r,r'}$ 
is contained in $B(\bar x,r')$;
\end{itemize}
and for every $x\in B(\bar x,r')$, 
\begin{itemize}
\item[(d)]
$-\eps'\le D_e g_{\bar x,r,r'}(x) \le 1$ 
and $D_e g_{\bar x,r,r'}(x) = 1$ if $x\in K_{\bar x,r}$, 
where $e:=e_k(\bar x)$;
\item[(e)]
$|d_W g_{\bar x,r,r'}(x)|\le 2\tan\!\big( 2\, \alpha(\bar x,r) \big)$ 
where $W := e^\perp = e_k(\bar x)^\perp$.
\end{itemize}
Finally we set
\[
\bar g_{\bar x,r,r'} := \pm \frac{1}{2} g_{\bar x,r,r'}
\]
where $\pm$ means $+$ if $D_k f'(\bar x) \le 0$ and $-$ otherwise.

\medskip
\textit{Step~5. There exists 
$r_0=r_0(\bar x) > 0$ such that for $r<r_0$ there holds}
\begin{equation}
\label{e-5.8}
U_\sigma (f'+ \bar g_{\bar x,r,r'})(x) 
\ge \frac{d_v(x)}{3\sqrt d}
\quad\textit{for every $x\in K_{\bar x,r}$.}
\end{equation}
\indent
In the following, 
given a quantity $m$ depending on $\bar x, r, r'$ 
and $x\in B(\bar x,r)$, 
write $m=o(1)$ to mean that, for every $\bar x$, 
$m$ tends to $0$ as $r\to 0$, uniformly in all 
remaining variables.%
\footnoteb{In other words, for every $\bar x$ and every 
$\eps>0$ there exists $\bar r>0$ such that $|m|\le \eps$ 
if $r\le \bar r$.}
To simplify the notation, we write $g$ and 
$\bar g$ for $g_{\bar x,r,r'}$ and $\bar g_{\bar x,r,r'}$.

For every $x\in K_{\bar x,r} \subset B(\bar x,r)$ 
we take $h=h(x)>0$ 
such that $x+hv$ belongs to $\bd B(\bar x, r')$. 
Then, taking into account that $|v|=1$ 
and \eqref{e-5.5.1}, we have
\[
r'-r \le h\le r+r' \le 3r \le \sigma
\, .
\] 
We can then apply estimate \eqref{e-5.0.0} 
to the function $f'':=f'+ \bar g$, and  
taking into account that $\bar g =\pm\frac{1}{2}g$ 
and $g(x+hv)=0$,%
\footnoteb{The support of $g$ is contained in $B(\bar x,r')$
by statement~(c) in Step~4.}
we get
\begin{align}
U_\sigma f''(x) 
& \ge \Big| D_v f'' (x) - \frac{f''(x+hv)-f''(x)}{h} \Big| 
  \nonumber \\
& = \Big| D_v \bar g(x) + D_vf'(x) - \frac{f'(x+hv)-f'(x)}{h}
        + \frac{\bar g(x)}{h} \Big|
  \nonumber \\
& \ge \frac{1}{2} \big| D_v g(x) \big| 
        - \Big| D_vf'(x) - \frac{f'(x+hv)-f'(x)}{h} \Big|
        - \frac{|g(x)|}{2h}
        \, .
  \label{e-5.9}
\end{align}
Since $f'$ is of class $C^1$, we clearly have
\begin{equation}
\label{e-5.11}
\Big| D_vf'(x) - \frac{f'(x+hv)-f'(x)}{h} \Big| 
= o(1)
\, .
\end{equation}
Using statement~(c) in Step~4, the inequality $r'-r<h$ given above, 
and the choice of $\eps'$, we get
\begin{equation}
\label{e-5.12}
\frac{|g(x)|}{h} \le \frac{\eps'}{r'-r} 
\le r = o(1)
\, .
\end{equation}
Finally, to estimate $|D_v g(x)|$ we decompose 
$v$ as $v = (v\cdot e)e + w$ with $e:=e_k(\bar x)$
and $w\in W:=e^\perp$. Then 
\[
D_v g(x)
= (v\cdot e) \, D_e g(x)  + \scal{d_W g(x)}{w}
\]
and therefore%
\,\footnoteb{The second inequality follows from 
statements (d) and (e) in Step~4
and the fact that $x\in K_{\bar x,r}$;
for the third one we used that $|v|=1$ and $e=e_k(\bar x)$;
the fourth follows from \eqref{e-5.6} and the fact 
$e_k(x)$ is continuous in $x$, and the last inequality
follows from \eqref{e-5.4} and the fact that 
$x\in K_{\bar x,r} \subset E'_k$.}
\begin{align}
|D_v g(x)|
& \ge |v\cdot e| \, |D_e g(x)| - |d_W g(x)| 
  \nonumber \\
& \ge |v\cdot e| 
    - 2 \tan\!\big( 2\,\bar\alpha(\bar x,r)\big)
  \nonumber \\
& \ge |v\cdot e_k(x)| - |e_k(x)-e_k(\bar x)|
    - 2 \tan\!\big( 2\, \bar\alpha(\bar x,r)\big)
  \nonumber \\
& \ge |v\cdot e_k(x)| - o(1) 
  \ge d_v(x) /\sqrt d - o(1)
  \, .
  \label{e-5.13}
\end{align}
Putting estimates \eqref{e-5.9}, 
\eqref{e-5.11}, \eqref{e-5.12}, \eqref{e-5.13}
together we get 
\[
U_\sigma (f'+\bar g)(x) 
= U_\sigma f''(x) 
\ge  \frac{d_v(x)}{2\sqrt d} - o(1)
\]
which clearly implies the claim in Step~5.

\medskip
\textit{Step~6. There exists $r_1=r_1(\bar x)>0$ such that
$f'+\bar g_{\bar x,r,r'} \in X$ if $r<r_1$.}
\\ 
\indent
Since $\bar g$ is supported in 
$B(\bar x,r')$ and $f'$ belongs to $X$ (Step~1), 
to prove that $f'+\bar g$ belongs to $X$ it suffices
to show that
\begin{equation}
\label{e-5.14}
\big| D_j(f'+\bar g_{\bar x,r,r'})(x) \big| \le 1
\quad\textit{for every $x\in B(\bar x,r')$ and $j=1,\dots,n$.}
\end{equation}
We begin with the case $j=k$.
Recalling the definition of $D_k$
(in~\S\ref{s-spacex}) and the identities 
$\bar g=\pm\frac{1}{2}g$, $e=e_k(\bar x)$, 
we obtain
\begin{align*}
  D_k \bar g(x) 
& = D_e\bar g(x) + \scal{d\bar g(x)}{e(x)-e} 
  = \pm \frac{1}{2} D_e g(x) + o(1) 
  \, , \\
  D_k f'(x) 
& = D_kf'(\bar x) + \big( D_kf'(x)-D_kf'(\bar x) \big)
  = D_kf'(\bar x) + o(1)
  \, , 
\end{align*}
and therefore 
\begin{equation}
\label{e-5.16}
\big| D_k (f'+\bar g)(x) \big| 
= \Big| D_kf'(\bar x)\pm \frac{1}{2} D_e g(x) \Big| 
  +o(1)
\, .
\end{equation}
Recall now that $-\eps'\le D_eg(x) \le 1$ 
(statement~(d) above), that the sign $\pm$ means $+$ 
when $D_kf'(\bar x)\le 0$ and $-$ otherwise, 
and that we chose $\eps'$ so that 
$\|D_kf'\|_\infty \le 1-\eps'$.
Using these facts we get the estimate
\[
\Big| D_kf'(\bar x)\pm \frac{1}{2} D_e g(x) \Big| 
\le 1 - \eps'/2
\, , 
\]
which, together with \eqref{e-5.16}, clearly implies that 
\eqref{e-5.14} holds for $r$ small enough.

To prove \eqref{e-5.14} for $j\ne k$ is actually simpler:
recall indeed that $\|D_jf'\|_\infty<1$ 
(statement~(b) above) and note that%
\,\footnoteb{For the second inequality we use 
statement~(e) in Step~4 and the fact that 
$\bar g=\pm\frac{1}{2}g$.}
\begin{align*}
  |D_j \bar g(x)|
& \le \big| \scal{d \bar g(x)}{e_j(\bar x)} \big| 
   +  \big| \scal{d \bar g(x)}{e_j(x)-e_j(\bar x)} \big| \\
& \le \tan\!\big( 2\, \alpha(\bar x,r)\big) 
   + |d \bar g(x)| \, |e_j(x)-e_j(\bar x)|
  =o(1) \, .
\end{align*}

\medskip
\textit{Step~7. Construction of the function $f''$ and the set $K$.}
\\ 
\indent
We consider the family $\G$ of all closed balls $B(\bar x,r)$
with $\bar x\in E'_k$ and $r>0$ such that the conclusions 
of Step~5 and Step~6 hold (that is, 
$r$ smaller than $r_0(\bar x)$ and $r_1(\bar x)$).
By a standard corollary of Besicovitch covering
theorem (see for example \cite{KP}, Proposition~4.2.13) we 
can extract from $\G$ finitely many disjoint balls 
$B_i=B(\bar x_i,r_i)$ such that
\begin{equation}
\label{e-5.18}
\sum_i \mu( E'_k \cap B_i) 
\ge \frac{1}{2} \mu (E'_k)
\, .
\end{equation}
Since the balls $B_i=B(\bar x_i,r_i)$ are closed and disjoint, 
for every $i$ we can find $r'_i>r_i$ 
such that the enlarged balls $B'_i:=B(\bar x_i,r'_i)$ are
still disjoint.
Finally, for every $i$ we set $\bar g_i:=\bar g_{\bar x_i,r_i,r'_i}$, 
$K_i:=K_{\bar x_i,r_i}$, 
and
\[
f'' := f' +\sum_i \bar g_i
\, , \quad
K:=\bigcup_i K_i
\, .
\]
We now check that $f''$ and $K$ satisfy all requirements.

The function $f''$ is smooth because so are $f'$ 
and $\bar g_i$, and the set $K$ is compact 
because so are the sets $K_i$.
 
Note that the supports of the functions $\bar g_i$ 
are disjoint (because they are contained in the balls $B'_i$), 
and therefore at every point $x\in\R^n$ the derivative 
of $f''$ agrees either with the derivative of $f'$ 
or with that of $f'+\bar g_i$ for some $i$.
Recalling the definition of $X$ in 
\S\ref{s-spacex}, we then deduce that $f''$ belongs 
to $X$ by the fact that $f'$ belongs to 
$X$ (Step~1) and $f'+\bar g_i$
belongs to $X$ for every $i$ (Step~6).

Statement~(i), namely that $\|f''-f\|\le 2\eps$, follows from 
statements~(a) in Step~1 and (c) in Step~4, and the fact that the functions 
$g_i$ have disjoint supports.

Statement~(ii), namely that $\mu(K) \ge \mu(E')/(4d)$,
follows from estimates \eqref{e-5.7}, 
\eqref{e-5.18}, and \eqref{e-5.5}. 

Consider now $x \in K_i$ for some $i$.
By Step~5, $U_\sigma(f'+\bar g_i)(x)\ge d_v(x)/(3\sqrt d)$. 
Moreover the proof of this estimates involves 
only the restriction of $f'+\bar g_i$
to the ball $B'_i$, where $f'+\bar g_i$ agrees with $f''$. 
Thus the same estimates holds for $U_\sigma f''(x)$ as well, 
which means that statement~(iii) holds.
\end{proof}

\begin{lemma}
\label{s-5.9}
Take $f\in X$ and $\sigma>0$.
If $U_\sigma$ is continuous at $f$ 
(as a map from $X$ to $L^1(\mu)$)
then \eqref{e-5.1} holds.
\end{lemma}

\begin{proof}
We assume that \eqref{e-5.1} fails and prove that 
$U_\sigma$ is not continuous at $f$.
Indeed, if \eqref{e-5.1} does not hold, we can find 
a set $E'$ contained in $E$ with $\mu(E')>0$
and $\delta>0$ such that
\[
U_\sigma f(x) \le \frac{d_v(x)}{3\sqrt d} -\delta
\quad\text{for every $x\in E'$.}
\]
Then we use Lemma~\ref{s-perturb3} to construct
a sequence of smooth functions $f_h\in X$ and of compact 
sets $K_h$ contained in $E'$ such that 
$f_h \to f$ uniformly as $h\to+\infty$, 
and for every $h$ there 
holds $\mu(K_h) \ge \mu(E')/(4d)$ and 
\[
U_\sigma f_h(x) \ge \frac{d_v(x)}{3\sqrt d} 
\quad\text{for every $x\in K_h$.}
\]
Thus $U_\sigma f_h$ does not converge 
to $U_\sigma f$ in the $L^1(\mu)$-norm, 
and more precisely 
\[
\big\| U_\sigma f_h - U_\sigma f \big\|_{L^1(\mu)}
\ge \int_{K_h} \big| U_\sigma f_h - U_\sigma f \big|\, d\mu
\ge \delta \mu(K_h) \ge \frac{\delta}{4d}\mu(E')
\, .
\qedhere
\]
\end{proof}

\begin{proof}
[Proof of Proposition~\ref{s-nondiff2}]
Let $D$ be a countable dense subset of $\R^n$, 
and let $N'$ be the intersection of all sets $N_v$ 
defined in Proposition~\ref{s-nondiff1} with $v\in D$.
By Proposition~\ref{s-nondiff1} the sets $N_v$ are 
residual in $X$, and then also $N'$ is residual.

Let now be given $f\in N'$. 
One readily checks that 
for $\mu$-a.e.~$x\in E$ inequality \eqref{e-nondiff} 
holds for every $v\in D$, and we deduce that 
it actually holds for every $v\in\R^n$ using
the fact that both sides of \eqref{e-nondiff} 
are continuous in $v$ (and $D$ is dense in $R^n$); 
notice indeed that the directional upper and lower 
derivatives $D_v^\pm f(x)$ are Lipschitz in $v$
(with the same Lipschitz constant as $f$).

We have thus proved that $f$ belongs to $N$, 
thus $N$ contains $N'$, and therefore it is
residual.
\end{proof}

For the proof of Theorem~\ref{s-nondiff3}
we need the following lemma, the proof of 
which is postponed to Section~\ref{s7}.

\begin{lemma}
\label{s-regularize}
Let $f$ be a Lipschitz function on $\R^n$, $K$ a 
compact set in $\R^n$, and $\phi$ an increasing, 
strictly positive function on $(0,+\infty)$. 
Then for every $\eps > 0$ there exists a Lipschitz 
function $g:\R^n\to\R$ such that
\begin{itemizeb}
\item[(i)]
$g$ agrees with $f$ on $K$ and is smooth in $\R^n\setminus K$;
\item[(ii)]
$| g(x) - f(x) | \le \phi(\dist(x,K))$ for every 
$x\in \R^n$;
\item[(iii)]
$\Lip(g) \le \Lip(f)+\eps$.
\end{itemizeb}
\end{lemma}

\begin{proof}
[Proof of Theorem~\ref{s-nondiff3}]
The strategy is simple: 
we cover $\R^n$ with
a countable family of pairwise disjoint sets $E_i$ 
which satisfy the assumption in \S\ref{s-spacex}, then 
we use Proposition~\ref{s-nondiff2} to find
functions $f_i$ which satisfy \eqref{e-nondiff}
for every $v$ and $\mu$-a.e.~$x\in E_i$, 
and we regularize these functions out of the 
set $E_i$ using Lemma~\ref{s-regularize}; 
finally we take as $f$ a weighted sum of these
modified functions.

\medskip
For every $x\in \R^n$ let $d(x)$ be the dimension 
of $V(\mu,x)^\perp$, and let $F_0$ be the
set of all $x$ such that $d(x)>0$.

\medskip
\textit{Step~1. For every (Borel) set $F$ contained 
in $F_0$ with $\mu(F)>0$ there exists a compact set 
$E \subset F$ with $\mu(E)>0$ which satisfies the 
assumption in \S\ref{s-spacex}.}
\\ 
\indent
Since the map $x\mapsto V(\mu,x)^\perp$, viewed as a 
closed-valued multifunction from $E$ to $\R^n$, is Borel 
measurable we can use Kuratowski and Ryll-Nardzewski's 
measurable selection theorem (see \cite{Sri}, Theorem~5.2.1) 
to choose Borel vectorfields $e_1,\dots, e_n$ on $\R^n$
so that
\begin{itemizeb} 
\item[(a)]
$e_1(x),\dots, e_n(x)$ form an orthonormal basis of 
$\R^n$ for every $x\in\R^n$; 
\item[(b)]
$\smash{e_1(x), \dots, e_{d(x)}(x)}$ span $\smash{V(\mu,x)^\perp}$ 
for every $x\in F$.
\end{itemizeb}

Then we use Lusin's theorem to find a compact set 
$E\subset F$ with $\mu(E)>0$ such that the restrictions 
of the function $d$ and the vectorfields $e_j$ to $E$
are continuous; thus $d$ is locally constant on $E$, 
and possibly replacing $E$ with a smaller subset we 
can further assume that $d$ is constant on $E$ and 
that the restrictions of each $e_j$ to $E$ 
takes values in the ball $B_j:=B(e_j(\bar x),\delta)$
for some $\bar x\in E$ and some (small) $\delta>0$. 

To conclude the proof we modify the vectorfields 
$e_j$ in the complement of $E$ so to they become 
continuous on the whole $\R^n$ and still satisfy
assumption~(a) above. 
This last step is achieved by first extending the 
restriction of each $e_j$ to $E$ to a continuous 
map from $\R^n$ to $B_j$ (using Tietze extension 
theorem) and then applying the Gram-Schmidt 
orthonormalization process to the resulting 
vectorfields (note that if $\delta$ is small enough 
these vectorfields are linearly independent at 
every point).

\medskip
\textit{Step~2. There exist a countable collection
$(E_i)$ of pairwise disjoint compact sets that
satisfy the assumption in \S\ref{s-spacex}
and $\mu(E_i)>0$, and their union contains 
$\mu$-a.e.~point.}
\\
\indent
Let $\G$ be the class of all countable collections
$(E_i)$ that satisfy all requirements with the possible
exception of the last one (the union contains
$\mu$-a.e.~point). The class $\G$ is nonempty 
and admits an element which is maximal with respect to
inclusion. Using Step~1 it is easy to prove that this
maximal element must also satisfy the last requirement.

\medskip
\textit{Step~3. Construction of the functions $f_i$ and $g_i$.}
\\
\indent
We take $(E_i)$ as in Step~2, and for every $i$ 
we use Proposition~\ref{s-nondiff2} to find 
a Lipschitz function $f_i$ with $\Lip(f_i) \le 1$
such that for $\mu$-a.e.~$x\in E_i$ 
\begin{equation}
\label{e-5.19}
D^+_v f_i(x) - D^-_v f_i(x) >0
\quad\text{for every $v\notin V(\mu,x)$.}
\end{equation}

Next we apply Lemma~\ref{s-regularize} to each 
$f_i$ to find a Lipschitz function $g_i$ with 
$\Lip(g_i) \le 2$ which agrees with $f_i$ on $E_i$, 
is smooth on $\R^n\setminus E_i$, and satisfies
\[
|g_i(x)-f_i(x)| \le \big(\dist(x,E_i)\big)^2
\quad\text{for every $x\in\R^n$.}
\]
This implies in particular that for every 
$x\in E_i$ and every $v\in\R^n$ there holds
\[
g_i(x+hv)=f_i(x+hv)+O(|h|^2)
\quad\text{for every $h\in\R$,}
\]
which yields $D^\pm_v g_i(x) = D^\pm_v f_i(x)$;
then \eqref{e-5.19} implies that
for $\mu$-a.e.\ $x\in E_i$ 
\begin{equation}
\label{e-5.20}
D^+_v g_i(x) - D^-_v g_i(x) >0
\quad\text{for every $v\notin V(\mu,x)$.}
\end{equation}

\medskip
\textit{Step~4. Construction of the function $f$.}
\\
\indent
We finally set
%
%
\[
f(x) := \sum_{i} \frac{g_i(x)}{2^i} 
\quad\text{for every $x\in\R^n$.}
\]
The function $f$ is clearly Lipschitz with $\Lip(f)\le 4$
(because $\Lip(g_i)\le 2$ for every $i$), and we claim that 
for $\mu$-a.e.~$x$ there holds 
$D^+_v f_i(x) - D^-_v f_i(x) >0$ for every $v\notin V(\mu,x)$.
Taking into account \eqref{e-5.20} and the fact that the 
union of the sets $E_i$ contains $\mu$-a.e.~$x$, it suffices 
to prove that for every $i$ and every $x\in E_i$ the function 
\[
\hat g_i := \sum_{j\ne i} \frac{g_j(x)}{2^j}
\]
is differentiable at $x$. 
If the sum is finite this is an immediate consequence 
of the fact that the functions $g_j$ are differentiable 
at $x$ for every $j\ne i$.%
\footnoteb{Recall that $g_j$ is smooth out of the set $E_j$, 
which does not intersect $E_i$, and then $x\notin E_j$.}
If the sum is infinite, 
this is an immediate consequence of the fact that for 
every $\delta>0$ the function $\hat g_i$ can be 
written as $\hat g_i=g_i'+g_i''$ where 
$g_i'$ differentiable at $x$ and $\Lip(g_i'')\le\delta$.%
\footnoteb{Let $g'_i$ and $g''_i$ be the sums 
of $g_j$ over all $j\ne i$ such that $j\le j_0$ and $j> j_0$,
respectively. For $j_0$ large enough there holds
$\Lip(g''_i) \le \sum_{j>j_0} \Lip(g_i)/2^j \le \delta$.}
\end{proof}

%
%

\section{Measures related to normal currents}
\label{s3}
In the main result of this section (Theorem~\ref{s-curr2}) 
we establish a connection between the decomposability 
bundle of a measure $\mu$ and the Radon-Nikod\'ym density 
of a normal current \wrt to $\mu$. 
Then we consider a few well-known formulas related to normal 
currents and smooth functions (or forms), and use 
the previous result to extend these formulas to the case 
of Lipschitz functions (or forms).
More precisely, we prove formulas for the action 
of the boundary of a normal current on a Lipschitz 
form (Proposition~\ref{s-boundary3}),
for the boundary of the interior product of a normal 
current and a Lipschitz form (Proposition~~\ref{s-boundary2}).
and for the push-forward of a normal current according 
to a Lipschitz map (Proposition~\ref{s-pushf2}).

\begin{parag}[Notation related to currents]
\label{s-not2}
We list here the notation from multilinear algebra and 
the theory of currents that is used in this section:
\begin{itemizeb}\leftskip 0.8 cm\labelsep=.3 cm
\item[$\largewedge_k(V)$]
space of $k$-vectors in the vector space $V$;
\item[$\largewedge^k(V)$]
space of $k$-covectors on the vector space $V$;%
\footnoteb{When $V=\R^n$ we endow $\largewedgef_k(V)$ 
and $\largewedgef^k(V)$ with the euclidean norms
determined by the canonical basis of these spaces.
Note that the choice of the norms on these 
spaces has no relevant consequence in what follows.}
\item[$v\wedge w$]
exterior product of the multi-vectors 
(or multi-covectors) $v$ and $w$;
\item[$\scal{\alpha}{v}$]
duality pairing of the $k$-covector $\alpha$ and the $k$-vector $v$, 
also written as $\scal{v}{\alpha}$;
\item[$v\trace\alpha$]
interior product of the $k$-vector $v$ and the 
$h$-covector $\alpha$ (\S\ref{s-intprod});
\item[$\scal{T}{\omega}$]
duality pairing of the $k$-current $T$ and the $k$-form 
$\omega$ (\S\ref{s-normcurr});
\item[$T\trace\omega$]
interior product of the $k$-current $T$ and the 
$h$-form $\omega$ (\S\ref{s-intprod});
\item[$d\omega$]
exterior derivative of the $k$-form $\omega$;
\item[$d_T\omega$]
exterior derivative of the $k$-form $\omega$ \wrt
the current $T$ (\S\ref{s-tandiff});
\item[$\bd T$]
boundary of the current $T$ (\S\ref{s-normcurr});
\item[$\Mass(T)$]
mass of the current $T$ (\S\ref{s-normcurr});

\item[{$[E,\tau,m]$}]
current associated to a rectifiable set $E$, 
an orientation $\tau$, and a multiplicity
$m$ (\S\ref{s-intcurr});

\item[$\Span(v)$]
span of the $k$-vector $v$ (\S\ref{s-span});
\item[$f^\# \omega$]
pull-back of the form $\omega$ according to
the map $f$ (\S\ref{s-pullb}).
\item[$\smash{ f_T^\# \omega }$]
restriction of $f^\# \omega$ to the tangent
bundle of $T$ (\S\ref{s-pullb}).
\item[$f_\# \, T$]
push-forward of the current $T$ according to
the map $f$ (\S\ref{s-pushf}).
\end{itemizeb}
\end{parag}

\begin{parag}[Currents and normal currents]
\label{s-normcurr}
We recall here the basic notions and terminology from
the theory of currents; introductory presentations 
of this theory can be found for instance
in \cite{KP}, \cite{Mo}, and \cite{Si}; 
the most complete reference remains \cite{Fe}.
A \emph{$k$-dimensional current} (or $k$-current) 
in $\R^n$ is a continuous linear functional on the 
space of $k$-forms on $\R^n$ which are smooth and 
compactly supported. 

The boundary of a $k$-current $T$ is the $(k-1)$-current 
$\bd T$ defined by $\scal{\bd T}{\omega} := \scal{T}{d\omega}$
for every smooth and compactly supported 
$(k-1)$-form $\omega$ on $\R^n$.
The \emph{mass} of $T$, denoted by 
$\Mass(T)$, is the supremum of $\scal{T}{\omega}$ over
all forms $\omega$ such that $|\omega|\le 1$
everywhere.%
\footnoteb{In the context of this paper the choice of 
the norm on the space of covectors, whether the Euclidean 
one or the co-mass, is not relevant.}
A current $T$ is called \emph{normal} if both $T$ 
and $\bd T$ have finite mass.
\end{parag}

\begin{parag}[Representation of currents with finite mass]
\label{s-repcurr}
By Riesz theorem a current $T$ with finite mass
can be represented as a finite measure with 
values in the space $\largewedge_k(\R^n)$
of $k$-vectors in $\R^n$,
and therefore it can be written in the form
$T=\tau\mu$ where $\mu$ is a finite positive measure
and $\tau$ is a $k$-vectorfield such that 
$\int |\tau|d\mu < +\infty$.
In particular the action of $T$ on a form
$\omega$ is given by
\[
\scal{T}{\omega} 
= \int_{\R^n} \scal{\omega(x)}{\tau(x)} \, d\mu(x)
\, ,
\]
and the mass $\Mass(T)$ is the total mass of $T$ as a 
measure, that is, $\Mass(T)=\int |\tau| d\mu$.%

In the following, whenever we write a current $T$ as $T=\tau\mu$ 
we tacitly assume that $\tau(x)\ne 0$ for $\mu$-a.e.~$x$;
in this case we say that $\mu$ is a measure
\emph{associated} to the current $T$.%
\footnoteb{Note that the measure $\mu$ and the 
$k$-vectorfield $\tau$ in the decomposition $T=\tau\mu$
are unique under the additional assumption that
$|\tau(x)|=1$ for $\mu$~a.e.$x$.}

Moreover, if $T$ is a $k$-current with finite mass 
and $\mu$ is an arbitrary measure, we can
write $T$ as $T=\tau\mu+\nu$ where $\tau$
is  $k$-vectorfield $\tau$ in $L^1(\mu)$, 
called the Radon-Nikod\'ym density of $T$ 
\wrt $\mu$, and $\nu$ is a measure 
with valued in $k$-vectors which is singular 
\wrt~$\mu$.
\end{parag}

\begin{parag}[Integral currents]
\label{s-intcurr}
Let $E$ be a $k$-rectifiable set.
An \emph{orientation} of $E$ is a $k$-vectorfield $\tau$ 
on $\R^n$ such that the $k$-vector $\tau(x)$ is \emph{simple},  
has norm $1$, and spans the approximate tangent 
space $\Tan(E,x)$ for $\Haus^k$-a.e.~$x\in E$.%
\footnoteb{The span of a simple $k$-vector 
$v=v_1\wedge\dots\wedge v_k$ in $\R^n$ is the 
linear subspace of $\R^n$ generated by the factors 
$v_1,\dots, v_k$, see also \S\ref{s-span}.}
A \emph{multiplicity} on $E$ is any integer-valued 
function $m$ such that $\int_E m \, d\Haus^k < +\infty$.
For every choice of $E, \tau, m$ as above we
denote by $[E,\tau,m]$ the $k$-current 
defined by $[E,\tau,m] := m\tau 1_E \, \Haus^k $, that is, 
\[
\bigscal{[E,\tau,m]}{\omega}
:= \int_E \scal{\omega}{\tau} \, m \, d\Haus^k
\, .
\]
Currents of this type are called 
\emph{integer-multiplicity rectifiable currents}.
A current $T$ is called \emph{integral} if both $T$ and $\bd T$ 
can be represented as integer-multiplicity rectifiable currents.
\end{parag}

The next statement contains a decomposition for 
normal $1$-currents which is strictly related to 
a decomposition given in \cite{Smirnov}.

\begin{theorem}
\label{s-decompcurr}
Let $T=\tau\mu$ be a normal $1$-current
with $|\tau(x)|=1$ for $\mu$-a.e.~$x$.
Then there exists a family of integral $1$-currents
$\big\{ T_t:=[E_t,\tau_t,1] : \ t\in I \big\}$, 
such that
\begin{itemizeb}
\item[(i)]
$T$ can be decomposed as $T=\int_I T_t\, dt$
(in the sense of \S\ref{s-measint}) and
\begin{equation}
\Mass(T) 
= \int_I \Mass(T_t) \, dt
= \int_I \Haus^1(E_t) \, dt
  \, ; 
\label{e-7.2}
\end{equation}
\item[(ii)]
$\tau_t(x)=\tau(x)$ for $\Haus^1$-a.e.~$x\in E_t$ and for a.e.~$t$;
\item[(iii)]
$\mu$ is decomposed as $\smash{\mu=\int_I \mu_t\, dt}$
(in the sense of \S\ref{s-measint})
where each $\mu_t$ is the restriction of $\Haus^1$
to the set $E_t$.
\end{itemizeb}
\end{theorem}

\begin{proof}
The existence of a family $\{T_t\}$ satisfying the decomposition 
in statement (i) and \eqref{e-7.2} can be found for instance 
in \cite{PaS}, Corollary~3.3.

To prove statement~(ii), we integrate the vectorfield $\tau$ 
against $T$, viewed as a vector measure, and using the 
decomposition of $T$ we obtain
\begin{align*}
  \Mass(T) 
  = \int_{\R^n} 1 \, d\mu(x)
& = \int_{\R^n} \scal{\tau(x)}{ dT(x) } \\
& = \int_I \bigg[ \int_{\R^n} \scal{\tau(x)}{ dT_t(x) } \bigg] \, dt \\
& = \int_I \bigg[ \int_{E_t} \scal{\tau(x)}{\tau_t(x)} \, d\Haus^1(x) \bigg] \, dt \\
& \le \int_I  \Haus^1(E_t) \, dt
  = \int_I  \Mass(T_t) \, dt
\end{align*}
where the inequality follows from the fact that $\tau(x)$ and $\tau_t(x)$ 
are unit vectors. Now \eqref{e-7.2} implies that this inequality
is actually an equality, which means that the vectors $\tau(x)$ and $\tau_t(x)$
agree for $\Haus^1$-a.e.~$x\in E_t$ and a.e.~$t$.

Finally, the identity of scalar measures $\mu=\int_I \mu_t \, dt$ in statement~(iii)
is obtained by multiplying the identity of vector measures $T=\int_I T_t \, dt$
by the vectorfield~$\tau$.
\end{proof}

A consequence of Theorem~\ref{s-decompcurr}
is the following.

\begin{proposition}\label{s-curr1}
Let $\mu$ be a positive measure and let $\tau$ be the 
Radon-Nikod\'ym density of a $1$-dimensional normal 
current $T$ \wrt $\mu$.
Then 
\begin{equation}
\label{e-curr1}
\Span(\tau(x)) \subset V(\mu,x) 
\quad\text{for $\mu$-a.e.~$x$.}
\end{equation}
\end{proposition}

\begin{proof} 
We write $T$ in the form 
$T=\tau'\mu'$ with $|\tau'(x)|=1$ for $\mu'$-a.e.~$x$,
and consider the decomposition $\mu'=\int \mu_t \, dt$ 
given in Theorem~\ref{s-decompcurr}:
for $\mu_t$-a.e.~$x$ and a.e.~$t$ we have that
$\Span(\tau'(x))$ agrees with $\Tan(E_t,x)$
which in turn is contained in $V(\mu',x)$
(by the definition of decomposability bundle), 
and this means that
\begin{equation}
\Span(\tau'(x)) \subset V(\mu',x) 
\quad\text{for $\mu'$-a.e.~$x$.}
\label{e-curr1.1}
\end{equation}
Let $E$ be the set of all $x$ such that $\tau(x)\ne 0$. 
Thus $1_E\mu\ll\mu'$, and therefore Proposition~\ref{s-basic}(i)
yields $V(\mu',x) = V(\mu,x)$ for $\mu$-a.e.~$x\in E$.
Moreover $\tau'(x)=\tau(x)/|\tau(x)|$
for $\mu$-a.e.~$x\in E$.
These facts together with \eqref{e-curr1.1} yield
that $\Span(\tau(x)) \subset V(\mu,x)$
for $\mu$-a.e.~$x\in E$, and since this inclusion
is trivially true for $x\notin E$, the proof
of \eqref{e-curr1} is complete.
\end{proof}

In order to extend Proposition~\ref{s-curr1} 
to currents with arbitrary dimension, we need 
some additional notions.

\begin{parag}[Interior product]
\label{s-intprod}
Let $h,k$ be integers with $0\le h\le k$.
Given a $k$-vector $v$ and an $h$-covector
$\alpha$ on $V$, the 
\emph{interior product} $v\trace\alpha$
is the $(k-h)$-vector uniquely defined 
by the duality pairing
\[
\scal{v\trace\alpha}{\beta} = \scal{v}{\alpha\wedge\beta}
\quad\text{for every $\beta\in\largewedge^{k-h}(V)$.}
\]
Accordingly, given a $k$-current $T$ in $\R^n$
and a smooth $h$-form $\omega$ on $\R^n$, 
the \emph{interior product} $T\trace\omega$ 
is the $(k-h)$-current defined by
\begin{equation}
\label{e-intprod}
\scal{T\trace\omega}{\sigma} = \scal{T}{\omega\wedge\sigma}
\end{equation}
for every smooth $(h-k)$-form $\sigma$ with compact support 
on $\R^n$. In this case a simple computation gives%
\,\footnoteb{Start from the definition of boundary
and apply identity \eqref{e-intprod} and then formula 
$d(\omega\wedge\sigma)= d\omega\wedge\sigma 
+ (-1)^h \omega\wedge d\sigma$.}
\begin{equation}
\label{e-intprodbdry}
\bd(T\trace\omega) 
= (-1)^h \big[ (\bd T) \trace\omega - T \trace d\omega \big]
\, .
\end{equation}

Note that if $T$ has finite mass and $\omega$ is bounded 
and continuous then formula \eqref{e-intprod} still makes 
sense, $T\trace\omega$ is a current with finite mass,
and given a representation $T=\tau\mu$ there holds
$T\trace\omega = (\tau\trace\omega)\,\mu$.
Along the same line, if $T$ is a normal current 
and $\omega$ is of class $C^1$, bounded and with bounded 
derivative, then $T\trace\omega$ is a normal current and
formula \eqref{e-intprodbdry} holds.
\end{parag}

\begin{parag}[Span of a $k$-vector and 
tangent bundle of a current]
\label{s-span}
Given vector space $V$ and a $k$-vector $v$ in $V$, 
we denote by $\Span(v)$ the smallest linear subspace 
$W$ of $V$ such that 
$v$ belongs to $\largewedge_k(W)$.%
\footnoteb{If $W$ is a linear subspace of $V$ 
then every $k$ vector on $W$ is canonically identified
with a $k$-vector on $V$ (and accordingly every 
$k$-covector on $V$ defines by restriction a 
$k$-covector on $W$).
Assuming this identification we have that
$\largewedgef_k(W) \cap \largewedgef_k(W') 
= \largewedgef_k(W\cap W')$
for every $W,W'$ subspaces of $V$, and therefore
the definition of $\Span(v)$ is well-posed.}

If $T=\tau\mu$ is a $k$-current with finite mass, 
we call $x\mapsto \Span(\tau(x))$ the \emph{tangent bundle}
of $T$. Note that $\Span(\tau(x))$ does not really depend
of the particular decomposition $T=\tau\mu$ we consider, 
in the sense that given another decomposition $T=\tau'\mu'$
then $\mu$ and $\mu'$ are absolutely continuous \wrt each 
other, and $\Span(\tau(x))=\Span(\tau'(x))$ for $\mu$-a.e.$x$.
\end{parag}

A few relevant properties of the span are given in 
the statement below, the proof of which is 
postponed to Section~\ref{s7}.

\begin{proposition}
\label{s-spancar}
Taken $v$ and $\Span(v)$ as above, we have that 
\begin{itemizeb}
\item[(i)]
if $v=0$ then $\Span(v)=\{0\}$;
\item[(ii)]
if $v\ne 0$ then $\Span(v)$ has dimension at least $k$;
\item[(iii)]
if $\Span(v)$ has dimension $k$ then $v$ is simple, that is, 
it can be written as $v=v_1\wedge\dots\wedge v_k$ with 
$v_1,\dots,v_k\in V$;
\item[(iv)]
if $v\ne 0$  and $v=v_1\wedge\dots\wedge v_k$ 
with $v_1,\dots,v_k\in V$ then $\Span(v)$ is the linear subspace 
of $V$ generated by the vectors $v_1,\dots, v_k$;
\item[(v)]
$\Span(v)$ consists of all vectors of the form 
$v\trace\alpha$ with $\alpha\in \largewedge^{k-1}(V)$.
\end{itemizeb}
\end{proposition}

\begin{theorem}
\label{s-curr2}
Let $\mu$ be a positive measure and let 
$\tau$ be the Radon-Nikod\'ym density of a $k$-dimensional 
normal current $T$ \wrt $\mu$.
Then $\Span(\tau(x))$ is contained in $V(\mu,x)$ for $\mu$-a.e.~$x$.
\end{theorem}

\begin{remark}
Note that $V(\mu,x)$ has dimension at least $k$ for $\mu$-a.e.~$x$ 
such that $\tau(x)\ne 0$ (cf.~Proposition~\ref{s-spancar}(ii)), 
and that every Lipschitz function $f$ on $\R^n$ is differentiable 
\wrt $\Span(\tau(x))$ at $\mu$-a.e.~$x$, that is, $\Span(\tau(x)) 
\in \D(f,x)$ for $\mu$-a.e.~$x$.
\end{remark}

\begin{proof}
For every $\alpha\in \largewedge^{k-1}(\R^n)$,  
$T\trace\alpha$ is a normal $1$-current whose 
Radon-Nikod\'ym density \wrt $\mu$ is $\tau\trace\alpha$
(see \S\ref{s-intprod}), and therefore
the vector $\tau(x) \trace \alpha$ belongs to
$V(\mu,x)$ for $\mu$-a.e.~$x$ (Proposition~\ref{s-curr1}).
In particular, taken a finite set $\{\alpha_j: \, j\in J\}$ 
that spans $\largewedge^{k-1}(\R^n)$, for $\mu$-a.e.~$x$ 
there holds
\begin{equation}
\label{e-3.6}
\tau(x) \trace \alpha_j \in V(\mu,x)
\quad\text{for every $j\in J$.}
\end{equation}
Moreover the vectors $\tau(x) \trace \alpha_j$ span 
$\{\tau(x) \trace \alpha: \, \alpha\in \largewedge^{k-1}(\R^n)\}$, 
which by Proposition~\ref{s-spancar}(v) agrees with $\Span(\tau(x))$.
This fact and \eqref{e-3.6} imply the claim.
\end{proof}

\medskip
In the rest of this section we give some applications 
of Theorem~\ref{s-curr2}.
We begin with a simple remark.

\begin{parag}[Exterior derivative of Lipschitz forms]
\label{s-tandiff}
Let $\mu$ be a positive measure on $\R^n$
and $\omega$ a Lipschitz $h$-form on $\R^n$. 
Then the (pointwise) exterior derivative 
$d\omega(x)$ is defined at $\Leb^n$-a.e.~$x$
but in general not at $\mu$-a.e.~$x$.
However, since the coefficients of $\omega$ \wrt any 
basis of $\largewedge^h(\R^n)$ are Lipschitz functions, 
they are differentiable \wrt $V(\mu,x)$ at $\mu$-a.e.~$x$, 
and therefore it is possible to define the exterior 
derivative of $\omega$ relative to $V(\mu,x)$ at 
$\mu$-a.e.~$x$, which we denote by $d_\mu\omega(x)$. 

The precise construction is the following:
given a basis $\{\alpha_i\}$ of
$\largewedge^h(\R^n)$, we denote by $\omega_i$ 
the coefficients of $\omega$ \wrt this basis, 
so that $\omega(x) = \sum_i \omega_i(x) \, \alpha_i$
for every $x\in\R^n$.
Then, given a point $x$ such that the functions 
$\omega_i$ are all differentiable at $x$ 
\wrt to $V=V(\mu,x)$, we chose a basis $\{e_j\}$ 
of $V$ and denote by $d_\mu\omega(x)$
the $(h+1)$-covector on $V$
defined by
\[
d_\mu\omega(x) 
:= \sum_{i,j} D_{e_j} \omega_i (x) \, e^*_j\wedge \alpha_i
\]
where $\{e^*_j\}$ is the dual basis associated to
$\{e_j\}$.%
\footnoteb{\label{f-3.1}
That is, the basis of the dual space $V^*$ defined by
$\scal{e^*_i}{e_j} = \delta_{ij}$ for every 
$1\le i,j \le n$, where $\delta_{ij}:=1$ if $i=j$ 
and $\delta_{ij}:=0$ if $i\ne j$, as usual.}

Assume now that $T=\tau\mu$ is a normal $k$-current on $\R^n$.
By Theorem~\ref{s-curr2}, $\Span(\tau(x))$ is contained
in $V(\mu,x)$ for $\mu$-a.e.~$x$, and therefore we can 
define the exterior derivative of $\omega$ \wrt $\Span(\tau(x))$
at $\mu$-a.e.~$x$, which we denote by $d_T\omega(x)$.%
\footnoteb{That is, $d_T\omega(x)$ is the $(h+1)$-covector 
on $\Span(\tau(x))$ given by the restriction of $d_\mu\omega(x)$.}
Note that $d_T\omega(x)$ is actually independent of the 
specific decomposition of $T=\tau\mu$, because so is 
$\Span(\tau(x))$ (see~\S\ref{s-span}).
\end{parag}

Now we turn our attention to the identity that defines 
the boundary of a $k$-current $T$, namely 
$\scal{\bd T}{\omega} = \scal{T}{d\omega}$ for every 
smooth $(k-1)$-form $\omega$ with compact support.
If $T$ is a normal current then both terms in this 
identity can be represented as integrals; therefore they 
make sense even when $\omega$ is a form of class $C^1$ with 
$\omega$ and $d\omega$ bounded, and a simple 
approximation argument proves that they agree.

The next result shows that the same is true for 
Lipschitz forms, having made the necessary changes.

\begin{proposition}
\label{s-boundary3}
Let $T=\tau\mu$ be a normal $k$-current on $\R^n$, 
$\omega$ a bounded Lipschitz $(k-1)$-form on $\R^n$.
Then 
\begin{equation}
\label{e-boundary3}
\scal{\bd T}{\omega}= \int_{\R^n}\scal{d_T\omega(x)}{\tau(x)} d\mu(x)
\, ,
\end{equation}
where $d_T\omega$ is taken as in \S\ref{s-tandiff}.%
\footnoteb{Note that, for $\mu$-a.e.~$x$, $d_T\omega(x)$ 
is a $k$-covector on $\Span(\tau(x))$ and therefore the 
duality pairing $\scal{d_T\omega(x)}{\tau(x)}$ 
is well-defined.}
\end{proposition}

For the proof we need the following approximation
lemma, the proof of which is postponed 
to Section~\ref{s7}.

\begin{lemma}
\label{s-approx}
Let $f$ be a Lipschitz function on $\R^n$, 
$\mu$ a measure on $\R^n$,
and $x\mapsto V(x)$ a Borel map from $\R^n$ to $\Gr(\R^n)$
such that $V(x)$ belongs to the differentiability bundle
$\D(f,x)$ for $\mu$-a.e.~$x$ (see~\S\ref{s-diffbund}).

Then there exists a sequence 
of smooth functions $f_j:\R^n\to\R$ 
such that the following statements 
hold (as $j \to +\infty$):
\begin{itemizeb}
\item[(i)]
the functions $f_j$ converge to $f$ uniformly;
\item[(ii)]
$\Lip(f_j)$ converge to $\Lip(f)$;
\item[(iii)]
$\dV f_j(x) \to \dV f(x)$ for $\mu$-a.e.~$x$.%
\footnoteb{Here $\dV f_j(x)$ is the restriction 
of the linear function $df_j(x)$ to the subspace $V(x)$; 
and convergence is intended in the sense 
of the operator norm for linear functions on $V$.}
\end{itemizeb}
\end{lemma}

\begin{proof}[Proof of Proposition~\ref{s-boundary3}]
We apply Lemma~\ref{s-approx} with $V(x):=\Span(\tau(x))$ 
to the coefficients of $\omega$ \wrt some basis of 
$\largewedge^{k-1}(\R^n)$ and construct a sequence of smooth 
$(k-1)$-forms $\omega_j$ which are uniformly bounded,
have uniformly bounded derivatives $d\omega_j$, 
converge to $\omega$ uniformly, 
and satisfy 
\begin{equation}
\label{e-3.8}
\lim_{j\to+\infty} d_T\omega_j(x) = d_T\omega(x) 
\quad\text{for $\mu$-a.e.~$x$.}
\end{equation}
Then%
\,\footnoteb{The first equality follows from the fact
that $\omega_j$ converge to $\omega$ uniformly, 
the fourth one from \eqref{e-3.8} and Lebesgue's 
dominated convergence theorem using the domination
$| \scal{d_T\omega_j(x)}{\tau(x)} | 
\le |d\omega_j(x)| \, |\tau(x)| \le m|\tau(x)|$ 
(recall that the forms $d\omega_j$ are uniformly bounded
by some constant $m$ and $\tau$ belongs to $L^1(\mu)$).}
\begin{align*}
  \scal{\bd T}{\omega} 
  =\lim_{j\to\infty}\scal{\bd T}{\omega_j}
& =\lim_{j\to\infty}\scal{T}{d\omega_j} \\
& =\lim_{j\to\infty}\int_{\R^n}
      \scal{d_T\omega_j(x)}{\tau(x)} d\mu(x) \displaybreak[1] \\
& =\int_{\R^n}\scal{d_T\omega(x)}{\tau(x)} d\mu(x)
  \, .
  \qedhere
\end{align*}
\end{proof}

\medskip
Next we consider the interior product $T\trace\omega$
of a $k$-current and a bounded Lipschitz $h$-form, 
and prove a variant of formula \eqref{e-intprodbdry}
for the boundary of $T\trace\omega$.

\begin{proposition}
\label{s-boundary2}
Let $T=\tau\mu$ be a normal $k$-current on $\R^n$ 
and $\omega$ a bounded Lipschitz $h$-form on $\R^n$
with $0\le h < k$.
Then $T\trace\omega = (\tau\trace\omega)\,\mu$ 
is a normal $(k-h)$-current with boundary
\begin{equation}
\label{e-prodbdry2}
\bd(T\trace\omega) 
= (-1)^h \big[ (\bd T)\trace\omega 
     - (\tau\trace d_T \omega) \, \mu \big]
\, ,
\end{equation}
where $d_T \omega$ is taken as in \S\ref{s-tandiff}.%
\footnoteb{For $\mu$-a.e.~$x$, $d_T\omega(x)$ 
is a $k$-covector on $\Span(\tau(x))$ and therefore the 
interior product $\tau(x) \trace d_T\omega(x)$ 
is a well-defined $(k-h-1)$-vector in $\Span(\tau(x))$, 
and therefore also an $(k-h-1)$-vector in $\R^n$.}
\end{proposition}

\begin{remark}
In the special case $h=0$ Proposition~\ref{s-boundary2}
can be restated as follows: if $T=\tau\mu$ is a normal 
$k$-current on $\R^n$ and $f$ a bounded Lipschitz function
on $\R^n$, then $fT = f\tau\mu$ is a normal $k$-current 
with boundary
\begin{equation}
\label{e-prodbdry}
\bd(fT) = f\, \bd T + (\tau\trace d_T f) \mu
\, .
\end{equation}
\end{remark}

\begin{proof}
We take a sequence of smooth forms $\omega_j$ exactly 
as in the proof of Proposition~\ref{s-boundary3}. Since 
each $\omega_j$ is smooth, we know that the currents 
$T\trace\omega_j$ are normal (cf.~\S\ref{s-intprod})
and it is easy to see that as $j\to+\infty$ they converge 
to $T\trace\omega$ in the mass norm. 
Moreover formula \eqref{e-intprodbdry} yields
\begin{equation}
\label{e-3.11}
\bd(T\trace\omega_j) 
= (-1)^h \big[ (\bd T) \trace\omega_j - T \trace d\omega_j \big]
\, ,
\end{equation}
which, together with the fact that the forms $\omega_j$ 
and the derivatives $d\omega_j$ are uniformly bounded, 
implies that the masses of $\bd(T\trace\omega_j)$ are 
also uniformly bounded.
Thus $T\trace\omega$ is a normal current.

To prove formula~\eqref{e-prodbdry2} we pass to the limit
in \eqref{e-3.11}, and the only delicate point is to show 
the convergence of $T \trace d\omega_j$ to $T \trace d_T\omega$.
To this end, we use that
\[
T \trace d\omega_j 
= (\tau\trace d\omega_j)\mu 
= (\tau\trace d_T\omega_j)\mu 
\, , \quad
T \trace d_T\omega 
= (\tau\trace d_T\omega)\mu 
\, , 
\]
and that the forms $d_T\omega_j$ are uniformly bounded and converge 
$\mu$-a.e.\ to $d_T\omega$ by assumption \eqref{e-3.8}.
\end{proof}

\medskip
We conclude this section by proving a formula for the 
push-forward of a normal current according to a 
Lipschitz map (Proposition~\ref{s-pushf2}).

\begin{parag}[Pull-back of a form]
\label{s-pullb}
Given a map $f:\R^n\to\R^m$ of class $C^1$ and 
a continuous $k$-form $\omega$ on $\R^m$, the 
pull-back of $\omega$ according to $f$ is the 
continuous $k$-form $f^\#\omega$ on $\R^m$ 
defined by
\begin{equation}
\label{e-3.12}
\bigscal{(f^\#\omega)(x)}{v_1\wedge\dots\wedge v_k}
:= \bigscal{\omega(f(x))}{(df(x)\,v_1)\wedge\dots\wedge (df(x)\,v_k)}
\end{equation}
for every $v_1,\dots,v_k\in\R^n$.

This definition clearly shows that when $f$ is a 
Lipschitz map we can only define $(f^\#\omega)(x)$ 
as a the $k$-covector on $\R^n$ at the points $x$ 
where $f$ is differentiable, that is, 
at $\Leb^n$-a.e.~$x$. But in general we cannot
define it at $\mu$-a.e.~$x$ when $\mu$ is an arbitrary 
measure on $\R^n$.
However, since $f$ is differentiable \wrt $V(\mu,x)$ 
at $\mu$-a.e.~$x$, we can still use formula \eqref{e-3.12} to
define the restriction of $(f^\#\omega)(x)$ to $V(\mu,x)$,
which is therefore a $k$-covector on $V(\mu,x)$
denoted by $\smash{ (f_\mu^\#\omega)(x) }$.

Finally, given a normal current $T=\tau\mu$ on $\R^n$, we use
that the tangent bundle $\Span(\tau(x))$ is contained in 
$V(\mu,x)$ (Theorem~\ref{s-curr2}) to define 
$\smash{ (f_T^\#\omega)(x) }$ as the 
restriction of $\smash{ (f_\mu^\#\omega)(x) }$ 
to $\Span(\tau(x))$ for $\mu$-a.e.~$x$.
\end{parag}

\begin{parag}[Push-forward of currents]
\label{s-pushf}
Given a smooth map $f:\R^n\to\R^m$ and 
a $k$-current $T$ in $\R^n$ with compact support, 
the \emph{push-forward} of $T$ according to $f$ is 
the $k$-current $f_\#T$ in $\R^m$ defined by
\begin{equation}
\label{e-pushf}
\scal{f_\#T}{\omega} := \scal{T}{f^\#\omega}
\end{equation}
for every smooth $k$-form $\omega$ on $\R^m$.%
\footnoteb{Since $T$ has compact support, 
$\scal{T}{\sigma}$ is well-defined for every 
smooth $k$-form $\sigma$ on $\R^n$, even without 
compact support, and in particular it is defined
for $\sigma:= f^\#\omega$ (see for instance \cite{Mo}, 
Section~4.3A).
The assumption that $T$ has compact support can be removed
if one assumes that $f$ is proper (see for instance 
\cite{KP}, Section~7.4.2.)}
If in addition $T$ has finite mass then identity 
\eqref{e-pushf} can be extended to all continuous
$k$-forms $\omega$ and can be used to define $f_\#T$ 
for all maps $f$ of class $C^1$.

When the map $f$ is only Lipschitz the right-hand 
side of formula \eqref{e-pushf} does not make sense 
because the form $f^\#\omega$ is not defined, 
not even $\mu$-a.e.
However, if $T$ is normal, it can be proved that, 
for every every sequence of smooth maps $f_j:\R^n\to\R^m$ 
that are uniformly Lipschitz and converge to $f$ uniformly, 
the push-forwards $(f_j)_\#T$ converge in the sense of currents
to the same normal current, which is then taken as definition of 
the push-forward $f_\#T$ (see \cite{Fe}, \S4.1.14, or 
\cite{KP}, Lemma~7.4.3).

In the next statement we show that a suitable modification 
of formula \eqref{e-pushf} holds even in this case.
\end{parag}

\begin{proposition}
\label{s-pushf2}
Let $T=\tau\mu$ be a normal $k$-current on $\R^n$ 
with compact support, let $f:\R^n\to\R^m$ be a Lipschitz map, 
and let $f_\#T$ be the push-forward of $T$ according to $f$. 
Then 
\begin{equation}
\label{e-pushf2}
\scal{f_\#T}{\omega} 
= \scal{T}{f_T^\#\omega}
=\int_{\R^n} \bigscal{(f_T^\#\omega)(x)}{\tau(x)} \, d\mu(x)
\end{equation}
for every continuous $k$-form $\omega$ on $\R^m$.%
\footnoteb{In this case $(f_T^\#\omega)(x)$ is a well-defined 
$k$-covector on $\Span(\tau(x))$ for $\mu$-a.e.~$x$, and therefore
the duality pairing $\bigscal{(f_T^\#\omega)(x)}{\tau(x)}$ makes 
sense.}
\end{proposition}

\begin{proof}
We use Lemma~\ref{s-approx} to choose the 
approximating maps $f_j$ used to define $f_\#T$
so that for $\mu$-a.e.~$x$ the linear maps 
$d_T f_j(x)$ converge to $d_T f(x)$.

Therefore, for every smooth $k$-form
$\omega$ on $\R^m$ with compact support
we have that $\smash{ f_j^\# \omega(x)}$, 
viewed as a $k$-covectors on $\Span(\tau(x))$, 
converge to $f_\mu^\# \omega(x)$
for $\mu$-a.e.~$x$ (cf.~\S\ref{s-pullb}), 
and in particular
\begin{equation}
\label{e-3.15}
\lim_{j\to+\infty} 
\bigscal{f_j^\# \omega(x)}{\tau(x)} 
= \bigscal{f_T^\# \omega(x)}{\tau(x)}  
\quad\text{for $\mu$-a.e.~$x$.}
\end{equation}
Hence%
\,\footnoteb{The first equality follows from the fact
that $(f_j)_\#T$ converge to $f_\#T$ in the sense of currents, 
the second one follows from \eqref{e-pushf}, 
the third one follows from \eqref{e-3.15} and Lebesgue's dominated 
convergence theorem using the domination
$\big| \bigscal{f_j^\# \omega(x)}{\tau(x)} \big| 
\le |df_j(x)|^k |\tau(x)| \le L^k |\tau(x)|$ where $L$ is
a constant such that $\Lip(f_j)\le L$ for every $j$ 
(recall that $\tau$ belongs to $L^1(\mu)$).}
\begin{align*}
  \scal{f_\#T}{\omega} 
& =\lim_{j\to+\infty} \scal{(f_j)_\#T}{\omega} 
   \displaybreak[1] \\
& =\lim_{j\to+\infty} 
   \int_{\R^n} \bigscal{(f_j^\# \omega)(x)}{\tau(x)} \, d\mu(x) 
   \\
& =\int_{\R^n} \bigscal{(f^\# \omega)(x)}{\tau(x)} \, d\mu(x) 
  \, .
\end{align*}
We have thus proved identity \eqref{e-pushf2} for every $\omega$ 
which is smooth and compactly supported, and we extend it to 
every continuous $\omega$ by a simple approximation argument.
\end{proof}

%
%

\section{A characterization of the decomposability bundle.}
\label{s6}
In this section we give a characterization of the 
decomposability bundle of a measure $\mu$ in terms of
normal $1$-currents (Theorem~\ref{s-6.3}), and more precisely
we show that $V(\mu,x)$ agrees for $\mu$-a.e.~$x$ with
the space $N(\mu,x)$ defined in the next subsection.
Building on this result we give a precise description of
the vectorfields $\tau$ on $\R^n$, $n\ge 2$, which can 
be obtained as the Radon-Nikod\'ym density of a $1$-dimensional 
normal current \wrt $\mu$ (Corollary~\ref{s-6.4}).

Through this section $\mu$ denotes a measure on $\R^n$, 
and $V(\mu,\cdot)$ is the corresponding decomposability bundle.

\begin{parag}[The auxiliary bundle $N(\mu,x)$]
\label{s-tanbund}
For every point $x$ in the support of $\mu$, 
we denote by $N(\mu,x)$ the set 
of all vectors $v\in\R^n$ such that there exists a normal 
$1$-current $T$ in $\R^n$ with $\bd T=0$ such that%
\,\footnoteb{In this section we view $1$-currents with 
finite mass on $\R^n$ as $\R^n$-valued measures; 
thus $|T-v\mu|$ denotes the total 
variation of the $\R^n$-valued measure $T-v\mu$.}
\begin{equation}
\label{e-tanbund}
\lim_{r\to 0} \frac{ |T-v\mu|(B(x,r)) }{ \mu(B(x,r)) }
= 0
\, .
\end{equation}
In the following we refer to condition \eqref{e-tanbund} 
by saying that $T$ is \emph{asymptotically equivalent} to 
$v\mu$ at the point $x$.

It is sometimes convenient to extend the definition of 
$N(\mu,x)$ to the points $x$ which are not in the 
support of $\mu$ by setting, for example, $N(\mu,x):=\{0\}$.
\end{parag}

The main results of this section are the following.

\begin{theorem}
\label{s-6.2}
Let $\tau$ be a Borel vectorfield on $\R^n$, $n\ge 2$, which belongs 
to $L^1(\mu)$ and satisfies 
$\tau(x) \in N(\mu,x)$ for $\mu$-a.e.~$x$.
Then there exists a normal $1$-current $T$ on $\R^n$ such that
\begin{itemizeb}
\item[(i)]
the Radon-Nikod\'ym density of $T$ \wrt $\mu$ agrees with $\tau$;
\item[(ii)]
$\bd T=0$ and and $\Mass(T) \le C\|\tau\|_{L^1(\mu)}$ where 
the $C$ depends only on $n$.
\end{itemizeb}
\end{theorem}

\begin{theorem}
\label{s-6.3}
For $n\ge 2$ there holds $V(\mu,x)=N(\mu,x)$ for $\mu$-a.e.~$x$.
\end{theorem}

Putting together Theorems~\ref{s-6.2} and \ref{s-6.3}
and Proposition~\ref{s-curr1} we obtain the following corollary.

\begin{corollary}
\label{s-6.4}
Let $\tau$ be a vectorfield on $\R^n$, $n\ge 2$, 
which belongs to $L^1(\mu)$.
Then $\tau$ can be written as the Radon-Nikod\'ym 
density of a $1$-dimensional normal current $T$ \wrt $\mu$
if and only if $\tau(x) \in V(\mu,x)$ for $\mu$-a.e.~$x$.
\end{corollary}

Before moving to proofs, we add some comments on the 
previous definition and results.

\begin{remark}
\label{s-6.5}
(i)~The set $N(\mu,x)$ is clearly a vector subspace 
of $\R^n$, and its definition is completely pointwise.
By contrast, the decomposability bundle $V(\mu,x)$ is only 
defined up to $\mu$-negligible subsets of $x$, and the linear
structure is imposed in the definition.

(ii)~If $\tau$ is the Radon-Nikod\'ym density of 
a normal $1$-current $T$ \wrt $\mu$, then $\tau(x)$ belongs 
to $N(\mu,x)$ for $\mu$-a.e.~$x$. More precisely, 
if we write $T=\tau\mu+\nu$ with $\nu$ singular \wrt $\mu$, 
then $\tau(x) \in N(\mu,x)$ for every $x$ where 
$\tau$ is $L^1(\mu)$-approximately continuous and 
the Radon-Nikod\'ym density of $|\nu|$ \wrt $\mu$ is $0$.

(iii)~We prove in Lemma~\ref{s-6.8} below that the 
map $x\mapsto N(\mu,x)$ agrees outside a suitable 
$\mu$-negligible Borel set with a Borel 
map from $\R^n$ to $\Gr(\R^n)$.
We actually believe that the map $x\mapsto N(\mu,x)$ itself
is Borel measurable, but the only proof we could devise 
is rather involved, and since this result is not strictly
needed in the following, we preferred to omit it.

(iv)~In dimension $n=1$, the only normal $1$-current $T$
with $\bd T=0$ is the trivial one, and therefore 
$N(\mu,x)=\{0\}$ for every $x$ and every $\mu$. 
Thus Theorem~\ref{s-6.3} does not hold, 
while Theorem~\ref{s-6.2} holds but is devoid of any meaning.
Corollary~\ref{s-6.4} does not hold either, 
because normal $1$-currents on $\R$ are of the form
$u\Leb^1$ where $u$ is a $BV$ function.

(v)~In dimension $n=2$ the bundle $N(\mu,\cdot)$ is closely related
to the bundle $E(\mu,\cdot)$ introduced in \cite{Alb}, 
Definition~2.1. More precisely $E(\mu,x)$ is the set of all 
vectors $v$ such that $v\mu$ is asymptotically equivalent at 
$x$ to the (distributional) gradient of a $BV$ function on $\R^2$.
Now, if $\lambda$ is a positive measure on $\R^2$ and $\tau$ is a 
vectorfield in $L^1(\lambda)$, then the vector measure $\tau\lambda$ 
is the gradient of a $BV$ function on $\R^2$ if and only if 
$\tau^\perp\lambda$ is a normal $1$-current without boundary
(here $v^\perp$ denotes the rotation of the vector $v$ 
by $90^\circ$ counterclockwise). This means that 
$N(\mu,x)$ is the set of all $v^\perp$ with $v\in E(\mu,x)$.

(vi)~Let $\mu$ be a singular measure on $\R^2$. 
The main result in \cite{Alb}, Theorem~3.1, states that 
in this case $E(\mu,x)$ has dimension at most $1$ for 
$\mu$-a.e.~$x$, and therefore, taking into account the 
previous remark and Theorem~\ref{s-6.3}, for $n=2$ 
we have that also $N(\mu,x)=V(\mu,x)$ has dimension 
at most $1$ for $\mu$-a.e.~$x$ (cf.~Remark~\ref{s-remrade}).

(vii)~There are many possible variants of the definition
of $N(\mu,x)$. Among these, the one given above 
imposes the strongest requirements on the elements of 
$N(\mu,x)$. Going to the opposite extreme, we may
consider the set $N'(\mu,x)$ of all $v\in\R^n$ 
such that there exist a sequence of positive real
numbers $r_j$ that converge to $0$ and a sequence of 
normal $1$-currents $T_j$ such that
\[
\lim_{j\to+\infty}
\frac{ |T_j-v\mu| (B(x,r_j)) }{ \mu(B(x,r_j)) }
=0
\, .
\]
Clearly $N'(\mu,x)$ contains $N(\mu,x)$ for every $x$, 
and it is easy to show that this inclusion may be strict. 
However, it should be possible to prove
that $N'(\mu,x) = V(\mu,x)$ for $\mu$-a.e.~$x$, 
which in view of Theorem~\ref{s-6.3} yields that 
$N'(\mu,x) = N(\mu,x)$ for $\mu$-a.e.~$x$
(we will not pursue this matter here).

(viii)~It is natural to ask if Corollary~\ref{s-6.4}
can be extended to normal $k$-currents with $k>1$.
More precisely, we know from Theorem~\ref{s-curr2} 
that if $\tau$ is the Radon-Nikod\'ym 
density of a normal $k$-current \wrt~$\mu$,
then $\Span(\tau(x)) \subset V(\mu,x)$ 
for $\mu$-a.e.~$x$, and we ask if the converse is true, 
that is, if every $k$-vectorfield $\tau$ in $L^1(\mu)$
that satisfies this inclusion can be obtained as 
the Radon-Nikod\'ym density of normal $k$-current \wrt~$\mu$.
This question is related to the issue raised 
in \S\ref{s-hddecomp}, and the answer is not clear.
\end{remark}

The rest of this section is devoted to the proofs of
Theorems~\ref{s-6.2} and \ref{s-6.3}.

Through these proofs we use the letter $C$ to denote every 
constant that depends only on the dimension $n$ (and the 
value may change at every occurrence).

For the proof of Theorem~\ref{s-6.2} we need the following lemmas. 

\begin{lemma}
\label{s-6.6}
Let $T$ be a normal $k$-current in $\R^n$, $n>k>0$, and 
let $B$ an open ball in $\R^n$ which does not intersect 
the support of $\bd T$. 
Then there exists a normal $k$-current $U$ in $\R^n$ such that
\begin{itemizeb}
\item[(i)]
the support of $U$ is contained in the closure $\barB$ of $B$; 
\item[(ii)]
the currents $U$ and $T$ agree on $B$, that is, $1_B \, U = 1_B \, T$;
\item[(iii)]
$\Mass(U) \le C \, |T|(\barB)$;
\item[(iv)]
$\bd U=0$.
\end{itemizeb}
\end{lemma}

\begin{proof}
First of all, we notice that it suffices to prove
the statement when $B$ is the open ball with center 
$0$ and radius $1$.

We begin with an outline of the construction of $U$.
We choose a point $x_0\in B$, and construct a 
retraction $p$ of $\R^n\setminus\{x_0\}$
onto $\R^n\setminus B$ as follows: 
for $x\notin B$ we let $p(x):=x$, and for $x\in B$ 
we let $p(x)$ be the point at the intersection 
of the sphere $\bd B$ and the half-line which starts
in $x_0$ and pass through $x$. Thus
\[
p(x):=x_0+tw
\quad\text{where}\quad
w:=\frac{x-x_0}{|x-x_0|}
\]
and $t>0$ is chosen so that $|p(x)|=1$, that is,
\[
\label{e-goofy2}
t:=\sqrt{ 1 + \scal{x_0}{w}^2 - |x_0|^2 } -\scal{x_0}{w}
\, .
\]

We then denote by $T'$ the push-forward of $T$ according to the map $p$, 
that is, $T'=p_\#T$.
Thus $T'=0$ in $B$ and $T'=T$ in $\R^n\setminus\barB$.
Moreover $\bd T'=p_\#(\bd T)$, and since $\bd T$ is supported in 
the complement of $B$, where $p$ agrees with the identity, 
we have that $\bd T'=\bd T$.

Finally we set $U:= T-T'$. Then $U=T$ in $B$, 
$U=0$ in $\R^n\setminus\barB$, and $\bd U=0$.

\smallskip
The main problem of this construction is that 
the map $p$ is discontinuous at $x_0$,  
and therefore the push-forward
$p_\#T$ is not well-defined.
Note that the same problem is met in the proof
of the polyhedral deformation theorem
(see for instance \cite{KP}, Section~7.7), 
and the solution given there works here as well. 
The idea is to approximate $p$ with a suitable
sequence of Lipschitz maps $p_\eps$ and define
$T'$ as the limit of the currents $(p_\eps)_\# T$,
and one of the key steps is the choice of the point $x_0$.

\medskip
\textit{Step~1. There exists $x_0\in B$ such that
\begin{equation}
\label{e-goofy3}
\int_\barB \frac{ d|T|(x) }{ |x-x_0|^k }
\le C \, |T|(\barB)
\, .
\end{equation}
}

We actually prove that the integral of the left-hand side 
of \eqref{e-goofy3} over all $x_0\in B$ (\wrt the Lebesgue 
measure) is bounded by $C\, |T|(\barB)$; this will imply 
that \eqref{e-goofy3} holds for a set of positive measure
of $x_0$. Indeed%
\,\footnoteb
{As usual, we write $dx_0$ for $d\Leb^n(x_0)$. 
For the first equality we apply Fubini's Theorem
and the change of variable $y=x-x_0$; for the first
inequality we use that the ball $B(-x_0,1)$ is contained
in $B(0,2)=B(2)$, the last equality follows from the
fact that $\smash{\int_{B(2)} dy/|y|^k}$ is finite
and does not depend on $x_0$.}
\begin{align*}
  \int_B \bigg[ 
      \int_\barB \frac{ d|T|(x) }{ |x-x_0|^k }
    \bigg] \, dx_0
& = \int_\barB \bigg[ 
      \int_{B(-x_0,1)} \frac{ dy }{ |y|^k }
    \bigg] \, d|T|(x) \\
& \le \int_\barB \bigg[ 
      \int_{B(2)} \frac{ dy }{ |y|^k }
    \bigg] \, d|T|(x)
  = C \, |T|(\barB)
  \, .
\end{align*}

\textit{Step~2. Construction of the approximating currents $T_\eps$.}
\\
\indent
Let $\varphi:[0,+\infty)\to [0,1]$ be a smooth function 
such that $\varphi(s)=0$ for $s\le 1/2$ and 
$\varphi(s)=1$ for $s\ge 1$.
Then for every $\eps>0$ we consider 
the map $p_\eps:\R^n\to\R^n$ given by 
\[
p_\eps(x) := \varphi\big(|x|/\eps\big) \, p(x)
\, .
\]
One easily checks that the map $p_\eps$ is Lipschitz,%
\footnoteb{The map $p_\eps$ is actually smooth in the 
complement of $\bd B$.}
and agrees with $p$ outside the ball $B(\eps)$.
Moreover, starting from the definition of $p$,
a straightforward computation gives 
\[
|dp(x)| \le C/|x-x_0|
\quad\text{for every $x\in\barB\setminus\{x_0\}$,}
\] 
and using this estimate we obtain 
\begin{equation}
\label{e-goofy4}
|dp_\eps(x)| \le \frac{C}{|x-x_0|} 
\quad\text{for every $x\in\barB$.}
\end{equation}
Finally, for every $\eps>0$ we set $T_\eps:=(p_\eps)_\# T$.

\medskip
\textit{Step~3. For every $\eps>0$, $T_\eps$ is a 
normal $k$-current. Moreover
\begin{itemizeb}
\item[(a)]
$T_\eps=T$ on $\R^n\setminus\barB$;
\item[(b)] 
$\Mass(1_\barB T_\eps) \le C \, |T|(\barB)$;
\item[(c)] 
$\Mass(1_B T_\eps) \to 0$ as $\eps\to 0$;
\item[(d)]
$\bd T_\eps=\bd T$.
\end{itemizeb}
}

Since the map $p_\eps:\R^n\to\R^n$ is Lipschitz and 
proper, and $T$ is a normal $k$-current, so is $T_\eps$.
Since $p_\eps$ maps $\barB$ into itself, and maps
$\R^n\setminus\barB$ into itself identically, then
\[
1_{\R^n\setminus\barB} T_\eps 
= (p_\eps)_\#(1_{\R^n\setminus\barB} T) = 1_{\R^n\setminus\barB} T
\, , \quad
1_\barB T_\eps 
= (p_\eps)_\#(1_\barB T) 
\, .
\]
The first identity amounts to statement~(a), 
while using the second identity we obtain statement~(b):%
\footnoteb{The first inequality follows by a standard estimate
for the mass of the push-forward, the second one follows from
\eqref{e-goofy4}, and the third one from \eqref{e-goofy3}.}
\begin{align*}
  \Mass(1_\barB T_\eps) 
& = \Mass\big( (p_\eps)_\#(1_\barB T) \big) \\
& \le \int_\barB |dp_\eps|^k d|T| 
  \le C \int_\barB \frac{ d|T|(x) }{ |x-x_0|^k }
  \le C\, |T|(\barB)
  \, .
\end{align*}

To prove statement~(c) note that $p_\eps$ maps 
$\R^n\setminus B(\eps)$ into $\R^n\setminus B$; 
then arguing as before we get 
\[
\Mass(1_B T_\eps) 
\le \Mass\big( (p_\eps)_\#(1_{B(\eps)} T) \big) 
\le \int_{B(\eps)} |dp_\eps|^k d|T| 
\le C \int_{B(\eps)} \frac{ d|T|(x) }{ |x-x_0|^k }
\, , 
\]
and the last integral converges to $0$ as $\eps\to 0$ 
by the dominated convergence theorem (note that this 
integral is finite when $\eps=1$ by \eqref{e-goofy3}). 

Finally, statement~(d) follows by the identity
$\bd T_\eps =(p_\eps)_\#(\bd T)$
and the fact that the support of $\bd T$ is contained 
in $\R^n\setminus B$, and on this set $p_\eps$ agrees 
with the identity map. 

\medskip
\textit{Step~4. Possibly passing to a subsequence, 
the currents $T_\eps$ converge to some normal current 
$T'$ as $\eps\to 0$.
Moreover the following statements hold:
\begin{itemizeb}
\item[(e)]
$T'=T$ on $\R^n\setminus\barB$;
\item[(f)]
$T'=0$ on $B$;
\item[(g)] 
$\Mass(1_\barB T') \le C \, |T|(\barB)$;
\item[(h)]
$\bd T'=\bd T$.
\end{itemizeb}
}

Using statements (a), (b) and (d) we easily obtain
that the currents $T_\eps$ and $\bd T_\eps$ have uniformly
bounded masses, and a standard compactness result for normal
currents yields the existence of $T'$.
Then statement~(e) follows from statement~(a);
statement~(f) follows from statement~(c);
statement~(g) follows from statements~(a) and (b), 
and statement~(h) follows from statement~(d). 

\medskip
\textit{Step~5. Completion of the proof.}
\\
\indent
As anticipated, we take $U:=T-T'$. Then statements~(i), (ii), (iii)
and (iv) follow respectively from statements (e), (f), (g) and (h).
\end{proof}

\begin{lemma}
\label{s-6.7}
Let $\tau$ be a Borel vectorfield on $\R^n$ which belongs 
to $L^1(\mu)$ and satisfies $\tau(x) \in N(\mu,x)$ for 
$\mu$-a.e.~$x$.
Then there exists a normal $1$-current $T$ on $\R^n$ 
such that, denoting by $\tilde\tau$ the Radon-Nikod\'ym density 
of $T$ \wrt $\mu$,
\begin{itemizeb}
\item[(i)]
$\| \tilde\tau-\tau\|_{L^1(\mu)} \le \frac{1}{2} \| \tau \|_{L^1(\mu)}$;
\item[(ii)]
$\bd T=0$ and $\Mass(T) \le C \| \tau \|_{L^1(\mu)}$.
\end{itemizeb}
\end{lemma}

\begin{proof}
We can clearly assume that $\tau$ is nontrivial (otherwise we take $T:=0$)
and then we set
\begin{equation}
m := \frac{ \|\tau\|_{L^1(\mu)} } {4 \, \Mass(\mu)}
\label{e-6.2.0}
\end{equation}
We begin with two well-known facts:
for all $x\in\R^n$ and all $r>0$ except at most 
countably many there holds
\begin{equation}
\mu(\bd B(x,r)) =0
\, , 
\label{e-6.2}
\end{equation}
and for $\mu$-a.e.~$x$ (and precisely for every $x$ where 
$\tau$ is $L^1(\mu)$-approximately continuous) 
and for $r>0$ small enough there holds
\begin{equation}
\int_{B(x,r)} |\tau - \tau(x)| \, d\mu 
\le m \, \mu(B(x,r)) 
\, .
\label{e-6.3}
\end{equation}
Moreover, by the definition of $N(\mu,x)$, 
for $\mu$-a.e.~$x$ (and precisely for every $x$ 
in the support of $\mu$ such that $\tau(x)\in N(\mu,x)$), 
there exists a normal $1$-current $T_x$
with $\bd T_x=0$ such that, for $r>0$ small enough,
\begin{equation}
| T_x- \tau(x) \, \mu| (B(x,r)) 
\le m \, \mu((B(x,r)) 
\, .
\label{e-6.4}
\end{equation}

Consider now the family of all closed balls $B(x,r)$ 
that satisfy \eqref{e-6.2}, \eqref{e-6.3} and \eqref{e-6.4}:
by a standard corollary of Besicovitch covering
theorem (see for example \cite{KP}, Proposition~4.2.13) 
we can extract from this family countably many balls 
$B_i=B(x_i,r_i)$ which are pairwise disjoint 
and cover $\mu$-almost every point.

For every $i$ we set $T_i:=T_{x_i}$, and use Lemma~\ref{s-6.6} 
to find a current $U_i$ with $\bd U_i=0$ which agrees with $T_i$ 
in the interior of $B_i$, is supported on $B_i$, and satisfies
\begin{equation}
\Mass(U_i) 
\le C \, |T_i|(B_i) 
\, ,
\label{e-6.5}
\end{equation}
and finally we set
\[
T := \sum_i U_i
\, . 
\]

We first show that $T$ is well-defined and satisfies statement~(ii).
Since the currents $U_i$ satisfy $\bd U_i=0$, it suffices to
show that $\sum_i \Mass(U_i) \le C \|\tau\|_{L^1(\mu)}$. 
And indeed%
\,\footnoteb
{The first inequality follows from \eqref{e-6.5}; 
the second one is obtained by writing the measure $T_i$ 
as sum of the measures $T_i-\tau(x_i)\mu$, $\tau(x_i)\mu-\tau\mu$
and $\tau\mu$;
the third one follows from \eqref{e-6.3} and \eqref{e-6.4};
the fourth one follows from the fact that the balls $B_i$ are 
pairwise disjoint, and finally the fifth one follows from 
\eqref{e-6.2.0}.}
\begin{align*}
  \sum_i \Mass(U_i)
& \le C \, \sum_i |T_i|(B_i) \\
& \le C \, \sum_i |T_i-\tau(x_i)\mu|(B_i) 
                + |\tau(x_i)\mu-\tau\mu|(B_i) + |\tau\mu|(B_i) \\
& \le C \, \sum_i m \, \mu(B_i) + m \, \mu(B_i) 
                + \int_{B_i} |\tau| \, d\mu \\
& \le C \big( 2m \, \Mass(\mu) + \|\tau\|_{L^1(\mu)} \big)
  \le 2C \|\tau\|_{L^1(\mu)}
  \, .
\end{align*}

Now we prove that $T$ satisfies statement~(i).
Since the balls $B_i$ are pairwise disjoint, 
in the interior of each $B_i$ the current $T$ agrees 
with $U_i$, which in turn agrees with $T_i$;
therefore $\tilde\tau$ agrees ($\mu$-a.e.) with $\tau_i$ 
in the interior of each $B_i$, or equivalently on $B_i$
(because the boundary of $B_i$ is $\mu$-negligible, 
cf.~\eqref{e-6.2}). 
Then%
\,\footnoteb
{The first equality follows by the fact the balls $B_i$ 
cover $\mu$-almost every point, 
for the second inequality we used \eqref{e-6.3}, 
the third one follows from \eqref{e-6.4}, 
the fourth one follows from the fact that the balls $B_i$ 
are disjoint, 
and finally the last equality follows from \eqref{e-6.2.0}.}
\begin{align*}
\| \tau - \tilde\tau \|_{L^1(\mu)}
& =\sum_i \int_{B_i} |\tau - \tau_i| \, d\mu \\
& \le \sum_i \int_{B_i} |\tau - \tau(x_i)| \, d\mu 
           + \int_{B_i} |\tau_i - \tau(x_i)| \, d\mu \\
& \le \sum_i m \, \mu(B_i) + |T_i-\tau(x_i) \, \mu| (B_i) \\
& \le \sum_i 2m \, \mu(B_i) 
  \le 2m \, \Mass(\mu) 
  = \textfrac{1}{2}\| \tau\|_{L^1(\mu)}
  \, .
  \qedhere
\end{align*}
\end{proof}

\begin{proof}[Proof of Theorem~\ref{s-6.2}]
We set $\tau_0:=\tau$
and then construct currents $T_j$ 
and vectorfields $\tilde\tau_j$, $\tau_j$
for $j=1,2,\dots$ according to 
the following inductive procedure:
we apply Lemma~\ref{s-6.7} to $\tau_{j-1}$ to obtain a normal 
$1$-current $T_j$ such that $\bd T_j=0$ and
\begin{equation}
\| \tau_{j-1} - \tilde\tau_j \|_{L^1(\mu)} 
\le \textfrac{1}{2} \| \tau_{j-1} \|_{L^1(\mu)}
\, , \quad
\Mass(T_j) 
\le C \| \tau_{j-1} \|_{L^1(\mu)}
\, ,
\label{e-6.7}
\end{equation}
where $\tilde\tau_j$ is the Radon-Nikod\'ym density of 
$T_j$ \wrt $\mu$; we then set $\tau_j:=\tau_{j-1}-\tilde\tau_j$.
We finally set 
\[
T:=\sum_{j=1}^\infty T_j
\, .
\]

We first prove that $T$ is well-defined 
and satisfies statement~(ii).
Since the currents $T_j$ satisfy $\bd T_j=0$, it suffices
to show that $\smash{ \sum_j \Mass(T_j) \le C\|\tau\|_{L^1(\mu)} }$.
To this regard, note that the first estimates in \eqref{e-6.7} 
can be rewritten as
$\smash{ \| \tau_j \|_{L^1(\mu)} 
\le \frac{1}{2} \| \tau_{j-1} \|_{L^1(\mu)} }$
and therefore, recalling that $\tau_0=\tau$, 
\begin{equation}
\| \tau_j \|_{L^1(\mu)} 
\le \textfrac{1}{2^j} \| \tau \|_{L^1(\mu)}
\, .
\label{e-6.8}
\end{equation}
Then, using the second estimates in \eqref{e-6.7},
\[
\sum_{j=1}^\infty \Mass(T_j) 
\le \sum_{j=1}^\infty C \| \tau_j \|_{L^1(\mu)}
\le \sum_{j=1}^\infty \frac{C}{2^j} \| \tau \|_{L^1(\mu)}
= C \| \tau \|_{L^1(\mu)}
\, .
\]

\smallskip
Next we show that $T$ satisfies statement~(i).
Since $\tilde\tau_j$ is the Radon-Nikod\'ym density 
of $T_j$ \wrt $\mu$, it suffices to show that the series
of all $\tilde\tau_j$ converge in $L^1(\mu)$ to $\tau$.
Since $\tau_0=\tau$ and 
$\tilde\tau_j=\tau_{j-1}-\tau_j$ for every $j$,
we have that 
\[
\tilde\tau_1+\cdots+\tilde\tau_j=\tau-\tau_j
\, , 
\]
and we conclude the proof by noticing that $\tau_j$ 
converge to $0$ in $L^1(\mu)$ by \eqref{e-6.8}.
\end{proof}

For the proof of Theorem~\ref{s-6.3}
we need the following lemmas.

\begin{lemma}
\label{s-6.8}
The map $x\mapsto N(\mu,x)$ is \emph{universally measurable}%
\,\footnoteb{That is, measurable \wrt any finite measure on $\R^n$.}
as a map from $\R^n$ to $\Gr(\R^n)$, and in particular 
it agrees outside a suitable $\mu$-negligible Borel set $E_0$ 
with a Borel map from $\R^n$ to $\Gr(\R^n)$.
\end{lemma}

\begin{proof}
[Sketch of proof]
Let $K$ be the support of $\mu$, and let $G$ be the graph 
of the restriction of $x\mapsto N(\mu,x)$ to $K$, that is, 
the set of $(x,v)$ such that $x\in K$ and $v\in N(\mu,x)$.
It suffices to prove that the set $G$ is \emph{analytic}
(cf.~\cite{Sri}, Chapter~4).

Let $N$ be the space of all normal $1$-currents 
$T$ on $\R^n$ with $\Mass(T)\le 1$ and $\bd T=0$, 
endowed with the weak* topology of currents
(as dual of smooth forms with compact support);
note that $N$ is compact and metrizable.
Now, for every $x\in K$, $v\in\R^n$, $T\in N$
we set
\begin{equation}
\label{e-ultima}
\psi(x,v,T) := \limsup_{r\to 0} \frac{ |T-v\mu|(B(x,r)) }{ \mu(B(x,r)) }
\, .
\end{equation}

It is easy to check that $v$ belongs to $N(\mu,x)$ if and 
only if there exists $T\in N$ such that $\psi(x,v,T)=0$,%
\footnoteb{Since we ask that $T$ belongs to $N$, 
we require that $\Mass(T)\le 1$; one should 
prove that this additional condition does not affect 
the definition of $N(\mu,x)$.}
which means that $G=p(\psi^{-1}(0))$ where $p$ 
is the projection of $K\times\R^n\times N$ on $K\times\R^n$.

It is also easy to see that the ratio at the right-hand side
of \eqref{e-ultima} is right-continuous in the variable $r$, 
and therefore the value of $\psi$ does not change if we assume 
that $r$ belongs to a given countable dense subset of $(0,+\infty)$. 
Moreover the ratio is Borel in the variables $x,v,T$, and 
therefore $\psi$ is Borel as well.
This implies that the set $\psi^{-1}(0)$ is Borel, 
and therefore $G=p(\psi^{-1}(0))$ is analytic.
\end{proof}

\begin{lemma}
\label{s-6.9}
Let $\{\sigma_t: \, t\in I\}$ be a family of measures 
on $\R^n$ which is Borel regular in $t$ (in the sense of
Remark~\ref{s-2.4}(i)). 
Assume that each $\sigma_t$ is the restriction of $\Haus^1$ 
to a $1$-rectifiable set $E_t$, and denote by
$D$ the set of all $(t,x)\in I\times\R^n$ such that 
the approximate tangent line $\Tan(E_t,x)$ exists. 

Then $D$ is a Borel set and 
$(t,x) \mapsto \Tan(E_t,x)$ is a Borel measurable map
from $D$ to $\Gr(\R^n)$.
\end{lemma}

\begin{proof}[Sketch of proof]
Let $\M^+$ denote the space of all positive, locally finite measures
on $\R^n$, endowed with the weak* topology.%
\footnoteb{As subset of the dual of the space of 
continuous functions with compact support in $\R^n$.}
Given a measure $\sigma\in\M^+$, a point $x\in\R^n$, 
and $r>0$, consider the rescaled measure $\sigma_{x,r}$ 
given by $\smash{\sigma_{x,r}(F) := \frac{1}{r} \sigma(x+rF)}$
for every Borel set $F$ in $\R^n$. 
Then the (one-dimensional) blow-up of $\sigma$ at $x$, denoted by
$\sigma_x$, is defined as the limit (in $\M^+$) of $\sigma_{x,r}$ 
as $r\to 0$, if such limit exists.

Let now $\sigma$ be the restriction of $\Haus^1$ 
to a $1$-rectifiable set $E$. 
The key point is that $E$ admits an approximate
tangent line at $x$ if and only if the blow-up 
$\sigma_x$ exists and belongs to the set $L$ of all 
measure in $\M^+$ given by the restriction of 
$\Haus^1$ to a line.

Now, since $(\sigma,x,r)\mapsto\sigma_{x,r}$ 
is a continuous map from $\M^+\times\R^n\times(0,1]$
in $\M^+$, it is easy to see that the set of all 
$(\sigma,x)\in \M^+\times\R^n$ such that $\sigma_x$ 
exists is Borel, and that $(\sigma,x)\mapsto\sigma_{x,r}$
is a Borel map from this set to $\M^+$. 
Since moreover $L$ is a closed subset of $\M^+$, 
then also the set of all $(\sigma,x)$ such that 
$\sigma_x$ exists and belongs to $L$ is Borel.

Using these facts and recalling that 
$t\mapsto\sigma_t$ is Borel we easily 
conclude the proof.
\end{proof}

\begin{lemma}
\label{s-6.10}
Let $C=C(e,\alpha)$ be a closed convex cone in 
$\R^n$ (cf.~\S\ref{s-cones}) and let $\Int(C)$ 
be the set of all lines spanned by vectors in the 
interior of $C$.
Let $\sigma$ be a non-trivial measure on $\R^n$ which 
can be decomposed as $\smash{\sigma=\int_I \sigma_t \, dt}$ 
where each $\sigma_t$ is the restriction of $\Haus^1$ 
to a $1$-rectifiable set $E_t$ such that
$\Tan(E_t,x)$ belongs to $\Int(C)$ for 
$\Haus^1$-a.e.~$x\in E_t$.

Then there exists a normal $1$-current $T$ 
with $\bd T=0$ whose 
Radon-Nikod\'ym density \wrt $\sigma$
belongs to $C$ for $\sigma$-a.e.\ point and 
is nonzero in a set of positive $\sigma$-measure (that is, 
the measures $|T|$ and $\sigma$ are not 
mutually singular).
\end{lemma}

\begin{proof}
We first construct a current $T$ that satisfies all
requirements except $\bd T=0$, and only at the end 
we explain how to modify $T$ so that $\bd T=0$.

The idea for the construction of $T$ is quite simple: for 
every $t$ we choose a $C$-curve $G_t$ (cf.~\S\ref{s-cones})
such that $\Haus^1(E_t\cap G_t)>0$; we then denote by
$T_t$ the $1$-current associated to $G_t$,
and set $\smash{T := \int T_t \, dt}$.
However, some care must be taken with measurability
issues (for example, $G_t$ should be chosen in a Borel 
measurable fashion \wrt~$t$).

%
%

Before starting with the construction of $T$, we 
note that, possibly replacing $I$ with a suitable Borel 
subset, we can assume that $\Haus^1(E_t)>0$ for every
$t\in I$.

Moreover we denote by $\X$ the class of all paths 
$\gamma:J:=[-1,1]\to\R^n$ such that $\Lip(\gamma)\le 1$ 
and $\dot\gamma(s)\in C$ for a.e.~$s\in J$,%
\footnoteb{Through this proof the interval $J$ is endowed 
with the Lebesgue measure, which we do not write explicitly. 
Note that each $\gamma(J)$ is a $C$-curve 
(cf.~\S\ref{s-cones}).}
and we endow $\X$ with the supremum distance.

The rest of the proof is divided in several steps.

\medskip
\textit{Step~1. For every $t\in I$ there exists $\gamma\in \X$  
such that}
\begin{equation}
\label{e-4old.1}
\Haus^1(E_t \cap \gamma(J)) = \sigma_t(\gamma(J)) > 0
\, .
\end{equation}
\indent
Since the set $E_t$ is rectifiable and $\Haus^1(E_t)>0$, 
we can find a curve $G$ of class $C^1$ such that 
$\Haus^1(E_t\cap G)>0$. We take a point
$x_0\in G$ such that $E_t\cap G$ has density $1$ at $x_0$.
Then $\Tan(G,x_0)$ agrees with $\Tan(E_t,x_0)$ and belongs 
to $\Int(C)$, which implies that $\Tan(G,x)$ is contained 
in $C$ for all $x$ in a suitable subarc $G'$ of $G$ 
that contains $x_0$, and clearly $\Haus^1(E_t\cap G')>0$.
We then take as $\gamma$ a suitable
parametrization of $G'$.

\medskip
\textit{Step~2. The set $F$ of all $(t,\gamma)\in I\times \X$
such that \eqref{e-4old.1} holds is Borel.}
\\ 
\indent
It suffices to show that $(t,\gamma) \mapsto \sigma_t(\gamma(J))$
is a Borel function on $I\times \X$, and this is an immediate 
consequence of the following facts:
\begin{itemizeb}
\item
$t\mapsto\sigma_t$ is a Borel map from $I$ to the space
$\M^+$ of finite positive Borel measures on $\R^n$ 
endowed with the weak* topology (cf.~Remark~\ref{s-2.4}(i));
\item
$\gamma\mapsto \gamma(J)$ is a Borel map from $\X$ to the space
$\K$ of compact subsets of $\R^n$ endowed
with the Hausdorff distance;
\item
$(K,\sigma)\mapsto \sigma(K)$ is a Borel function on $\K\times\M^+$.
\end{itemizeb}

\medskip
\textit{Step~3. For every $t\in I$ we can choose
$\gamma_t\in \X$ so that \eqref{e-4old.1} holds
and $t\mapsto \gamma_t$ agrees with a Borel map 
in a subset $I'$ with full measure in $I$.}
\\ 
\indent
The set $F$ defined in Step~2 is a Borel 
subset of $I\times \X$, and by Step~1 its projection on $I$
agrees with $I$ itself. 
Thus we can use the von~Neumann measurable selection theorem 
(see \cite{Sri}, Theorem~5.5.2),
to choose $\gamma_t \in \X$ for every $t\in I$
so that $(t,\gamma_t)$ belongs to $F$ (that is, $\gamma_t$
satisfies \eqref{e-4old.1}) and the map $t\mapsto\gamma_t$ 
is \emph{universally measurable}, and in particular it
agrees with a Borel map in a subset $I'$ with full 
measure in $I$. 

\medskip
\textit{Step~4. Construction of the normal current $T$.}
\\ 
\indent
We let $T$ be the integral (over $t\in I'$) of the $1$-currents
canonically associated to the paths $\gamma_t$, that is,
\begin{equation}
\label{e-4old.2}
\scal{T}{\omega}:= 
\int_{I'} \bigg[ 
   \int_J \scal{\omega(\gamma_t(s))}{\dot\gamma_t(s)} \, ds
\bigg] \, dt
\end{equation}
for every smooth $1$-form $\omega$ on $\R^n$ with compact support.%
\footnoteb{The integral in this formula is well-defined because 
$t\mapsto \gamma_t$ is a Borel map from $I'$ to $\X$ (Step~3),
and then $t\mapsto \dot\gamma_t$ is a Borel map from $I'$ to 
$L^1(J;\R^n)$.}
A simple computation shows that
\begin{equation}
\label{e-4old.3}
\scal{\bd T}{\varphi}
= \scal{T}{d\varphi}
= \int_{I'} \big[ \varphi(\gamma_t(1)) - \varphi(\gamma_t(-1)) \big] \, dt
\end{equation}
for every smooth $0$-form (or function) $\varphi$ on $\R^n$ 
with compact support.
It follows immediately from \eqref{e-4old.2} and \eqref{e-4old.3} that 
both $T$ and $\bd T$ have finite mass, and therefore $T$ is normal.

\medskip
\textit{Step~5. The Radon-Nikod\'ym density of $T$ \wrt $\sigma$ 
takes values in $C$.}
\\ 
\indent
It suffices to show that $T$, viewed as a measure, takes values 
in $C$. Take indeed a Borel set $E$ in $\R^n$: formula \eqref{e-4old.2} 
yields
\begin{equation}
\label{e-4old.4}
T(E) = 
\int_{I'} \bigg[ 
   \int_{\gamma_t^{-1}(E)} \dot\gamma_t(s) \, ds 
\bigg] \, dt
\, ,
\end{equation}
and since $\dot\gamma_t(s)$ belongs to the cone $C$, 
which is closed and convex, so does $T(E)$.

\medskip
\textit{Step~6. The measures $\sigma$ and $|T|$ are 
not mutually singular.}
\\ 
\indent
For every $t\in I'$ let $\sigma'_t$ be the restriction 
of $\Haus^1$ to $E_t\cap \gamma_t(J)$, or equivalently 
the restriction of $\sigma_t $ to $\gamma_t(J)$, and set
$\sigma':=\int_{I'} \sigma'_t \, dt$.

Note that the measure $\sigma'$ is a nontrivial (because 
of the choice of $\gamma_t$) and therefore we can prove 
the claim by showing that $\sigma'\le \sigma$ and 
$\sigma'\le m |T|$ with $m:=1/\cos\alpha$. 
The first inequality is immediate. 
Concerning the second one, for every Borel set 
$E$ in $\R^n$ we have that%
\,\footnoteb{The first equality follows from \eqref{e-4old.4};
the first inequality follows from the fact that 
$\dot\gamma_t(s)$ belongs to $C=C(e,\alpha)$ and that every 
$v\in C$ satisfies $v \cdot e \ge m |v|$; the second inequality 
follows from the area formula, and the last one from the definition 
of $\sigma'$.} 
\begin{align*}
  |T| (E)
  \ge T(E) \cdot e 
& = \int_{I'} \bigg[ 
     \int_{\gamma_t^{-1}(E)} \dot\gamma_t(s) \cdot e \, ds 
  \bigg] \, dt \nonumber \\
& \ge \int_{I'} \bigg[ 
     \int_{\gamma_t^{-1}(E)} \cos\alpha \, |\dot\gamma_t(s)| \, ds
  \bigg] \, dt \nonumber \\
& \ge \cos\alpha \int_{I'} \Haus^1(\gamma_t(J) \cap E) \, dt
  \ge \cos\alpha \, \sigma'(E)
  \, . 
\end{align*}

\medskip
\textit{Step~7. How to modify $T$ to obtain $\bd T=0$.}
\\ 
\indent
We choose an open ball $B$ such that $\sigma(B)>0$.
Then, possibly replacing $\sigma$ with its restriction
to $B$, we can assume that $\sigma$ is supported 
in $B$, which means the set $E_t$ is contained 
in $B$ up to an $\Haus^1$-negligible subset
for (almost) every $t$.

We then proceed with the construction of $T$ shown 
above, with the only difference that $\X$ is now the
class of all paths $\gamma$ from $J=[-1,1]$ to the 
closure of $B$ such that the endpoints $\gamma(\pm 1)$ 
belong to $\bd B$, 
$\Lip(\gamma)\le r/\cos\alpha$ where $r$ is the 
radius of $B$, and $\dot\gamma(s)\in C$ 
for a.e.~$s\in J$, as before.%
\footnoteb{The only modification in the proof
occurs in step~1, where the path $\gamma$ must be 
suitably extended so that the endpoints belongs 
to $\bd B$.}

We thus obtain a current $T$ that satisfies
the same properties as before, and in addition 
its boundary is supported on $\bd B$ 
(see \eqref{e-4old.3}).
Finally we apply Lemma~\ref{s-6.6} to the current
$T$ and the ball $B$ to obtain a current $U$ without 
boundary that agrees with $T$ in $B$.
Using this property and the fact that $\sigma$ is 
supported in $B$ we easily obtain that 
Radon-Nikod\'ym density of $U$ \wrt $\mu$
agrees with that of $T$.
We conclude the proof by replacing $T$ by~$U$. 
\end{proof}

\begin{proof}[Proof of Theorem~\ref{s-6.3}]
We first prove that $N(\mu,x)\subset V(\mu,x)$ for 
$\mu$-a.e.~$x$.

We argue by contradiction, and assume that this inclusion 
does not hold. Then, using the Kuratowski and 
Ryll-Nardzewski's measurable selection theorem 
(see \cite{Sri}, Theorem~5.2.1),
we can find a bounded Borel vectorfield $\tau$ on $\R^n$
such that $\tau(x) \in N(\mu,x) \setminus V(\mu,x)$
for every $x$ in a set of positive $\mu$-measure
(here we need Lemma~\ref{s-6.8}).

Then Theorem~\ref{s-6.2} yields 
a normal $1$-current $T$ whose Radon-Nikod\'ym density
\wrt $\mu$ agrees with $\tau$, and Proposition~\ref{s-curr1}
implies that $\tau(x) \in V(\mu,x)$ for $\mu$-a.e.~$x$, 
in contradiction with the choice of~$\tau$.

\medskip
We now prove that $V(\mu,x)\subset N(\mu,x)$ for 
$\mu$-a.e.~$x$. 

First of all, we use Lemma~\ref{s-6.8} to modify 
the map $x\mapsto N(\mu,x)$ in a $\mu$-negligible set 
and make it Borel measurable.

By the definition of $V(\mu,\cdot)$ 
it suffices to show that the map $x\mapsto N(\mu,x)$ belongs to
the class $\G_\mu$ (see \S\ref{s-decomp}).
In other words, given a measure 
$\mu'$ of the form $\smash{\mu'=\int_I \mu_t \, dt}$ 
such that $\mu'\ll\mu$ and each $\mu_t$ is the restriction 
of $\Haus^1$ to a $1$-rectifiable set $E_t$, 
we must show that 
\begin{equation}
\label{e-4old.7}
\Tan(E_t,x) \subset N(\mu,x)
\quad\text{for $\mu_t$-a.e.~$x$ and a.e.~$t$.}
\end{equation}

We now argue by contradiction, and assume that \eqref{e-4old.7}
does not hold. 

\medskip
\textit{Step~1. There exist a cone $C=C(e,\alpha)$ and a 
non-trivial measure $\sigma$ such that
\begin{itemizeb}
\item[(a)]
$C$ and $\sigma$ satisfy the assumptions in  
Lemma~\ref{s-6.10};
\item[(b)]
$\sigma\ll\mu'\ll\mu$;
\item[(c)]
$N(\mu,x) \cap C=\{0\}$ for $\sigma$-a.e.~$x$.
\end{itemizeb}
}

Let $\mu''$ be the measure on $I\times\R^n$ 
given by $\mu'' := \int_I (\delta_t \times \mu_t) \, dt$
where $\delta_t$ is the Dirac mass at $t$, and let $F$
be the set of all $(t,x)\in I\times\R^n$ 
such that $\Tan(E_t,x)$ exists and is not contained 
in $N(\mu,x)$.%
\footnoteb{Using Lemma~\ref{s-6.9} one can prove that
$F$ is Borel.}
Then the assumption that \eqref{e-4old.7} does not
hold can be restated by saying that $\mu''(F)>0$.

Now, let $\F$ be a family of cones $C=C(e,\alpha)$
with $e$ ranging in a given countable dense subset of
the unit sphere in $\R^n$ and $\alpha$ ranging in a given 
countable dense subset of $(0,\pi/2)$, and for 
every $C\in\F$ let $F_C$ be the subset of 
$(t,x) \in F$ such that $\Tan(E_t,x)$ and $N(\mu,x)$
are separated by $C$, that is,
$\Tan(E_t,x) \in \Int(C)$ and $N(\mu,x) \cap C=\{0\}$.%
\footnoteb{The set $F_C$ is Borel, and this follows 
from the Borel measurability of 
the set $F$ and of the maps $(t,x)\mapsto\Tan(E_t,x)$ 
and $x\mapsto N(\mu,x)$.}

Then the sets $F_C$ with $C\in\F$ form a countable
cover of $F$, and since $\mu''(F)>0$ there exists 
at least one $C\in\F$ such that $\mu''(F_C)>0$.

We then take $\sigma$ equal to the push-forward
according to $p$ of the restriction of $\mu''$ 
to the set $F_C$, where $p$ is the projection of
$I\times\R^n$ on $\R^n$.
Note that $\sigma=\int_I \sigma_t \, dt$
where $\sigma_t$ is the restriction of $\mu_t$ 
to the set of all $x$ such that $(t,x)\in F_C$.

\medskip
\textit{Step~2. Completion of the proof.}
\\ 
\indent
By applying Lemma~\ref{s-6.10} to the cone $C$
and the measure $\sigma$ constructed in Step~1
we obtain a normal $1$-current $T$ with $\bd T=0$
whose Radon-Nikod\'ym density \wrt $\sigma$
belongs to $C$ $\sigma$-a.e., and 
is nonzero on a set of positive 
$\sigma$-measure.

Since $\sigma\ll\mu$ (statement~(b) above) 
we deduce that also the Radon-Nikod\'ym density of 
$T$ \wrt $\mu$, which we denote by $\tau$, 
belongs to $C$ $\sigma$-a.e.\ and 
is nonzero on a set of positive $\sigma$-measure.

Moreover we have that $\tau(x) \in N(\mu,x)$ 
for $\mu$-a.e.~$x$ (cf.\ Remark~\ref{s-6.5}(ii))
and therefore also for $\sigma$-a.e.~$x$.
Therefore $N(\mu,x) \cap C \ne \{0\}$
for a set of positive $\sigma$-measure of $x$, 
in contradiction with statement~(c) above.
\end{proof}

%
%

\section{Appendix: proofs of technical results}
\label{s7}
In this appendix we prove several 
technical results used in the previous sections, 
in the following order: Lemma~\ref{s-regularize}, 
Proposition~\ref{s-spancar}, 
Lemma~\ref{s-rainwatercor3},  
Proposition~\ref{s-rainwatercor4}, 
Lemma~\ref{s-perturb2}, 
Lemma~\ref{s-approx}.

\begin{proof}[Proof of Lemma~\ref{s-regularize}]
We let $L:=\Lip(f)$ and for every $k=1,2,\dots$ we set%
\,\footnoteb{For $k=1$ we convene that $1/0=+\infty$.}
\[
A_k := 
\Big\{ x\in\R^n: \ 
   \frac{1}{k+1} < \dist(x,K) < \frac{1}{k-1} \Big\}
\, .
\] 
Then $\{A_k\}$ is an open cover of the open set 
$A:=\R^n\setminus K$, and therefore we can take 
a smooth partition of unity $\{\sigma_k\}$ of $A$ 
subject to this cover.%
\footnoteb{This means that each $\sigma_k$ is 
a non-negative smooth function on $A$ with support
contained in $A_k$ and that every $x\in A$ admits 
a neighbourhood where $\sigma_k$ vanishes for all
but finitely many $k$, and $\sum_k \sigma_k(x)=1$.
\label{f-7.1}}

Next we choose a decreasing sequence of positive 
real numbers $r_k$ such that for every $k$
there holds%
\,\footnoteb{Each $\| d\sigma_k \|_\infty$ 
is finite because $d\sigma_k$ 
is continuous and compactly supported in $A$.}
\begin{equation}
\label{e-pippo1}
L \, \| d\sigma_k \|_\infty r_k \le 2^{-k}\eps 
\quad\text{and}\quad
Lr_k \le \phi\Big(\frac{1}{k+1}\Big)
\, ,
\end{equation}
and  a sequence of positive mollifiers $\rho_k$ 
with support contained in the ball $B(r_k)$.
Finally we set
\begin{equation}
\label{e-pippo2}
g:= f + \sum_{k=1}^{+\infty} \sigma_k(f*\rho_k-f)
\, .
\end{equation}

To prove statement~(i) 
note first that $g$ agrees with $f$ on $K$ 
because every $x\in K$ does not belong to $A_k$ 
for every $k$, and then $\sigma_k(x)=0$.
On the other hand $g$ is well-defined and smooth 
on the open set $A$ because the functions in the sum 
in \eqref{e-pippo2} are locally null for all but 
finitely many indexes $k$ (recall footnote~\ref{f-7.1}) 
and $g$ can be rewritten as
\[
g = \sum_{k=1}^{+\infty} \sigma_k (f*\rho_k)
\, .
\]

Let us now prove statement~(ii).
Since the support of $\rho_k$ is contained 
in the ball $B(r_k)$ and $f$ has Lipschitz
constant $L$, a simple computation
shows that for every $x\in\R^n$ there holds
\begin{equation}
\label{e-pippo2.1}
|f*\rho_k(x) - f(x)| \le Lr_k 
\, .
\end{equation}
Therefore, given $x\in\R^n\setminus K$
and denoting by $k(x)$ the smallest
$k$ such that $x\in A_k$, we have%
\,\footnoteb{For the first inequality we use 
that $x\notin A_k$ (and then $\sigma_k(x)=0$) 
for $k<k(x)$;
the second inequality follows from \eqref{e-pippo2.1};
the third one follows from the fact 
that the sum of all $\sigma_k(x)$ is $1$ 
and $r_k(x) \ge r_k$ for every $k\ge k(x)$;
the fourth one follows from the second
inequality in \eqref{e-pippo1}, the fifth one
from the fact that $x$ belongs to $A_{k(x)}$ 
(and the definition of the sets $A_k$).}
\begin{align*}
|g(x)-f(x)| 
& \le \sum_{k\ge k(x)} \sigma_k(x) \, |f*\rho_k(x) - f(x)| \\
& \le \sum_{k\ge k(x)} \sigma_k(x) \, Lr_k \\
& \le Lr_{k(x)}
  \le \phi \Big( \frac{1}{k(x)+1} \Big) 
  \le \phi (\dist(x,K))
  \, .
\end{align*}

We conclude the proof by showing that $g$ is Lipschitz and
satisfies statement~(iii). For every $h=1,2,\dots$ set
\[
g_h := f + \sum_{k=1}^{h} \sigma_k(f*\rho_k-f)
\, .
\]
Since the functions $g_h$ are Lipschitz and
converge pointwise to $g$ as $h\to+\infty$, it suffices
to show that $\Lip(g_h) \le L +\eps$ for every $h$, 
or equivalently that the distributional derivatives 
$dg_h$ satisfies 
\begin{equation}
\label{e-pippo3}
\|dg_h\|_\infty \le L+\eps
\, .
\end{equation}

Let $h$ be fixed for the rest of the proof.
We can write $g_h$ as 
\[
g = \sum_{k=	0}^h \sigma_k f_k
\]
where we have set 
$\sigma_0:= 1 - (\sigma_1+\cdots+\sigma_h)$, 
$f_0:=f$, and $f_k:=f*\rho_k$ for $0<k \le h$.
Then%
\,\footnoteb{For the second equality we use that 
$d\sigma_0 + \cdots + d\sigma_h=0$, which is obtained 
by deriving the identity $\sigma_0 + \cdots + \sigma_h=1$.%
\label{f-7.2}}
\begin{equation}
\label{e-pippo4}
dg_h 
= \sum_{k=0}^h \sigma_k \, df_k + f_k \, d\sigma_k
= \sum_{k=0}^h \sigma_k \, df_k
    + \sum_{k=1}^h (f_k -f) \, d\sigma_k
    \, .
\end{equation}

Observe now that $df_k = df *\rho_k$ where
$df$ is the distributional derivative of $f$, 
and then 
$\|df_k\|_\infty \le \|df\|_\infty \|\rho_k\|_1 \le L$;
hence the second sum in line \eqref{e-pippo4}
is a (pointwise) convex combinations of
functions with $L^\infty$-norm at most $L$, 
and therefore its $L^\infty$-norm is
at most $L$ as well. 
Thus it remains to show that 
the $L^\infty$-norm of the third sum in line
\eqref{e-pippo4} is at most $\eps$, and indeed%
\,\footnoteb{The second inequality follows from
the fact that $f_k=f*\rho_k$ and \eqref{e-pippo2.1}, 
the third one from the first inequality in \eqref{e-pippo1}.}
\[
\bigg\| \sum_{k=1}^h (f_k -f) \, d\sigma_k \bigg\|_\infty
\le \sum_{k=1}^h \| f_k -f \|_\infty  \, \| d\sigma_k \|_\infty
\le \sum_{k=1}^h Lr_k \| d\sigma_k \|_\infty 
\le \eps
\, .
\qedhere
\]
\end{proof}

\begin{proof}[Proof of Proposition~\ref{s-spancar}]
Statement~(i) is immediate, while statements (ii) and (iii)
are consequence of the following general facts, respectively:
if $\dim(W)<k$ then every $k$-vector in $W$ is null, and
if $\dim(W)=k$ then every $k$-vector in $W$ is simple.

\medskip
To prove statement~(iv), we denote by $W$ the linear subspace 
of $V$ generated by $v_1,\dots, v_k$. Clearly $\Span(v)$ is
contained in $W$; moreover $\Span(v)$ has dimension at 
least $k$ by statement~(ii) while $W$ has dimension at 
most $k$; therefore $\Span(v)$ and $W$ agree and have 
both dimension~$k$.

\medskip
To prove statement~(v), we need some additional notation.
Let $n:=\dim(V)$, 
let $\{e_i: \, i=1,\dots,n\}$ be a basis of $V$, 
and let $\{e^*_i: \, i=1,\dots,n\}$ be the corresponding 
dual basis (see footnote~\ref{f-3.1} in Section~\ref{s3}).
For every integer $k$ with $0<k\le n$ we denote by
$I(n,k)$ the set of all multi-indexes $\bfi=(i_1,\dots,i_k)$
whose coordinates are integers and satisfy
$1\le i_1 < i_2 < \dots < i_k\le n$, 
and for every such $\bfi$ we set
$e_\bfi := e_{i_1}\wedge\dots\wedge e_{i_k}$
and $e^*_\bfi := e^*_{i_1}\wedge\dots\wedge e^*_{i_k}$.%
\footnoteb{Recall that $\{e_\bfi: \, \bfi\in I(n,k)\}$ is a basis 
of $\largewedgef_k(V)$ while $\{e^*_\bfi: \, \bfi\in I(n,k)\}$
is the corresponding dual basis of $\largewedgef^k(V)$.}

It is then easy to check that for every 
$\bfi\in I(n,k)$ and $\bfj\in I(n,k-1)$
there holds
\begin{equation}
\label{e-pippo}
e_\bfi \trace e^*_\bfj =
\begin{cases}
 (-1)^{h-1} e_{i_h}
   & \text{if $\bfj=(i_1,\dots,i_{h-1},i_{h+1},\dots,i_k)$} \\
   & \text{with $h=1,\dots,k$,} \smallskip \\ 
 0 & \text{otherwise.} 
\end{cases}
\end{equation}

Let now $W$ denote the set of all
$v\trace\alpha$ with $\alpha\in \largewedge^{k-1}(V)$.
The proof of statement~(v), namely that $\Span(v)=W$, 
is divided in two steps.

\medskip
\textit{Step~1. $W$ is contained in $\Span(v)$.}
\\ 
\indent
We must show that
$\tau\trace\alpha \in \Span(v)$ 
for ever y $\alpha\in\largewedge^{k-1}(V)$.

We let $k':=\dim(\Span(v))$ and
choose the basis $\{e_i: \, i=1,\dots,n\}$ of $V$ 
so that $\{e_i: \, i=1,\dots,k'\}$ is a basis of $\Span(v)$.
By definition $\tau$ is a $k$-vector in $\Span(v)$, 
and therefore it can be written as a linear combination 
of the $e_\bfi$ with $\bfi\in I(n,k)$
such that $i_k\le k'$, while $\alpha$ can be written 
as a linear combination of $\smash{e^*_\bfj}$ with $\bfj\in I(n,k-1)$.

Thus $\tau\trace\alpha$ is a linear combination of 
the vectors $\smash{e_\bfi \trace e^*_\bfj}$ with $\bfi, \bfj$ 
as above, and therefore it suffices to prove that every 
such vector belongs to $\Span(v)$.
To this end, notice that \eqref{e-pippo}
implies that $\smash{e_\bfi \trace e^*_\bfj}$ is either $0$ or $e_i$ 
for some $i\le i_k$, and therefore it belongs $\Span(v)$
whenever $i_k\le k'$.

\medskip
\textit{Step~2. $\Span(v)$ is contained in $W$, that is, $v$ is 
a $k$-vector in $W$.}
\\ 
\indent
We let now $k':=\dim(W)$ and
choose the basis $\{e_i: \, i=1,\dots,n\}$ of $V$ 
so that $\{e_i: \, i=1,\dots,k'\}$ is a basis of $W$.
We must show that $v$ is a linear combination 
of the $e_\bfi$ with $\bfi\in I(n,k)$ such that $i_k\le k'$,
or equivalently that $\scal{v}{e^*_\bfj}=0$
for all $\bfj\in I(n,k)$ such that $j_k > k'$. 

Observe now that each of these $e^*_\bfj$ can be written
as $e^*_{\bfj'}\wedge e^*_j$ with $\bfj'\in I(n,k-1)$ and
$j>k'$. Therefore
\[
\scal{v}{e^*_\bfj} 
=\scal{v}{e^*_{\bfj'}\wedge e^*_j}
=\scal{v \trace e^*_{\bfj'}}{e^*_j}
= 0
\]
because $v \trace e^*_{\bfj'}$ belongs by definition to
$W$, and $\scal{w}{e^*_j}=0$ for every $w\in W$ and 
every $j\ge k'$ by the choice of the basis.
\end{proof}

Lemma~\ref{s-rainwatercor3} and 
Proposition~\ref{s-rainwatercor4}
are based on the following result.

\begin{parag}[Rainwater's Lemma]
\label{s-rainwater}
(See \cite{Rainwater} or \cite{Rudin}, Lemma~9.4.3).
\textit{
Let $X$ be a compact metric space, $\F$ 
a family of probability measures on $X$ 
which is convex and weak* compact, and 
$\mu$ a measure on $X$ which is singular
with respect to every $\lambda\in\F$.
Then $\mu$ is supported on a Borel set $E$ which 
is $\lambda$-null for every $\lambda\in\F$.
}
\end{parag} 

For our purposes we need the following 
variant of Rainwater's lemma:

\begin{corollary}
\label{s-rainwatercor1}
Let $X$ be a compact metric space and $\F$ a 
weak* compact family of probability measures on $X$.
Then for every measure $\mu$ on $X$ one of the 
following (mutually incompatible) alternatives holds:
\begin{itemizeb}
\item[(i)]
$\mu$ is supported on a Borel set $E$ which 
is $\lambda$-null for every $\lambda\in\F$;
\item[(ii)]
there exists a probability measure $\sigma$ 
supported on $\F$ and a Borel set $E$ such that 
the measure 
\[
\int_{\lambda\in\F} (1_E \, \lambda) \, d\sigma(\lambda)
\]
is nontrivial and absolutely continuous \wrt $\mu$.%
\footnoteb{It is easy to check that this measure is 
well-defined in the sense of \S\ref{s-measint}.}
\end{itemizeb}
\end{corollary} 

\begin{proof}
We denote by $P(\F)$ the space of probability 
measures on the compact space $\F$,
and for every $\sigma\in P(\F)$ we denote by 
$[\sigma]$ the corresponding average of the 
elements of $\F$, that is,  
the measure on $X$ given by
\[
[\sigma] := \int_{\lambda\in\F} \lambda \, d\sigma(\lambda)
\, .
\]

We claim that the class $\F'$ of all $[\sigma]$ with
$\sigma\in P(\F)$ is convex and compact (\wrt the
weak* topology of measures on $X$).
Convexity is indeed obvious, and compactness follows
from the compactness of the space $P(\F)$
(endowed with the weak* topology of measures on $\F$)
and the continuity of the map $\sigma \mapsto [\sigma]$, 
which in turn follows from the identity
$\scal{[\sigma]}{\varphi}= \scal{\sigma}{\hat\varphi}$
where $\varphi$ is any continuous function on $X$ and 
$\hat\varphi$ is the continuous function on $\F$ defined by
$\hat\varphi(\lambda):=\scal{\lambda}{\varphi}$.

There are now two possibilities: either $\mu$ is 
singular with respect to all measures in $\F'$ or not.

In the first case Theorem~\ref{s-rainwater} implies
that $\mu$ is supported on a set $E$ which is
null \wrt all measures in $\F'$, 
and therefore also \wrt all measures in $\F$
(because $\F$ is contained in $\F'$).
Thus (i) holds.

In the second case there exists $\sigma\in P(\F)$ 
such that $\mu$ is not singular with respect to
$[\sigma]$, and therefore by the Lebesgue-Radon-Nikod\'ym
theorem there exists a set $E$ such that the restriction 
of $[\sigma]$ to $E$ is nontrivial and absolutely continuous
\wrt $\mu$. Thus (ii) holds with this $\sigma$ and this $E$.
\end{proof}

\begin{lemma}
\label{s-rainwatercor2}
Let $C=C(e,\alpha)$ be a cone in $\R^n$ with axis $e$ 
and angle $\alpha$ (cf.~\S\ref{s-cones}).
Then, for every measure $\mu$ on $\R^n$, one of the 
following (mutually incompatible) alternatives 
holds:
\begin{itemizeb}
\item[(i)]
$\mu$ is supported on a Borel set $E$ which 
is $C$-null (see~\S\ref{s-cones});
\item[(ii)]
there exists a nontrivial measure of
the form $\smash{\mu'=\int_I \, \mu_t\, dt}$
such that $\mu'$ is absolutely continuous \wrt $\mu$, 
each $\mu_t$ is the restriction of $\Haus^1$ 
to some $1$-rectifiable set $E_t$, and
\[
\Tan(E_t,x) \subset (C \cup (-C))
\quad\text{for $\mu_t$-a.e.~$x$ and a.e.~$t$.}
\]
\end{itemizeb}
\end{lemma} 

\begin{proof}
The idea is to apply Corollary~\ref{s-rainwatercor1}
to the measure $\mu$ and a sequence of suitably chosen 
families $\F_k$ of probability measures.

\medskip
\textit{Step~1. Construction of the families $\F_k$.}
\\ 
\indent
Given $k=1,2,\dots$, we define the following objects:
\begin{itemizeb}\leftskip 0.2 cm\labelsep=.3 cm
\item[$\G_k$] 
set of all paths $\gamma$ from $[0,1]$
to the closed ball $B(k)$ such that
$\Lip(\gamma)\le 1$ and 
$\dot\gamma(s) \cdot e \ge \cos\alpha$ 
for $\Leb^1$-a.e.~$s\in [0,1]$;
\item[$G_\gamma$]
$:=\gamma([0,1])$, image of the path $\gamma\in \G_k$;
\item[$\mu_\gamma$] 
restriction of $\Haus^1$ to the curve $G_\gamma$;
\item[$\lambda_\gamma$] 
push-forward according $\gamma$
of the Lebesgue measure on $[0,1]$;
\item[$\F_k$] 
set of all $\lambda_\gamma$ with $\gamma\in\G_k$.
\end{itemizeb}
One easily check that each $G_\gamma$ is a 
$C$-curve (see~\S\ref{s-cones}) contained in $B_k$,
and $\lambda_\gamma$ is a probability 
measure supported on $G_\gamma$ such that
\begin{equation}
\label{e-7.1}
\mu_\gamma 
\le \lambda_\gamma 
\le \frac{1}{\cos\alpha} \, \mu_\gamma 
\, .
\end{equation}
In particular $\F_k$ is a subset of the space 
$P(B_k)$ of probability measures on $B_k$.

\medskip
\textit{Step~2. Each $\F_k$ is a weak* compact subset
of $P(B_k)$.}
\\ 
\indent
This is a consequence of the following statements:
\begin{itemizeb}
\item[(a)]
the space $\G_k$ endowed with 
the supremum distance is compact;
\item[(b)] 
$\F_k$ is the image of $\G_k$ according to 
the map $\gamma\mapsto \lambda_\gamma$ and this map
is continuous (as a map from $\G_k$ to $P(B_k)$ 
endowed with the weak* topology).
\end{itemizeb}
Statement~(a) follows from the usual compactness 
for the class of all paths $\gamma:[0,1]\to B(k)$ 
with $\Lip(\gamma)\le 1$ and the fact that we can 
rewrite the second constraint in the definition 
of $\G_k$ as 
\[
(\gamma(s')-\gamma(s)) \cdot e \ge \cos\alpha \, (s'-s)
\quad\text{for every $s,s'$ with $0 \le s \le s' \le 1$,}
\]
which is clearly closed with respect to uniform convergence.
To prove statement~(b) we observe that for every $\gamma\in\G_k$ 
and every continuous test function $\varphi: B_k\to\R$ there holds
\[
\scal{\lambda_\gamma}{\varphi}
= \int_{B_k} \varphi \, d\lambda_\gamma 
= \int_{[0,1]} \varphi(\gamma(s)) \, d\Leb^1(s)
\, ,
\]
and therefore the function $\gamma \mapsto \scal{\lambda_\gamma}{\varphi}$ 
is continuous on $\G_k$.

\medskip
\textit{Step~3. Completion of the proof.}
\\ 
\indent
For every $k=1,2,\dots$ we apply 
Corollary~\ref{s-rainwatercor1} to the measure $\mu_k$
given by the restriction of $\mu$ to  $B_k$, and family $\F_k$, 
which  by Step~2 is a weak* compact subset of $P(B_k)$.

There are now two possibilities: either there exists $k$ such 
that statement~(ii) of Corollary~\ref{s-rainwatercor1} holds, 
or statement~(i) of Corollary~\ref{s-rainwatercor1} holds 
for every $k$. 

\smallskip
In the first case there exists a probability measure $\sigma$
on the space $\G_k$ and a Borel set $E$ such that the measure
\[
\int_{\G_k} (1_E  \, \lambda_\gamma) \, d\sigma(\gamma)
\] 
is nontrivial and absolutely continuous \wrt $\mu_k$, and
therefore also \wrt $\mu$. Then, using \eqref{e-7.1} we obtain 
that also the measure 
\[
\mu' 
:= \int_{\G_k} (1_E \, \mu_\gamma) \, d\sigma(\gamma)
\] 
is nontrivial and absolutely continuous \wrt $\mu$, and 
since each measure $1_E \, \mu_\gamma$ is the restriction 
of $\Haus^1$ to a subset of the $C$-curve $G_\gamma$, 
we have that $\mu'$ satisfies all the requirements
in statement~(ii), which therefore holds true.

\smallskip
In the second case
we obtain that for every $k$ the measure $\mu_k$ is
supported on a set $E_k$ contained in $B_k$  
which is null \wrt all measures in $\F_k$, 
and using the first inequality in 
\eqref{e-7.1}) we obtain that
\begin{equation}
\label{e-7.1.1}
\Haus^1(E_k\cap G_\gamma) = 0
\quad\text{for every $\gamma\in\G_k$.}
\end{equation}
Now notice that intersection of every $C$-curve 
$G$ with $B_k$ can be covered by finitely many 
curves $G_\gamma$ with $\gamma\in\G_k$,%
\footnoteb{We use that $C$-curves can be characterized 
as the sets $\gamma(J)$ where $J$ is a compact
interval in $\R$ and $\gamma:J\to\R^n$ satisfies
$\Lip(\gamma)\le 1$ and $e\cdot\dot\gamma(s) \ge \cos\alpha$
for $\Leb^1$-a.e.~$s\in J$.}
and therefore \eqref{e-7.1.1} implies $\Haus^1(E_k\cap G) = 0$.
We have thus proved that $E_k$ is $C$-null. 

We then let $E$ be the union of all $E_k$ 
and observe that $E$ is $C$-null, too, and 
$\mu$ is supported on $E$. Thus (i) holds.
\end{proof}

\begin{proof}
[Proof of Lemma~\ref{s-rainwatercor3}]
We choose a finite family of cones $\{C_i\}$ 
(in the sense of \S\ref{s-cones}) the interiors 
of which cover $\R^n\setminus\{0\}$, 
and then apply Lemma~\ref{s-rainwatercor2}
to $\mu$ and to each $C_i$. 
There are now two possibilities: either there 
exists $i$ such that statement~(ii) of 
Lemma~\ref{s-rainwatercor2} holds, or 
statement~(i) of Lemma~\ref{s-rainwatercor2} 
holds for every~$i$. 

In the first case we immediately obtain that
statement~(ii) holds.
In the second case, for every $i$ there exists
a set $E_i$ which supports $\mu$ and is $C_i$-null.
We then let $E$ be the intersection of all $E_i$
and claim that $E$ satisfies the requirements 
in statement~(i), which therefore holds true.

It is indeed obvious that $E$ supports $\mu$. 
Concerning the unrectifiability of $E$, note that
since the interiors of the cones $C_i$ cover
$\R^n\setminus\{0\}$, we can cover
every curve $G$ of class $C^1$ in $\R^n$ 
by countably many sub-arcs $G_j$, 
each one contained in a $C_i$-curve for some $i$. 
Therefore $\Haus^1(E\cap G_j)=0$ because 
$E$ is $C_i$-null. Hence $\Haus^1(E\cap G)=0$,
and we have proved that $E$ is purely unrectifiable
(cf.~\S\ref{s-rectif}).
\end{proof}

\begin{proof}
[Proof of Proposition~\ref{s-rainwatercor4}]
Let $\tilde\mu$ be the restriction of $\mu$ to the set $F$; 
thus $V(\tilde\mu,x)=V(\mu,x)$ for $\tilde\mu$-a.e.~$x$
by Proposition~\ref{s-basic}(i), and in particular
\begin{equation}
\label{e-7.4}
V(\tilde\mu,x)\cap C=\{0\}
\quad\text{for $\tilde\mu$-a.e.~$x$.}
\end{equation}
We must prove that $\tilde\mu$ is supported on 
a $C$-null set, and for this it suffices to
apply Lemma~\ref{s-rainwatercor2} (to the measure 
$\tilde\mu$ and the cone $C$) and show that, 
of the two alternatives given in that statement,
only alternative~(i) is viable.
Indeed the definition of the decomposability bundle
in \S\ref{s-decomp} and \eqref{e-7.4} imply
that for every family $\{\mu_t: \, t\in I\}$ 
in $\F_{\tilde\mu}$ there holds
$\Tan(E_t,x) \cap C=\{0\}$ for $\mu_t$-a.e.~$x$ 
and a.e.~$t$, and this contradicts alternative~(ii).
\end{proof}

We now prove Lemma~\ref{s-perturb2}.
The proof relies on the following proposition, which 
is a simplified version of a result contained in 
\cite{ACP3} 
(we include a proof for the sake of completeness).

\begin{proposition}
\label{s-perturb1}
Let be given a measure $\mu$ on $\R^n$, 
a cone $C=C(e,\alpha)$ in $\R^n$, 
and a $C$-null compact set $K$ in $\R^n$.
Then for every $\eps>0$ there exists a 
smooth function $f:\R^n\to\R$ such that, 
for every $x\in\R^n$, 
\begin{itemize}
\item[(i)]
$0 \le f(x)\le \eps$;
\item[(ii)]
$0 \le D_e f(x) \le 1$ and $D_e f(x) = 1$ if $x\in K$;
\item[(iii)]
$|\dV f(x)|\le 1/\tan\alpha$ where $V:=e^\perp$.%
\footnoteb{Here $\dV f(x)$ is the derivative of $h$ at $x$ \wrt $V$
(see \S\ref{s-not}), and $|\dV f(x)|$ is its operator norm.}
\end{itemize}
\end{proposition}

\begin{proof}
We first construct a Lipschitz function $g$ that satisfies
statements (i), (ii) and (iii) with $K$ replaced by a 
suitable open set $A$ that contains $K$, and then
we regularize $g$ by convolution to obtain $f$. 

\medskip
\textit{Step~1. There exists an open 
set $A$ such that $K\subset A$ and (cf.~\S\ref{s-cones})}
\begin{equation}
\label{e-7.5}
\Haus^1 \big( A\cap G \big) \le\eps	
\quad\textit{for every $C$-curve $G$.}
\end{equation}
\indent
More precisely, we claim  
that there exists $\delta>0$ such that 
$\Haus^1 \big( K_\delta\cap G \big) \le\eps$ for every
$C$-curve $G$, where $K_\delta$ is the set of all $x$ such that
$\dist(x,K)\le \delta$, and then it suffices to take $A$ equal 
to the interior of $K_\delta$. 

We argue by contradiction: if the claim does not hold, 
then for every $\delta>0$ there exists a $C$-curve $G_\delta$
such that $\Haus^1 \big( K_\delta\cap G_\delta \big)\ge\eps$.
Let now $R$ be the line in $\R^n$ spanned by $e$ and let
$J'$ be a compact interval such that the segment
$J:=\{te: \, t\in J'\} \subset R$ contains the projections 
of $K$ on $R$.
We can then assume that 
the projections of each $G_\delta$ on $R$ agrees with~$J$. 
Moreover, the fact that $G_\delta$ is a $C$-curve means 
that it can be parametrized by a Lipschitz path 
$\gamma_\delta:J'\to\R^n$ of the form
\[
\gamma_\delta(s) = se + \eta_\delta(s) 
\quad\text{with $\eta_\delta(s)\in V:=e^\perp$ for every $s\in J'$,}
\]
where $\eta_\delta:J'\to V$ is Lipschitz and satisfies 
\[
|\dot\eta_\delta(s)| \le \tan\alpha 
\quad\text{for a.e.~$s\in J'$.}
\]
Finally we set
$K'_\delta:= \gamma_\delta^{-1}(K_\delta)
=\gamma_\delta^{-1}(K_\delta\cap G_\delta)$.

Possibly passing to a subsequence, we can assume that, 
as $\delta\to 0$, the maps $\eta_\delta$ converge uniformly 
to some Lipschitz map $\eta_0:J'\to V$, and the 
compact sets $K'_\delta$ converge to some compact set 
$K'_0\subset J'$ in the Hausdorff distance. 

Therefore the paths $\gamma_\delta$ converge 
to $\gamma_0$ given by $\gamma_0(s):=se+\eta_0(s)$, 
the set $G_0:=\gamma_0(J')$ is a $C$-curve, and finally
$K\cap G_0$ contains $K_0:=\gamma_0(K'_0)$.
We prove next that $K_0$ has positive length, 
which contradicts the fact that $K$ is a $C$-curve.
Indeed%
\,\footnoteb{The second inequality follows from 
the upper semicontinuity of the Lebesgue measure
\wrt the Hausdorff convergence of compact sets; 
the third inequality follows from the fact that 
$\Haus^1(\gamma_0(E)) \le \Leb^1(E)/\cos\alpha$
for every set $E\subset J'$, which in turn follows
from the fact that $|\dot\eta_0(s)| \le \tan\alpha$ 
for a.e.~$s$.}
\begin{align*}
  \Haus^1(K_0) 
  \ge \Leb^1(K'_0) 
& \ge \limsup_{\delta\to 0} \Leb^1(K'_\delta) \\
& \ge \limsup_{\delta\to 0} \big( \cos\alpha \, \Haus^1(K_\delta\cap G_\delta) \big)
  \ge \eps \, \cos\alpha > 0
  \, .
\end{align*}

\medskip
\textit{Step~2. Construction of $g$.}
\\ 
\indent
For every $x\in\R^n$ we denote by $\G_x$ the class of
all $C$-curves $G=\gamma([a,b])$ whose end-point 
$x_G:=\gamma(b)$ is of the form $x_G=x+se$ for 
some $s \ge 0$, and we set
\[
g(x):= \sup_{G\in\G_x} \big( \Haus^1(A\cap G) -|x_G-x| \big)
\]
Starting from the definition one can readily checks that 
the following properties hold for every $x\in\R^n$:
\begin{itemize}
\item[(a)]
$0 \le g(x)\le \eps$ (here we use \eqref{e-7.5});
\item[(b)]
$g(x) \le g(x+se) \le g(x)+s$ for every $s>0$,  
and if the segment $[x,x+se]$ is contained in $A$ 
then $g(x+se)= g(x)+s$;
\item[(c)]
$|g(x+v) - g(v)| \le |v|/\tan\alpha$ for every $v\in V:=e^\perp$;
\end{itemize}
Statements (b) and (c) imply that $g$ is Lipschitz and
\begin{itemize}
\item[(b')]
$0 \le D_e g(x) \le 1$ for $\Leb^n$-a.e.~$x$ 
and $D_e g(x) = 1$ for $\Leb^n$-a.e.~$x\in A$;
\item[(c')]
$|\dV g(x)|\le 1/\tan\alpha$ for $\Leb^n$-a.e.~$x$.
\end{itemize}

\medskip
\textit{Step~3. Construction of $f$.}
\\ 
\indent
We take $r$ so that $0<r<\dist(K,\R^n\setminus A)$
and set $f:=g*\rho$ where $\rho$ is a mollifier with 
support contained in the ball $B(r)$.
Then statements (i), (ii) and (iii) follow from 
statements (a), (b') and (c'), respectively.
\end{proof}

\begin{proof}
[Proof of Lemma~\ref{s-perturb2}]
We fix for the time being $\eps'>0$ 
and take a smooth function $f$ that satisfies
statements~(i), (ii) and (iii) in 
Proposition~\ref{s-perturb1} with $\eps'$ 
instead of $\eps$.
Then we set
\[
g := \varphi f
\]
where $\varphi:\R^n\to [0,1]$ is a smooth cut-off
function such that $\varphi(x)=1$ for $x\in B(\bar x,r)$
and $\varphi(x)=0$ for $x\in B(\bar x,r')$.
Thus $g$ satisfies statement~(i) of
Lemma~\ref{s-perturb2}.

Moreover we can choose $\varphi$ so that
$|d\varphi(x)|\le 2/(r'-r)$ for every $x$; then
\[
dg = \varphi \, df + f \, d\varphi
\quad\text{and}\quad
\|f \, d\varphi\|_\infty \le \frac{2\eps'}{r'-r}
\, , 
\]
and therefore $g$ satisfies statements~(ii) and (iii) 
of Lemma~\ref{s-perturb2} if we take $\eps'$ 
so that $2\eps'/(r'-r)$ is smaller than 
$\eps$ and $1/\tan\alpha$.
\end{proof}

We conclude this section with two measurability
results (Lemmas~\ref{s-meas1} and \ref{s-meas2}) 
which are used in Section~\ref{s4},
and an approximation result (Lemma~\ref{s-approx2})
which yields Lemma~\ref{s-approx} as a corollary.
For the rest of this section $f$ is a Lipschitz
function on $\R^n$, and we take $\D(f,x)$ and 
$\D^*(f,x)$ as in \S\ref{s-diffbund}. 

We begin with a definition.

\begin{parag}[Deviation from linearity]
\label{s-errordef}
Given a Lipschitz function $f:\R^n\to\R$,
a point $x\in\R^n$,   
a linear subspace $V$ of $\R^n$,
a linear function $\alpha:V\to\R$, 
and $\delta>0$, 
we set 
\[
m(f,x,V,\alpha,\delta) :=
\sup_{h\in V, \, 0<|h|\le\delta}
\frac { | f(x+h) - f(x) - \alpha h | } {|h|}
\, .
\]
Thus $m(f,x,V,\alpha,\delta)$ measures the deviation 
of $f$ from the linear function $\alpha$ around $x$.
In particular we have that $f$ is differentiable 
at $x$ \wrt $V$ with derivative 
$\dV f(x)=\alpha$ if an only if for every $\eps>0$
there exists $\delta>0$ such that 
$m(f,x,V,\alpha,\delta) \le \eps$, that is,
$m(f,x,V,\alpha,\delta)$ tends to $0$ as 
$\delta\to 0$ (recall that $m$ is increasing in $\delta$).
\end{parag}

\begin{lemma}
\label{s-apdiff}
Let be given $f$, $x$ and $V$ as above, 
$W$ and $W'$ linear subspaces of $V$,
$\alpha$ and $\alpha'$ linear functions on $V$.
Then, setting $m:=m(f,x,W,\alpha,\delta)$, 
$m':=m(f,x,W',\alpha',\delta)$, we have
\[
m \le m' + |\alpha' - \alpha| + (L+|\alpha'|) d 
\, ,
\]
where $L:=\Lip(f)$, $d:=\dgr(W,W')$ is the distance 
between $W$ and $W'$ in $\Gr(\R^n)$, and the norm for linear
functionals is, as usual, the operator norm.
\end{lemma}

\begin{proof}
We fix $h \in W$ with $|h|\le\delta$, 
and denote by $h'$ the orthogonal 
projection of $h$ on $W'$. 
Then, taking into account the definition
of $\dgr$, we have 
\begin{equation}
\label{e-mickey1}
|h'| \le |h| \le \delta
\quad\text{and}\quad
|h-h'| \le d|h|
\, .
\end{equation}
Now, writing $f(x+h) - f(x) - \alpha h$ as
$I + I\!I + I\!I\!I + I\!V$ with 
\begin{align*}
  I 
& := f(x+h) - f(x+h') 
  \, , \\
  I\!I  
& := f(x+h') - f(x) - \alpha' h' 
  \, , \displaybreak[1] \\
  I\!I\!I 
& := \alpha' h' -\alpha' h 
  \, , \\
  I\!V 
& := \alpha' h -\alpha h
  \, , 
\end{align*}
and using the estimates%
\,\footnoteb{The first and third estimates 
follows from the second inequality in \eqref{e-mickey1}
and the fact that $f$ is Lipschitz,
the second one follows from the definition 
of $m'$ and the first inequality in \eqref{e-mickey1}.}
\begin{align*}
  |I| 
& \le L|h-h'| \le Ld |h| 
  \, , \\
  |I\!I|  
& \le m'|h'| \le m'|h| 
  \, , \displaybreak[1] \\
  |I\!I\!I| 
& \le |\alpha'| \, |h'-h| \le d|\alpha'| \, |h| 
  \, , \\
  |I\!V| 
& \le |\alpha' -\alpha| \, |h| 
  \, , 
\end{align*}
we obtain 
\[
|f(x+h) - f(x) - \alpha h| 
\le \big[ m' + |\alpha' - \alpha| + (L+|\alpha'|) d \big] \, |h|
\, , 
\]
which implies the desired estimate.
\end{proof}

\begin{lemma}
\label{s-meas1}
Let $f$ be a Lipschitz function on $\R^n$, 
$E$ a Borel set in $\R^n$, and $x\mapsto V(x)$ 
a Borel map from $E$ to $\Gr(\R^n)$ such that
$V(x)$ belongs to $\D(f,x)$ for every $x$. 
For every $x\in E$ we denote by $\dV f(x)$ the derivative
of $f$ at $x$ \wrt $V(x)$, and we extend it to a linear function
on $\R^n$ by setting $\dV f(x)\, h=0$ for every $h\in V(x)^\perp$. 
Then $x\mapsto \dV f(x)$ is a Borel map  from $E$ to the dual of $\R^n$.
\end{lemma}

\begin{proof}[Sketch of proof]
Possibly subdividing $E$ into finitely many Borel sets, 
we can assume that $V(x)$ has constant dimension $d$
for all $x\in E$.

Since the map $x\mapsto V(x)$, viewed as a closed-valued 
multifunction from $E$ to $\R^n$, is Borel measurable 
(cf.~\ footnote~\ref{f-4.1} in Section~\ref{s4})
we can use Kuratowski and Ryll-Nardzewski's measurable 
selection theorem (see \cite{Sri}, Theorem~5.2.1) 
to find Borel vectorfields $e_1,\dots, e_n$
defined on $E$ so that 
$\{e_1(x),\dots,e_n(x)\}$ is an orthonormal 
basis of $\R^n$ for every $x\in E$, 
and $\{e_1(x),\dots,e_d(x)\}$ a basis of $V(x)$.

Then for every $h>0$ and every $x\in E$ we consider
the linear function $T_h(x):\R^n\to\R$ defined by
\[
\scal{T_h(x)}{e_i(x)} :=
\begin{cases}
  \displaystyle\frac{f(x+he_i(x))-f(x)}{h} 
    & \text{for $i=1,\dots,d$,} \\
  0 & \text{for $i=d+1,\dots,n$.} \\
\end{cases}
\]
Using that each $e_i$ is Borel and that $f$ is continuous 
one easily verifies that $x\mapsto T_h(x)$ is a Borel
map from $E$ to the dual of $\R^n$ for every $h>0$. 
Moreover, since $f$ is differentiable \wrt $V(x)$
at each $x\in E$, $T_h(x)$ converges to $\dV f(x)$ as $h\to 0$.
Thus $x\mapsto \dV f(x)$ is the pointwise limit of
a sequence of Borel maps, and therefore it is Borel. 
\end{proof}

\begin{lemma}
\label{s-meas2}
Let $f$ be a Lipschitz function on $\R^n$. Then $\D(f,x)$ 
and $\D^*(f,x)$ are closed, nonempty subsets of 
$\Gr(\R^n)$ for every $x$. Moreover $x\mapsto\D(f,x)$ and 
$x\mapsto\D^*(f,x)$ are Borel-measurable, 
closed-valued multifunctions from $\R^n$ to $\Gr(\R^n)$.
\end{lemma}

\begin{proof}[Sketch of proof]
We set $L:= \Lip(f)$, denote by $B$ the set of all 
linear functions $\alpha$ on $\R^n$ with $|\alpha|\le L$, 
and by $G$ the graph of the multifunction $x\mapsto\D(f,x)$, 
namely the set of all $(x,V)\in \R^n\times\Gr(\R^n)$ 
such that $V\in \D(f,x)$.
We then define the function $g$ on $\R^n\times\Gr(\R^n)$
given by
\[
g(x,V) := \inf_{\delta>0, \, \alpha\in B} m(f,x,V,\alpha,\delta)
\]
where $m(f,x,V,\alpha,\delta)$ is given in \S\ref{s-errordef}.

For every $\delta>0$ the function $m$ is Borel measurable 
in the variables $x,V,\alpha$, and by Lemma~\ref{s-apdiff} it is 
Lipschitz in the variables $\alpha,V$ with a Lipschitz constant 
independent of $\delta$.
Using this fact one can easily prove that $g$ is Lipschitz in 
the variable $V$ and that the infimum that defines $g$ can be 
replaced by the infimum over a countable dense family
of couples $\delta,\alpha$, which means that $g$ is the infimum
of a countable family of Borel measurable functions, and
therefore it is Borel measurable itself.

Moreover $V$ belongs to $\D(f,x)$ if and only if $g(x,V)=0$
(cf.~\S\ref{s-errordef}), which means that $G=g^{-1}(0)$.
Since $g$ is continuous in $V$ then $\D(f,x)$ is closed for all $x$, 
and since $g$ is Borel measurable then $G$ is a Borel set.
Thus $x\mapsto\D(f,x)$ is a closed-valued multifunction 
with Borel graph, which implies by a standard argument that 
$x\mapsto\D(f,x)$ is Borel measurable.

Finally, the claim concerning $x\mapsto\D^*(f,x)$ 
can be easily obtained from the claim on $x\mapsto\D(f,x)$
(we omit the details).
\end{proof}

\begin{lemma}
\label{s-approx2}
Let $f$ be a Lipschitz function on $\R^n$, 
$\mu$ a measure on $\R^n$,
and $V:\R^n\to\Gr(\R^n)$ Borel map such that 
$V(x)$ belongs to $\D(f,x)$ for $\mu$-a.e.~$x$. 
Then for every $\eps>0$ there exist a compact set $K$ in $\R^n$ 
and a function $g:\R^n\to\R$ of class $C^1$ such that:
\begin{itemizeb}
\item[(i)]
$\mu(\R^n\setminus K)\le \eps$;
\item[(ii)]
$\| g - f \|_\infty \le \eps$;
\item[(iii)]
$\Lip(g) \le \Lip(f) + \eps$;
\item[(iv)]
$| \dV g(x) - \dV f(x) | \le\eps$
for every $x \in K$.
\end{itemizeb}
\end{lemma}

\begin{remark}
In the special case where $V(x)$ does not depend on $x$
we can take $g := f *\rho$ with a suitable 
function $\rho$ with integral equal to $1$.
More precisely, when $n=2$ and $V$ is the line $\R\times\{0\}$, 
we can take as $\rho$ the characteristic function of the rectangle 
$[-r,r]\times [-r^2,r^2]$, renormalized so to have integral 
equal to $1$ and with $r$ sufficiently small.
This is the idea behind Step~5 in the proof below.
\end{remark}

\begin{proof}
We set $L:= \Lip(f)$ and denote  by $E$ be the set of all 
$x\in\R^n$ such that $V(x) \in \D(f,x)$. Then it follows 
from Lemma~\ref{s-meas2} that $E$ is a Borel set.

For every $x$ in $E$ we extend the linear function 
$\dV f(x)$ to a linear function $\alpha(x)$ on $\R^n$
by setting $\alpha(x) \, h := 0$ for every 
$h\in V(x)^\perp$; thus $|\alpha(x)| = |\dV f(x)| \le L$.
Note that the map $x\mapsto\alpha(x)$ is a Borel measurable
map from $E$ to the dual of $\R^n$ by Lemma~\ref{s-meas1}.

The rest of the proof is divided in several steps.

\medskip
\textit{Step~1. 
There exist $\delta >0$ and finitely many pairwise disjoint
compact sets $K_i$ with the following properties:
\begin{itemizeb}
\item[(a)]
$\mu (\R^n\setminus K) \le \eps$ where $K$ is the union of 
all $K_i$ (thus statement~(i) holds);
\end{itemizeb}
and for every $i$,
\begin{itemizeb}
\item[(b)]
$\dgr(V(x),V(x')) \le \eps/L$ for every $x,x'\in K_i$;
\item[(c)]
$|\alpha(x)-\alpha(x')| \le \eps$ for every $x,x'\in K_i$;
\item[(d)]
$m(f,x,V(x),\alpha(x),\delta) \le \eps$ 
for every $x\in K_i$.
\end{itemizeb}}
\indent
For every $x\in E$ the function $f$ is differentiable 
\wrt $V(x)$ with derivative $\alpha(x)$, and therefore
there exists $\delta >0$, depending on $x$, such that 
the estimate in (d) holds
(cf.\ \S\ref{s-errordef}).
Since moreover $\mu(\R^n\setminus E)=0$,
we can find a subset $E'$ of $E$
such that $\mu(\R^n\setminus E') \le \eps/2$ 
and the estimate in (d) 
holds with the same $\delta$ for all $x\in E'$.
This value of $\delta$ is the one we choose. 

Next we partition $E'$ into finite a number $N$ 
of Borel sets $E_i$ 
such that the oscillations of the maps $x\mapsto V(x)$ 
and $x\mapsto\alpha(x)$ on each $E_i$ are less that 
$\eps/L$ and $\eps$, respectively. 
Finally for every $i$ we take a compact set $K_i$ 
contained in $E_i$ such that
$\mu(E_i\setminus K_i) \le \eps/(2N)$.
It is now easy to check that statements (a--d)
hold.

\medskip
\textit{Step~2. For every $i$ we choose $x_i\in K_i$ 
and set $V_i:=V(x_i)$ and  $\alpha_i:=\alpha(x_i)$. 
Then for every $x\in K_i$ there holds
$m(f,x,V_i,\alpha_i,\delta) \le 4\eps$.}
\\ 
\indent
We obtain this estimate by applying
Lemma~\ref{s-apdiff} together with 
the estimates in statements~(b), (c) and (d) 
and the fact that $|\alpha(x)|\le L$.

\medskip
\textit{Step~3. 
Given $h\in\R^n$ and an index $i$, we write 
$h=h'+h''$ with $h'\in V_i$ and $h''\in V_i^\perp$.
Then for every $x\in K_i$ and every $h$ 
with $|h'|\le\delta$ there holds}
\begin{equation}
\label{e-pluto2}
|f(x+h) - f(x) - \alpha_i h'| \le 4\eps |h'| + L |h''|
\, .
\end{equation}
\indent
The estimate in Step~2 yields 
$|f(x+h')- f(x) - \alpha_i h'| \le 4\eps |h'|$, 
and using that 
$|f(x+h)-f(x+h')| \le L|h''|$
we obtain \eqref{e-pluto2}.

\medskip
\textit{Step~4. Let $\rho$ be a positive function 
on $\R^n$ with integral $1$ and support contained 
in the ball $B(r)$.
Then $f*\rho$ is a function of class $C^1$ that satisfies
\begin{itemizeb}
\item[(e)]
$\|f - f*\rho \|_\infty \le Lr$; 
\item[(f)]
$\| d(f*\rho) \|_\infty \le L$.
\end{itemizeb}}
\indent
Statement~(e) is obtained by a simple computation
taking into account that $f$ is Lipschitz and 
that the support of $\rho$ is contained in $B(r)$.

The distributional derivative $d(f*\rho) = df * \rho$, 
being the convolution of an $L^\infty$ and an $L^1$ function,
is bounded  and continuous, which means that $f*\rho$ is
of class $C^1$ and has bounded derivative.
Moreover 
$\| d(f *\rho)\|_\infty \le \| df\|_\infty \|\rho\|_1 = L$, 
and statement~(f) is proved.

\medskip
\textit{Step~5. For every $i$ and every $r>0$
there exists a positive function $\rho_i$ with
integral $1$ and support contained in $B(r)$ such that 
$f_i := f*\rho_i$ satisfies the following property:
for every $x\in K_i$
the restriction of the linear function 
$df_i(x) - \alpha_i$ to the subspace $V_i$ 
has norm at most $M\eps$, 
where the constant $M$ depends only on $n$.}
\\ 
\indent
We assume that $k:=\dim(V_i)>0$, otherwise there is
nothing to prove.
We then take $r'>0$ and denote by $B'$ the ball with 
center $0$ and radius $r'$ contained in $V_i$, and 
by $B''$ the ball with center
$0$ and radius $r'':= \eps r'/L$ contained in $V_i^\perp$.
We then identify $\R^n$ with the product $V_i\times V_i^\perp$
and set
\[
\rho_i := c \, 1_{B'\times B''}
\quad\text{with}\ 
c:=\frac{1}{\Leb^k(B')\, \Leb^{n-k}(B'')}
\, .
\]

We claim that if $r'\le\delta/2$ then $f_i:=f*\rho_i$ satisfies
\begin{equation}
\label{e-pluto4}
\big| f_i(x+h) - f_i(x) - \alpha_i h \big| 
\le M\eps|h|
\end{equation}
for every $x \in K_i$, every $h\in V_i$ with $|h|\le r'$,
and a suitable $M$.
This inequality shows that $f_i$ has the property required in 
Step~5. 

We fix $x$ and $h$ as above.
A simple computation yields
\begin{equation}
\label{e-pluto5}
f_i (x+h) - f_i (x) -\alpha_i h  
= \int_{\R^n} e(z) \, (\rho_i(h-z) - \rho_i(-z)) \, d\Leb^n(z)
\end{equation}
where $e(z) := f(x+z) - f(x) - \alpha_iz$ for every $z\in\R^n$.
We observe now that estimate \eqref{e-pluto2} yields
\begin{align}
\label{e-pluto5.0}
|e(z)| \le 4\eps|z'| + L|z''|
\end{align}
for every $z\in\R^n$ such that $|z'|\le\delta$, 
where $z'$ and $z''$ come from the decomposition
$z=z'+z''$ with $z'\in V_i$ and $z''\in V_i^\perp$
(cf.\ Step~3).
Therefore, in order to use \eqref{e-pluto5.0} to 
estimate the integral in \eqref{e-pluto5}, 
we must check that
$|z'|\le\delta$ for every $z$ such that 
$\rho_i(h-z) - \rho_i(-z) \ne 0$.
Indeed, taking into account the definition of $\rho_i$ 
and the fact that $h$ belongs to $V_i$, we obtain
\begin{equation}
\label{e-pluto5.1}
\rho_i(h-z) - \rho_i(-z) =
\begin{cases}
  \pm c & \text{if $z'\in (B'+h)\triangle B'$ and $z''\in B''$,} \\
  0 & \text{otherwise,}
\end{cases}
\end{equation}
and therefore if $\rho_i(h-z) - \rho_i(-z) \ne 0$ then 
$z'$ belongs to the symmetric difference $(B'+h)\triangle B'$;
in particular $|z'| \le r'+|h| \le 2r'\le \delta$, as required. 

Then, denoting by $c_h$ the volume of the unit ball 
in $\R^h$ for $h=0,1,\dots$, we obtain%
\,\footnoteb{The first inequality follows from \eqref{e-pluto5}
and \eqref{e-pluto5.0}, for the second one we use that 
$|z'|\le 2r'$ and $|z''|\le r''$ whenever
$\rho_i(h-z) - \rho_i(-z)\ne 0$ (cf.~\eqref{e-pluto5.1}); 
for the third one we use that $r''=\eps r'/L$, 
formula \eqref{e-pluto5.1}, and the definition of $c$;
for the fourth one we use that the volume of $B'$ is 
$c_k(r')^k$ and the volume of 
$(B'+h)\triangle B'$ is at most $2c_{k-1}(r')^{k-1}|h|$.}
\begin{align*}
\big| f_i (x + h) & - f_i(x) -\alpha_i h  \big| 
      \\
& \le \int_{\R^n} \big[ 4\eps|z'| + L|z''| \big] 
      \, \big| \rho_i(h-z) - \rho_i(-z) \big| \, d\Leb^n(z)
      \\
& \le \big[ 8 \eps r'+ L r'' \big]  
      \int_{\R^n} |\rho_i(h-z) - \rho_i(-z)| \, d\Leb^n(z) 
      \\
& \le 9\eps r \, 
       \frac{\Leb^k\big( (B'+h)\triangle B' \big)}{\Leb^k(B')} \, 
  \le \frac{18 \, c_{k-1}}{c_k} \eps|h|
      \, .
\end{align*}
We have thus proved \eqref{e-pluto4} with $M$ equal to the 
maximum of $18 \, c_{k-1} / c_k$ over all $k=1,\dots,n$.

\medskip
\textit{Step~6. Take $M$ and $f_i$ as in Step~5. 
Then for every $x\in K_i$ there holds}
\begin{equation}
\label{e-pluto8}
| \dV f_i(x) - \dV f(x) | \le (M+3)\eps
\, .
\end{equation}
\indent
Taking into account \S\ref{s-errordef}
and the fact that $\dV f(x)$ agrees 
with $\alpha(x)$ on $V(x)$
we rewrite \eqref{e-pluto8} as
\begin{equation}
\label{e-pluto8.1}
m(df_i(x),0,V(x),\alpha(x),1) \le (M+3)\eps
\, ,
\end{equation}
and the estimate in Step~5 as
\begin{equation}
\label{e-pluto8.2}
m(df_i(x),0,V_i,\alpha_i,1) \le M\eps
\, . 
\end{equation}
We then obtain \eqref{e-pluto8.1}
from \eqref{e-pluto8.2} by applying 
Lemma~\ref{s-apdiff} together with the 
following estimates:
$\dgr(V(x),V_i) \le \eps/L$ (statement~(b)),
$|\alpha(x) - \alpha_i| \le \eps$ (statement~(c)), 
and $|\alpha_i|\le L$.

\medskip
\textit{Step~7. Construction of the function $g$.} 
\\ 
\indent
Since the compact sets $K_i$ are pairwise disjoint, 
there exists a smooth partition of unity
$\{\sigma_i\}$ of $\R^n$ such that the functions
$\sigma_i$ take the value $1$ 
on some neighbourhood of $K_i$, and are constant
outside some compact set.%
\footnoteb{Take as $\{\sigma_i\}$ a smooth 
partition of unity of $\R^n$ subject to the open 
cover $\{A_i\}$ constructed as follows:
$U$ is a bounded open set that contains 
all $K_i$, $C_1$ be the union of $K_1$ and the 
complement of $U$, and $C_i:=K_i$ for every $i> 1$, 
and finally $A_i$ is the complement of the union 
of all $K_j'$ with $j\ne i$.}
Thus the derivatives $d\sigma_i$ have
compact support and therefore are bounded, 
and
\[
m := \max\Big\{1 ; \, \sum_i \|d\sigma_i \|_\infty \Big\}
\]
is  a finite number.
Now we take $f_i=f*\rho_i$ as in Step~5, 
where $\rho_i$ supported in the ball $B(r)$ 
with $r:= \eps/ (mL)$, and set
\[
g:= \sum_i \sigma_if_i 
\, .
\]
The function $g$ is clearly of class $C^1$.
We prove next that $g$ 
satisfies statements~(ii), (iii) and (iv).

Note that statement~(e) and the choice of $r$ and $m$ yield
\begin{equation}
\label{e-pluto10}
|f_i(x)-f(x)| 
\le Lr 
= \frac{\eps}{m}
\le \eps
\quad\text{for every $x\in\R^d$,}
\end{equation}
and since $g(x)$ is a convex combination of the numbers 
$f_i(x)$, it must satisfies $|g(x)-f(x)| \le \eps$ as well, 
which proves statement~(ii).

Given $x\in K$, take $i$ such that $x\in K_i$, 
and note that $g=f_i$ on the neighbourhood of
$K_i$ where $\sigma_i=1$; hence \eqref{e-pluto8} 
becomes $| \dV g(x) - \dV f(x)| \le (M+3)\eps$, 
which is the inequality in statement~(iv)
with $(M+3)\eps$ instead of $\eps$\dots

It remains to prove statement~(iii), namely that
that $|dg(x)|\le L+\eps$ for every~$x$. 
This estimate is an immediate consequence 
of the identity%
\,\footnoteb{Here we use that $\sum_i d\sigma_i(x)=0$, 
cf.\ footnote~\ref{f-7.2}.}
\[
dg(x) 
=\sum_i \sigma_i(x) \, df_i(x) 
      + \sum_i (f_i(x) - f(x)) \, d\sigma_i(x) 
\, , 
\]
and the inequalities
$|df_i(x)|\le L$ (statement~(f)),
$|f_i(x) - f(x)|\le \eps/m$ (by \eqref{e-pluto10}), 
and $\sum_i |d\sigma_i(x)| \le  m$
(by the choice of $m$).
\end{proof}

\begin{proof}[Proof of Lemma~\ref{s-approx}]
We first construct a sequence of approximating functions 
$f_n$ of class $C^1$ that satisfy requirements~(i), 
(ii) and (iii) using Lemma~\ref{s-approx2} above, 
and then regularize these functions by convolution
to make them smooth.
\end{proof}

%
%

\bibliographystyle{plain}

%
%

\vskip .5 cm

{\parindent = 0 pt\begin{footnotesize}

G.A.
\\
Dipartimento di Matematica, 
Universit\`a di Pisa
\\
largo Pontecorvo~5, 
56127 Pisa, 
Italy 
\\
e-mail: {\tt galberti1@dm.unipi.it}

\medskip
A.M.
\\
Max-Planck-Institut f\"ur Mathematik
in den Naturwissenschaften
\\
Inselstrasse~22,
04103 Leipzig,
Germany 
\\
e-mail: {\tt andrea.marchese@mis.mpg.de}

\end{footnotesize}
}

\end{document}